\documentclass[a4paper,11pt]{amsart}

\usepackage{amssymb, manfnt}
\usepackage{hyperref}
\usepackage[normalem]{ulem}
\usepackage{tikz}
\usepackage{booktabs}
\usepackage{multirow}
\usepackage{MnSymbol}
\usepackage{graphicx}
\usepackage{longtable}
\usepackage{amsmath,amscd}
\usepackage{tkz-graph}
\tikzstyle{vertex}=[circle, draw, inner sep=0pt, minimum size=6pt]
\newcommand{\vertex}{\node[vertex]}
\usepackage[all]{xy}
\graphicspath{ {graphs/} }
\usetikzlibrary{decorations.text}

\def\altdb{\vadjust{\vbox to 0pt{\vss\hbox{\kern \hsize
\quad{\dbend}}\kern\baselineskip\kern-10pt}}}

\tikzstyle{mid>}=[decoration={markings, mark=at position 0.5 with {\arrow{>}}}, postaction={decorate}]
\tikzstyle{mid<}=[decoration={markings, mark=at position 0.5 with {\arrow{<}}}, postaction={decorate}]
\tikzstyle{upper>}=[decoration={markings, mark=at position 0.8 with {\arrow{>}}}, postaction={decorate}]
\tikzstyle{upper<}=[decoration={markings, mark=at position 0.8 with {\arrow{<}}}, postaction={decorate}]
\tikzstyle{lower>}=[decoration={markings, mark=at position 0.25 with {\arrow{>}}}, postaction={decorate}]
\tikzstyle{lower<}=[decoration={markings, mark=at position 0.25 with {\arrow{<}}}, postaction={decorate}]
\tikzstyle{tag}=[decoration={markings, mark=at position 0.5 with {\arrow{Rays[n=1]}}}, postaction={decorate}]

\newcommand{\arxiv}[1]{\href{http://arxiv.org/abs/#1}{\tt arXiv:\nolinkurl{#1}}}
\newcommand{\doi}[1]{\href{http://dx.doi.org/#1}{{\tt DOI:#1}}}

\newcommand{\googlebooks}[1]{(preview at \href{http://books.google.com/books?id=#1}{google books})}

\makeatletter
\let\@@pmod\pmod
\DeclareRobustCommand{\pmod}{\@ifstar\@pmods\@@pmod}
\def\@pmods#1{\mkern4mu({\operator@font mod}\mkern 6mu#1)}
\makeatother

\newcommand\FPdim{\operatorname{FPdim}}
\newcommand\Out{\operatorname{Out}}

\newcommand\BrPic{\operatorname{BrPic}}

\newcommand\Aut{\operatorname{Aut}}

\newcommand\Rep{\operatorname{Rep}}

\newcommand\Hom{\operatorname{Hom}}

\newcommand\ad{\operatorname{Ad}}
\newcommand\BrAut{\operatorname{Aut^\text{br}}}

\theoremstyle{plain}
\newtheorem{theorem}{Theorem}[section]
\newtheorem*{theorem*}{Theorem}
\newtheorem*{prop*}{Proposition}
\newtheorem{cor}[theorem]{Corollary}
\newtheorem{lemma}[theorem]{Lemma}
\newtheorem{prop}[theorem]{Proposition}

\theoremstyle{remark}
\newtheorem{rmk}[theorem]{Remark}

\theoremstyle{definition}
\newtheorem{dfn}[theorem]{Definition}
\newtheorem{cons}[theorem]{Construction}

\usepackage{color}

\numberwithin{equation}{section}

\DeclareRobustCommand*{\nicefrac}{\@UnitsNiceFrac}%
\makeatother

\newcommand{\JW}[1]{f^{(#1)}}
\newcommand{\zDiagram}[1][1] { \begin{tikzpicture}[scale = #1,baseline={([yshift=-.5ex]current bounding box.center)}]
\draw [thick] (-.5,-.5) -- (-.5,.5);
\draw [thick] (-.5,.5) -- (.5,.5);
\draw [thick] (.5,.5) -- (.5,-.5);
\draw [thick] (.5,-.5) -- (-.5,-.5);
\node () at (0,0) {$Z$};
\draw [thick] (-.4,-.5) arc [radius = .3, start angle = 0, end angle = -45];
\draw [thick] (.4,-.5) arc [radius = .3, start angle = 180, end angle = 225];
\draw [thick] (-.4,.5) arc [radius = -.3, start angle = 180, end angle = 225];
\draw [thick] (.4,.5) arc [radius = -.3, start angle = 0, end angle = -45]; 
 \end{tikzpicture}
}

\newcommand{\rhoZDiagram}[1][1] { \begin{tikzpicture}[scale = #1,baseline={([yshift=-.5ex]current bounding box.center)}]
\draw [thick] (-.5,-.5) -- (-.5,.5);
\draw [thick] (-.5,.5) -- (.5,.5);
\draw [thick] (.5,.5) -- (.5,-.5);
\draw [thick] (.5,-.5) -- (-.5,-.5);
\node () at (0,0) {$\rotatebox{90}{Z}$};
\draw [thick] (-.4,-.5) arc [radius = .3, start angle = 0, end angle = -45];
\draw [thick] (.4,-.5) arc [radius = .3, start angle = 180, end angle = 225];
\draw [thick] (-.4,.5) arc [radius = -.3, start angle = 180, end angle = 225];
\draw [thick] (.4,.5) arc [radius = -.3, start angle = 0, end angle = -45]; 
 \end{tikzpicture}
}

\newcommand{\Id}[1][1]{\begin{tikzpicture}[scale = #1,baseline={([yshift=-.5ex]current bounding box.center)}]

\draw [thick](0,0) arc [radius  = -.707,start angle = 135, end angle = 225];
\draw [thick](1,0) arc [radius  = .707,start angle = 225, end angle = 135];

\end{tikzpicture}}

\newcommand{\CupCap}[1][1]{\begin{tikzpicture}[scale = #1,baseline={([yshift=-.5ex]current bounding box.center)}]
\draw [thick](0,0) arc [radius  = .707,start angle = 135, end angle = 45];
\draw [thick](0,1) arc [radius  = .707,start angle = -135, end angle = -45];

\end{tikzpicture}}

\newcommand{\LOOP}[1][1]{
\begin{tikzpicture}[scale = #1,baseline={([yshift=-.5ex]current bounding box.center)}]
\draw [thick] (0,0) circle (.5);
\end{tikzpicture}}

\newcommand{\Braid}[1][1]{\begin{tikzpicture}[scale = #1,baseline={([yshift=-.5ex]current bounding box.center)}]

\draw [thick](0,0) -- (1,1);
\draw [thick,draw=white,double=black,double distance=\pgflinewidth,ultra thick] (0,1) -- (1,0);

\end{tikzpicture}}

\newcommand{\xyJonesWenzlIdempotentSixPlusOne}[1]{\xybox{(0,-9.)*\xybox{*\xybox{
(1.5,0);(19.5,18) **i@{-},(3,18);(3,14.4) **@{-}, (3,3.6);(3,0) **@{-},(6,18);(6,14.4) **@{-}, (6,3.6);(6,0) **@{-},(9,18);(9,14.4) **@{-}, (9,3.6);(9,0) **@{-},(15,18);(15,14.4) **@{-}, (15,3.6);(15,0) **@{-},(18,18);(18,14.4) **@{-}, (18,3.6);(18,0) **@{-},(12,16.2) *h+++{\cdots},(12,1.8) *+++{\cdots},(2.25,3.6);(18.75,14.4) **\frm{-}, (10.5,9.) *+++{#1}}!R!(-1.5,0)*\xybox{
(1.5,0);(4.5,18) **i@{-},(3,18);(3,0)  **\crv{(3,5.4)&(3,12.6)},
}!R!(-1.5,0)}}}\newcommand{\xyJonesWenzlIdempotentSeven}[1]{\xybox{(0,-9.)*\xybox{
(1.5,0);(22.5,18) **i@{-},(3,18);(3,14.4) **@{-}, (3,3.6);(3,0) **@{-},(6,18);(6,14.4) **@{-}, (6,3.6);(6,0) **@{-},(9,18);(9,14.4) **@{-}, (9,3.6);(9,0) **@{-},(15,18);(15,14.4) **@{-}, (15,3.6);(15,0) **@{-},(18,18);(18,14.4) **@{-}, (18,3.6);(18,0) **@{-},(21,18);(21,14.4) **@{-}, (21,3.6);(21,0) **@{-},(12,16.2) *h+++{\cdots},(12,1.8) *+++{\cdots},(2.25,3.6);(21.75,14.4) **\frm{-}, (12.,9.) *+++{#1}}}}\newcommand{\xyWenzlRecurrenceLastTermSixPlusOne}[1]{\xybox{(0,-20.)*\xybox{*\xybox{*\xybox{
(1.5,0);(19.5,18) **i@{-},(3,18);(3,14.4) **@{-}, (3,3.6);(3,0) **@{-},(6,18);(6,14.4) **@{-}, (6,3.6);(6,0) **@{-},(9,18);(9,14.4) **@{-}, (9,3.6);(9,0) **@{-},(15,18);(15,14.4) **@{-}, (15,3.6);(15,0) **@{-},(18,18);(18,14.4) **@{-}, (18,3.6);(18,0) **@{-},(12,16.2) *h+++{\cdots},(12,1.8) *+++{\cdots},(2.25,3.6);(18.75,14.4) **\frm{-}, (10.5,9.) *+++{#1}}!R!(-1.5,0)*\xybox{
(1.5,0);(4.5,18) **i@{-},(3,18);(3,0)  **\crv{(3,5.4)&(3,12.6)},
}!R!(-1.5,0)}!U*\xybox{*\xybox{
(1.5,0);(10.5,4) **i@{-},(3,4);(3,0)  **\crv{(3,1.2)&(3,2.8)},
(6,4);(6,0)  **\crv{(6,1.2)&(6,2.8)},
(9,4);(9,0)  **\crv{(9,1.2)&(9,2.8)},
}!R!(-1.5,0)*\xybox{*i\xybox{
(1.5,0);(4.5,4) **i@{-},(3,4);(3,0)  **\crv{(3,1.2)&(3,2.8)},
}}!R!(-1.5,0)*\xybox{
(1.5,0);(10.5,4) **i@{-},(3,4);(3,0)  **\crv{(3,1.2)&(3,2.8)},
(6,4);(9,4)  **\crv{(6,2.2)&(9,2.2)},
(9,0);(6,0)  **\crv{(9,1.8)&(6,1.8)},
}!R!(-1.5,0)}!U*\xybox{*\xybox{
(1.5,0);(19.5,18) **i@{-},(3,18);(3,14.4) **@{-}, (3,3.6);(3,0) **@{-},(6,18);(6,14.4) **@{-}, (6,3.6);(6,0) **@{-},(9,18);(9,14.4) **@{-}, (9,3.6);(9,0) **@{-},(15,18);(15,14.4) **@{-}, (15,3.6);(15,0) **@{-},(18,18);(18,14.4) **@{-}, (18,3.6);(18,0) **@{-},(12,16.2) *h+++{\cdots},(12,1.8) *+++{\cdots},(2.25,3.6);(18.75,14.4) **\frm{-}, (10.5,9.) *+++{#1}}!R!(-1.5,0)*\xybox{
(1.5,0);(4.5,18) **i@{-},(3,18);(3,0)  **\crv{(3,5.4)&(3,12.6)},
}!R!(-1.5,0)}!U}}}

\title{The Brauer-Picard Groups of fusion categories coming from the $ADE$ subfactors}

\author{Cain Edie-Michell}
\address{Cain Edie-Michell\\
Mathematical Sciences Institute\\
Canberra ACT 2601\
Australia}
\email{cain.edie-michell@anu.edu.au}

\allowdisplaybreaks
\begin{document}

\maketitle

\begin{abstract}
We compute the group of Morita auto-equivalences of the even parts of the $ADE$ subfactors, and Galois conjugates. To achieve this we study the braided auto-equivalences of the Drinfeld centres of these categories. We give planar algebra presentations for each of these Drinfeld centres, which we leverage to obtain information about the braided auto-equivalences of the corresponding categories. We also perform the same calculations for the fusion categories constructed from the full $ADE$ subfactors.

Of particular interest, the even part of the $D_{10}$ subfactor is shown to have Brauer-Picard group $S_3 \times S_3$. We develop combinatorial arguments to compute the underlying algebra objects of these invertible bimodules.
\end{abstract}

\section{Introduction}\label{sec:intro}
Let $C$ be a pivotal fusion category generated by a symmetrically self-dual object $X$ of Frobenius-Perron dimension strictly less than $2$. In the unitary setting, a result announced by Ocneanu, and later proved by several authors \cite{MR1193933, MR1145672, MR1313457, MR1929335, MR1308617, MR1617550,MR1976459} shows that $C$ must be one of the $ADET$ fusion categories. These are the categories whose fusion graph for tensoring by $X$ is one of the $ADET$ Dynkin diagrams $A_N$, $D_{2N}$, $E_6$, $E_8$, or $T_n$. Essential to this result was the classical classification of undirected graphs with graph norm less than $2$ (as the Frobenius-Perron dimension of $X$ is defined to be the graph norm of its fusion graph). This reduces the possible fusion rings for $C$ to a subset which makes the classification feasible. 

With the above classification result in mind it is natural to attempt to drop the condition that the generating object $X$ be self-dual, and obtain a full classification of fusion categories generated by an object of Frobenius-Perron dimension less than $2$. If we assume that the category is braided, then such a classification was achieved in \cite{MR1239440}. To attempt the complete classification one might emulate the techniques used in the self-dual case. However one straightaway one arrives at the stumbling block that there is no classification of directed graphs of norm less than $2$. Thus without a braiding assumption a classification result still appears out of hand with the current tools available. However the following unpublished theorem of Morrison and Snyder may allow us to give a partial result.
\begin{theorem}\label{thm:cyclic_ext}
Let $C$ be a unitary pivotal fusion category generated by an object $X$ of dimension less than $2$. If $X\otimes X^* \cong X^*\otimes X$ then $C$ is a unitary cyclic extension of the even part of one of the $ADE$ subfactors.
\end{theorem}
An analogue of this Theorem almost certainly holds for non-unitary categories. The only changes being that now we have to consider non-unitary cyclic extensions of the even parts of the $ADE$ subfactors, along with Galois conjugates.

Given a fusion category $C$ the results of \cite{MR2677836} allow us to classify $G$-graded extensions of $C$. This main ingredient of such a classification is $\BrPic(C)$, the group of Morita auto-equivalences of $C$. As well as being useful in classification problems this group also appears in the study of subfactors. If $C$ is unitary then $\BrPic(C)$ classifies all subfactors whose even and dual even parts are both $C$. The process of computing Brauer-Picard groups of fusion categories is currently receiving attention in the literature by both researchers interested in subfactors \cite{MR2909758,MR3449240}, and fusion categories \cite{MR3210925,1603.04318,MR3373393}.

In light of Theorem~\ref{thm:cyclic_ext}, it becomes a natural question to ask what the Brauer-Picard groups of the even parts of the $ADE$ subfactors (and Galois conjugates) are. As the $ADE$ subfactors are unshaded, they themselves can be thought of as fusion categories. For this paper we will refer to these categories as the $ADE$ fusion categories. The $ADE$ fusion categories contain as a subcategory the even part of the associated subfactor. In fact, the even part is the adjoint subcategory. The main goal of this paper will be to prove the following theorem.
\begin{theorem}\label{thm:BrPic}
Let $C$ be an $ADE$ fusion category (or Galois conjugate). Then the Brauer-Picard group of $C$ (or its adjoint subcategory) is:
\begin{table}[h!]

\centering 
    
    \begin{tabular}{c | l | cc l}
    	\toprule
			\multicolumn{2}{l |}{$C $}      &       $\BrPic(C)$     & $\BrPic(\ad(C))$ \\
	\midrule
	                  $A_N$   & $N=3$ 									    		 &     $\{e\}$ & $\mathbb{Z} /2 \mathbb{Z}$ \\
				& $N=7$											&    $(\mathbb{Z} / 2 \mathbb{Z})^2$ & $D_{2\cdot 4}$ \\
	                                & $N \equiv 0 \pmod 2 $ and $C$ admits a modular braiding 		 &   $\{e\}$   & $\{e\}$ \\
	                                & $N  \equiv 0\pmod 2$ and $C$ doesn't admit a modular braiding & $\mathbb{Z} / 2 \mathbb{Z}$  & $\{e\}$ \\
	                                & $N \equiv 1 \pmod 4 $									 & $\mathbb{Z} / 2 \mathbb{Z}$ &$\mathbb{Z} / 2 \mathbb{Z}$\\
	                                & $N  \equiv 3 \pmod 4$ and $N\neq \{3,7\}$						& $(\mathbb{Z} / 2 \mathbb{Z})^2$ &$(\mathbb{Z} / 2 \mathbb{Z})^2$\\
	                                \midrule
	                 $D_{2N}$ & $N = 5$												& $S_3 $&$(S_3)^2$ \\
	                         	      & $N\neq 5$											& $ \mathbb{Z} / 2 \mathbb{Z}  $&$ (\mathbb{Z} / 2 \mathbb{Z})^2$\\
		      \midrule
	                 $E_6$    &													& $\mathbb{Z} / 2 \mathbb{Z} $ &$\mathbb{Z} / 2 \mathbb{Z}$   \\
	                 \midrule
	                 $E_8$    &													& $\{e\} $ & $\mathbb{Z} / 2 \mathbb{Z}$. \\
				                 
    	\bottomrule
    \end{tabular}
\end{table}

\end{theorem}
We leave the task of working through the rest of the extension theory, and completing the classification of pivotal fusion categories generated by a normal object $X$ of dimension less than $2$ to a later paper.

For the examples we are interested in it turns out to be quite difficult to compute the Morita auto-equivalences explicitly, such as was done in \cite{MR2909758}. Thankfully the isomorphism of groups from \cite{MR2677836}:
\begin{equation*}
\BrPic(C) \cong \BrAut(Z(C)),
\end{equation*}
allows an alternate method to compute the Brauer-Picard group of a fusion category. Here $\BrAut(Z(C))$ is the group of braided auto-equivalences of the Drinfeld centre of $C$. We find the latter group much easier to compute, and hence spend this paper computing Drinfeld centres and braided auto-equivalence groups, rather than a direct computations of invertible bimodules. Although in a few difficult cases we use both descriptions, utilising the relative strengths of both.

Our main tool to compute braided auto-equivalences of categories is via the braided automorphisms of their associated braided planar algebras. It has long been known that there is a one-to-one correspondence between pivotal categories with a distinguished generating object, and planar algebras. This correspondence turns out to be functorial. In particular we have the following:
\begin{prop}\label{prop:PA_to_FC}
Let $P$ a be (braided) planar algebra, and $C_P$ be the associated pivotal (braided) tensor category with distinguished generating object $X$. Let $\operatorname{Aut}^\text{piv}(C_P;X)$ be the group of pivotal (braided) auto-equivalences of $C_P$ that fix $X$ on the nose, and $\Aut(P)$ be the group of (braided) planar algebra automorphisms of $P$. Then there is an isomorphism of groups 
\begin{equation*}
 \Aut(P) \xrightarrow{\sim} \operatorname{Aut}^\text{piv}(C_P;X).
\end{equation*}
\end{prop} 
This proposition is a direct consequence of \cite[Theorem 2.4]{1607.06041}. While the fact that planar algebra automorphisms only give us pivotal auto-equivalences of the corresponding category may seem restrictive, it turns out that all of the auto-equivalences we care about in this paper are pivotal (although not a priori). In particular the gauge auto-equivalences (auto-equivalences whose underlying functor is the identity) of a pivotal category are always pivotal functors when the category has trivial universal grading group, or when every object of the category is self-dual. These two cases include every category we consider in this paper. While it seems somewhat reasonable to expect gauge auto-equivalences are always pivotal, regardless of extra assumptions on the category, the author was unable to produce a general proof.

Note that every functor that fixes $X$ up to isomorphism is naturally isomorphic to a functor which fixes $X$ on the nose. Thus we can think of $\operatorname{Aut}^\text{piv}(C_P;X)$ as the group of pivotal auto-equivalences that fix $X$ up to isomorphism. Also note that the definition of $\operatorname{Aut}^\text{piv}(C_P;X)$ from \cite{1607.06041} is not up to natural isomorphism of auto-equiavalences, but in fact strict equality! If we consider $\operatorname{Aut}^\text{piv}(C_P;X)$ up to natural isomorphism (which we will for the remainder of this paper unless explicitly stated), then the above map becomes a surjection. We observe that the kernel of this surjection motivates the notion of a natural isomorphism between planar algebra homomorphisms. We study these planar algebra natural isomorphisms in an appendix to this paper. It is shown throughout this paper that non-trivial planar algebra natural isomorphisms exist, and thus the category of planar algebras is actually an interesting 2-category! 

The structure of the paper is as follows:

In section~\ref{sec:centres} we compute the Drinfeld centres of the $ADE$ fusion categories and adjoint subcategories. For the full categories these have already been computed in \cite{MR1815993,AN-Survey}. Using a theorem from \cite{MR2587410} we are able to compute the Drinfeld centres of the corresponding adjoint subcategories from the centres of the full categories. Furthermore we give planar algebra presentations for each of the centres. To our knowledge the planar algebras corresponding to $Z(\ad(A_{2N+1}))$ and $Z(\ad(E_6))$ are new to the literature. In particular the planar algebra for $Z(\ad(A_{3}))$ gives a graphical description of the toric code. For several cases the Drinfeld centre is not generated by a single simple object, but is a product of two modular subcategories. For these cases we describe the planar algebras of the modular subcategories.

In section~\ref{sec:autos} we apply the planar algebras from the previous section to help compute the braided auto-equivalence group of the associated category. Our approach is as follows:
\begin{itemize}
\item Compute the group of braided automorphisms of the associated planar algebra, and hence the group of pivotal braided auto-equivalences of the Drinfeld centre that fix the generating object. In particular this tells us the gauge auto-equivalences of the category. For all our examples there turn out to be no non-trivial gauge auto-equivalences.

\item Compute the fusion ring automorphisms of the category. As we only care about braided auto-equivalences we restrict our attention to fusion ring automorphisms that preserve the $T$ matrix of the category. As the category has no non-trivial gauge auto-equivalences, this gives us an upper bound on the braided auto-equivalence group of the category.

\item Explicitly construct braided auto-equivalences of the category to achieve the upper bound.
\end{itemize}

In section~\ref{sec:BrPic} we develop techniques to deal with the Drinfeld centres that are products of two smaller modular categories. We also combine all the previous results to give a proof of our main theorem.

In section~\ref{sec:D10} we explicitly describe the invertible bimodules over $\ad(D_{10})$ and $\ad(A_7)$. For the other examples the invertible bimodules are already known, and the tensor structure of these bimodules is obvious. For the $\ad(D_{10})$ fusion category however it is not obvious what the $36$ invertible bimodules are, and for the $\ad(A_7)$ categories the tensor structure of the bimodules is not obvious. We combine combinatorial techniques developed in \cite{MR2909758} along with knowledge of the braided auto-equivalences of $Z(C)$ to explicitly describe the invertible bimodules. Of particular note these computations reveal the structure of two subfactors of index $\approxeq 6.411$.

Some of the computations for computing the Brauer-Picard group of the modular $A_N$ categories were performed in \cite{1605.08398}. These computations use similar techniques as in this paper, bounding the size of the Brauer-Picard group by considering modular data of the centre.

Our results appear to be related to the work of Xu in \cite{MR2670925}. This is certainly no coincidence as given a subfactor there is a relationship between the intermediate subfactors and the Brauer-Picard groupoid of the even part. See \cite{MR2909758} for an example of this relationship. However this relationship is not yet completely understood, and hence our work adds to the existing literature.

At an initial glance it also appears that our results overlap with \cite{1703.06543}. However the tensor categories we work with are different to the tensor categories the cited paper deals with. Evidence for this can be seen in the fact that their Brauer-Picard groups are infinite algebraic groups, while ours our finite. Also their results are independent of choice of root of unity, where we see varying behaviour.

\section*{Acknowledgements}
We would like to thank Corey Jones and Scott Morrison for many of helpful conversations. We thank Pinhas Grossman for his discussions on graded module categories. We thank Dave Penneys for helpful conversations about planar algebra natural transformations. We thank Noah Snyder for pointing out to us that the subfactor appearing in Section~\ref{sec:D10} is the GHJ subfactor corresponding to the odd half of $E_7$. This research is supported by an Australian Government Research Training Program (RTP) Scholarship. The author was partially supported by the Discovery Project "Subfactors and symmetries" DP140100732 and "Low dimensional categories" DP160103479. Parts of this paper were written during a visit to the Isaac Newton Institute, who we thank for their hospitality.

\section{Preliminaries} \label{sec:prelim}

\subsection{Fusion Categories and Module Categories}
This subsection briefly recalls the essentials we need on fusion categories and modules categories. For more details on fusion categories see \cite{MR2183279}, and for module categories see \cite{MR1976459}.

\begin{dfn}\cite{MR2183279}
Let $\mathbb{K}$ be an algebraically closed field. A fusion category over $\mathbb{K}$ is
a rigid semisimple $\mathbb{K}$-linear monoidal category $C$ with finitely many isomorphism classes of simple objects and
finite dimensional spaces of morphisms, such that the unit object $\mathbf{1}$ of $C$ is simple.
\end{dfn}

For the purpose of this paper our algebraically closed field will be $\mathbb{C}$. To avoid repetition whenever we refer to a fusion category from now on we implicitly mean a fusion category over $\mathbb{C}$.

The Grothendieck ring of a fusion category $C$ is the ring whose elements are isomorphism classes of objects in $C$. The addition operation is given by direct sum, and the product by tensor product. We will denote the Grothendieck ring of $C$ by $K(C)$. To every object $X$ of a fusion category $C$ we define the Frobenius-Perron dimension of $X$, $\FPdim(X)$, to be the largest real eigenvalue of the matrix for multiplication by $[X]$ in $K(C)$. The Frobenius-Perron dimension extends to a ring homomorphism of $K(C)$ \cite[Theorem 8.2]{MR2183279} and hence is multiplicative over tensor products and additive over direct sums. 

A braided fusion category is a fusion category along with a natural isomorphism called the braiding $\gamma_{X,Y} : X\otimes Y \cong Y\otimes X$ that satisfies certain naturality conditions. For details see \cite{MR1250465}.

A pivotal fusion category is a fusion category along with an isomorphism from the double dual functor to the identity functor.

There are two useful invariants one can compute for a pivotal braided fusion category $C$. They are the $S$ and $T$ matrices, indexed by simple objects of $C$.

\hspace{2.2cm}$S_{X,Y}$ = \begin{tikzpicture}[baseline={([yshift=-.5ex]current bounding box.center)}]
\draw [thick] (0,0) arc [radius = 1, start angle = 0, end angle = 360];
\draw [thick,draw=white,double=black,double distance=\pgflinewidth,ultra thick] (-.43,.707) arc [radius = 1, start angle = 135, end angle = 135+360 - 10] ;
\node (X) at (.25,0) {$X$};
\node (Y) at (-.95,0) {$Y$};
\draw [thick] (-2,0) -- (-2+.15,.15);
\draw [thick] (-2,0) -- (-2.15,.15);

\draw [thick] (1.28,0) -- (1.28+.15,.15);
\draw [thick] (1.28,0) -- (1.28-.15,.15);

\end{tikzpicture},\hspace{2.2cm} $T_{X,X}$ =\begin{tikzpicture}[baseline={([yshift=-.5ex]current bounding box.center)}]

\draw [thick](0,0) arc [radius = .5, start angle = 180, end angle = 180+360 - 35];
\draw [thick] (0,0) -- (0,1.5);
\draw [thick] (0,1) -- (-.15,.85);
\draw [thick] (0,1) -- (.15,.85);
\draw [thick] (0,-.35) -- (0,-1.5 );
\node (X) at (.25,-1.5) {$X$};

\end{tikzpicture}.

If a pivotal braided tensor category has invertible $S$-matrix then we say that the category is modular. An important example of a modular category is the Drinfeld centre of a spherical fusion category $C$. For details see \cite{MR1966525}.
%
%
%

For a given fusion category $C$ the collection of auto-equivalences of $C$ forms a finite $2$-group, with the $2$-morphisms being the natural isomorphisms between auto-equivalences. We will write $\Aut_\otimes(C)$ to be the group of tensor auto-equivalences of $C$, up to natural isomorphism . Similarly we define $\BrAut(C)$ as the group of braided auto-equivalences of a braided fusion category.

\begin{dfn}\cite{MR1976459}
A left module category $M$ over a fusion category $C$ is a $\mathbb{C}$-linear category along with a bi-exact functor $ \otimes : C\times M \to M$, and natural isomorphisms $(X\otimes Y) \otimes M \to X \otimes (Y \otimes M)$ satisfying a straightfoward pentagon equation. 
\end{dfn}

One can think of module categories over a fusion category $C$ as categorifications of modules over the ring $K(C)$. 

In a similar way we can define a bimodule category over $C$. The collection of bimodule categories over a fusion category $C$ has the structure of a monoidal $2$-category. The morphisms between bimodules $M$ and $N$ are bimodule functors. The $2$-morphisms between bimodule functors are the natural transformations between the functors. The tensor product of bimodule categories is somewhat more complicated. We never explicitly work with this tensor product so we avoid giving a description. If the reader is interested details can be found in \cite{MR2677836}.


The $2$-category of bimodule categories over $C$ can be truncated to a group by taking only the invertible bimodules and collapsing isomorphisms to equalities. This group is called the Brauer-Picard group of $C$, or $\BrPic(C)$ for short. Explicitly this group consists of isomorphism classes of bimodule categories $[M]$ over $C$ such that there exists a $C$ bimodule category $N$ satisfying $M \boxtimes_C N \simeq C$ thinking of $C$ as a bimodule category over itself. 

Strongly related to the theory of modules over fusion categories are algebra objects in fusion categories.
\begin{dfn}\cite{MR1976459}
 An algebra object in a fusion category is an object $A$ together with morphisms $\mathbf{1} \to A$ and $A \otimes A \to A$ satisfying associator and unit axioms.
\end{dfn}
One can define left (right) module objects over an algebra $A$ in $C$. The category of left (right) $A$ modules in $C$ forms a right (left) module category over $C$. A result of Ostrik shows that every right (left) module category over $C$ arises as the category of left (right) $A$-modules for some algebra $A$ in $C$.

An invertible bimodule over $C$ consists of both an algebra $A$ in $C$ such that the category of $A$ bimodule objects, $A-\operatorname{bimod}_C$, is equivalent to $C$, and a choice of outer auto-equivalence of $C$. See \cite{MR2909758} for details.

The following lemma about algebra objects will be useful later in the paper.

\begin{lemma}\label{lem:algRes}
Let $G$ be a finite group, and $C \simeq \bigoplus_G C_g$ be a $G$-graded fusion category. If $A \in C_e$ is an algebra object then
\begin{equation*}
A-\operatorname{bimod}_C \simeq \bigoplus_G A-\operatorname{bimod}_{C_g},
\end{equation*}
where 
\begin{equation*}
A-\operatorname{bimod}_{C_g} :=   C_g \cap (A-\operatorname{bimod}_C).
\end{equation*}
\end{lemma}
\begin{proof}
There is a fully faithful functor $\bigoplus_G A-\operatorname{bimod}_{C_g}$ to $A-\operatorname{bimod}_C$ simply given by forgetting the homogeneous grading. Each of the $A-\operatorname{bimod}_{C_g}$ is an invertible bimodule over $A-\operatorname{bimod}_{C_e}$, so each has Frobenius-Perron dimension equal to that of $A-\operatorname{bimod}_{C_e}$. The Frobenius-Perron dimension of  $\bigoplus_G A-\operatorname{bimod}_{C_g}$ is therefore $|G|\operatorname{FPdim}(A-\operatorname{bimod}_{C_e}) =  |G|\operatorname{FPdim}(C_e)$. The Frobenius-Perron dimension of $A-\operatorname{bimod}_C$ is $\operatorname{FPdim}(C) = |G|\operatorname{FPdim}(C_e)$. Thus the above functor is a fully faithful functor between categories with the same Frobenius-Perron dimensions, and hence is an equivalence by \cite[Proposition 2.11]{MR2609644}.
\end{proof}

In general the problem of finding invertible bimodules over a fusion category is difficult. The following isomorphism due to Etingof, Nikshych, and Ostrik gives us a more concrete way to compute the Brauer-Picard group of a fusion category.
\begin{equation}\label{eq:iso}
\BrPic(C) \cong \BrAut(Z(C)).
\end{equation}
We give a brief discussion of one direction of the isomorphism in Section~\ref{sec:D10}. For more detail see \cite[Theorem 1.1]{MR2677836}.

\subsection{Planar Algebras}
Planar algebras were first introduced in \cite{math.QA/9909027} as an axiomatisation of the standard invariant of a subfactor.
\begin{dfn}
A planar algebra is a collection of vector spaces $\{P_n : n\in \mathbb{N}\}$ along with a multi-linear action of planar tangles.
\end{dfn} For more details see the above mentioned paper.

Given a planar algebra such that $P_0 \cong \mathbb{C}$ one can construct a pivotal monoidal category and vice versa.
\begin{dfn}\label{dfn:dsa}\cite{MR2559686}
Given a planar algebra $P$ we construct a monoidal category $C_P$ as follows:
\begin{itemize}
\item An object of $C_P$ is a projection in the algebra $P_{2n}$.
\item For two projections $\pi_1 \in P_{2n}, \pi_2 \in P_{2m}$ the morphism space $\Hom(\pi_1,\pi_2)$ is the vector space $\pi_1 P_{n+m} \pi_2$.
\item The tensor product of two projections is the disjoint union.
\item The tensor identity is the empty picture.
\item The dual of a projection is given by rotation by $180$ degrees.
\end{itemize}
\end{dfn}

For direct sums to make sense in this constructed category one has to work in the matrix category of $C_P$ where objects are formal direct sums of objects in $C_P$ and morphisms are matrices of morphisms in $C_P$. For more details on the matrix category see \cite{MR2559686}.

For simplicity of notation when we will write $P$ for both the planar algebra $P$, and the matrix category of $C_P$ as constructed above. The context of use should make it clear to the reader if we are referring to the planar algebra or corresponding fusion category.

Conversely, given a pivotal rigid monoidal category $C$, along with choice of symmetrically self-dual object $X$, one can construct a planar algebra $\text{PA}(C;X)_n := \operatorname{Hom}_C(\mathbf{1} \to X^{\otimes n})$. More details can be found in \cite{1607.06041}, where it is shown that this construction is inverse to the one described in Definition~\ref{dfn:dsa}.

Some of the simplest examples of planar algebras are the $ADE$ planar algebras. These are two infinite families $A_N$ and $D_{2N}$, and two sporadic examples, $E_6$ and $E_8$. These planar algebras were given the following generator and relation presentations in \cite{MR2559686,MR2577673}. We include these presentations in our preliminaries as these planar algebras (and associated fusion categories) are the main objects of study in this paper. 

%
%
%

To describe the $ADE$ planar algebras, and the planar algebras in the rest of this paper, we adopt the notation of \cite{MR2577673} for this paper so that for a diagram $X$:

\hspace{1.5cm}$\rho(X) := $  \begin{tikzpicture}[baseline={([yshift=-.5ex]current bounding box.center)}]
 \node [thick, rectangle, draw, minimum size = 29] at (0,0) (Z1) {$X$};
 \draw [thick] (-.4,.5) -- (-.4,.8);
 \draw [thick] (.25,.5) -- (.25,.8);
  \draw [thick] (-.25,-.5) -- (-.25,-.8);
    \draw [thick] (.4,-.5) -- (.4,-.8);
    \draw [dotted,thick] (.1,.65) -- (-.3,.65);
    \draw [dotted,thick] (.3,-.65) -- (-.1,-.65);
   \draw [thick] (.4,.5) arc [radius = .15,start angle = 180, end angle = 0];
   \draw [thick] (.7,.5) -- (.7,-.8);
   \draw [thick] (-.4,-.5) arc [radius = -.15,start angle = 180, end angle = 0];
   \draw [thick] (-.7,.8) -- (-.7,-.5);

 \end{tikzpicture} ,   \hspace{1.5cm}  $\tau(X) := $  \begin{tikzpicture}[baseline={([yshift=-.5ex]current bounding box.center)}]
 \node [thick, rectangle, draw, minimum size = 29] at (0,0) (Z1) {$X$};
 \draw [thick] (-.4,.5) arc [radius = .15, start angle = 180, end angle = 0];
 \draw [thick] (.4,.5) -- (.4,.8);
  \draw [thick] (0,.5) -- (0,.8);
  \draw [thick] (-.4,-.5) -- (-.4,-.8);
    \draw [thick] (.4,-.5) -- (.4,-.8);
    \draw [dotted,thick] (.3,.65) -- (.1,.65);
    \draw [dotted,thick] (.3,-.65) -- (-.3,-.65);

 \end{tikzpicture} ,\hspace{1.5cm}  $\hat{X} := $  \begin{tikzpicture}[baseline={([yshift=-.5ex]current bounding box.center)}]
 \node [thick, rectangle, draw, minimum size = 29] at (0,0) (Z1) {$X$};
  \draw [thick] (-.4,-.5) -- (-.4,-.8);
    \draw [thick] (.4,-.5) -- (.4,-.8);
    \draw [dotted,thick] (.3,-.65) -- (-.3,-.65);
    \draw [thick] (-.7,-.8) -- (-.7,.2);
     \draw [thick] (.7,-.8) -- (.7,.2);
     \draw [thick] (-.7,.2) arc [radius = .7, start angle = 180, end angle = 0];

 \end{tikzpicture}.

Composition of two elements in $P_{2N}$ is given by vertical stacking with $N$ strings pointing up and $N$ strings pointing down.

\begin{dfn}
Let $q$ be a root of unity. For $n$ a natural number we define the quantum integer $[n]_q := \frac{q^n - q^{-n}}{q - q^{-1}}$.
\end{dfn}

\begin{dfn}\cite{MR873400}
We define the Jones-Wenzl idempotents $f^{(n)}$ in a planar algebra $P$ via the recursive relation:
\[
\xyJonesWenzlIdempotentSeven{\JW{n+1}} =
\xyJonesWenzlIdempotentSixPlusOne{\JW{n}} - \frac{[n]_q}{[n+1]_q	}
\xyWenzlRecurrenceLastTermSixPlusOne{\JW{n}}.
\]
\end{dfn}

The $A_N$ planar algebras have loop parameter $[2]_q$, no generators, and a single relation. The $D_{2N}$, $E_6$, and $E_8$ planar algebras have loop parameter $[2]_q$, a single uncappable generator $S\in P_{k}$ with rotational eigenvalue $\omega$, satisfying relations as in the following table:

\begin{table}[h!]

    \centering
    
    \begin{tabular}{c|cccc}
    	\toprule
			     &   $q$ & $k$ & $\omega$ & additional relations  \\ 
	\midrule
			    $A_N$ & A primitive $2N+2$th root of unity & - & - & $f^{(N)} = 0$ \\[.3cm]
			    $D_{2N}$ &A primitive $8N-4$th root of unity & $4N-4$ & $\pm i$ & $S \otimes S = f^{(2N-2)}$      \\
			    &&&& $f^{(4N-3)} = 0$ 									 \\[.3cm] 	  
			    $E_6$ & A primitive $24$th root of unity& $6$ & $e^\frac{\pm 4\pi i}{3}$ & $S^2 = S + [2]_q^2[3]_qf^{(3)}$ \\ 	
			    &&&& $\hat{S}\circ f^{(8)} = 0$ \\[.3cm]
			    $E_8$ &A primitive $60$th root of unity & $10$ & $e^\frac{\pm 6\pi i}{5}$ & $S^2 = S + [2]_q^2[3]_qf^{(5)}$ \\	
			    &&&& $\hat{S}\circ f^{(12)} = 0$.\\
			                   
    	\bottomrule
    \end{tabular}
    \end{table}

%
%
%
%
%
%
%
%

\begin{rmk}
The $A_{N}$ planar algebras can be equipped with a braiding. There are multiple different braidings one can choose. The two we will use in this paper are the standard braiding 
\begin{center}
$\Braid = iq^\frac{1}{2} \Id - iq^\frac{-1}{2}\CupCap$ 
\end{center}
and the opposite braiding 
\begin{center}
$\Braid = - iq^\frac{-1}{2} \Id +iq^\frac{1}{2}\CupCap$. 
\end{center}
We call these braided planar algebras $A_{N}$ and $A_{N}^\text{bop}$ respectively. 
\end{rmk}

From these (braided) planar algebras we get (braided) fusion categories via the idempotent completion construction described earlier. We present the fusion graphs for these categories below:

\newpage
$A_{N}$ : \begin{tikzpicture}[baseline={([yshift=-.5ex]current bounding box.center)}]
        \vertex[label=$f^{(0)}$](f0) at (0,0) {};
        \vertex[label=$f^{(1)}$](f1) at (2,0) {};
        \vertex[label=$f^{(2)}$](f2) at (4,0) {};
         \vertex[label=$f^{(N-2)}$](f-2) at (7,0) {};
         \vertex[label=$f^{(N-1)}$](f-1) at (9,0) {};
        \Edge(f0)(f1)
        \Edge(f1)(f2)
        \Edge(f-2)(f-1)
     \tikzset{EdgeStyle/.style={dashed}}
        \Edge (f2)(f-2)
       
    \end{tikzpicture}\vspace{1cm}

$D_{2N}$ :   \begin{tikzpicture}[baseline={([yshift=-.5ex]current bounding box.center)}]
        \vertex[label=$f^{(0)}$](f0) at (0,0) {};
        \vertex[label=$f^{(1)}$](f1) at (2,0) {};
        \vertex[label=$f^{(2)}$](f2) at (4,0) {};
         \vertex[label=$f^{(2N-3)}$](f-3) at (7,0) {};
         \vertex[label=$P$](P) at (9,1) {};
         \vertex[label=$Q$](Q) at (9,-1) {};
        \Edge(f0)(f1)
        \Edge(f1)(f2)
        \Edge(f-2)(P)
        \Edge(f-2)(Q)
     \tikzset{EdgeStyle/.style={dashed}}
        \Edge (f2)(f-2)
       
    \end{tikzpicture}\vspace{1cm}

$E_{6}$ : \begin{tikzpicture}[baseline={([yshift=-.5ex]current bounding box.center)}]
        \vertex[label=$f^{(0)}$](f0) at (0,0) {};
        \vertex[label=$f^{(1)}$](f1) at (2,0) {};
        \vertex[label=above right:{$f^{(2)}$}](f2) at (4,0) {};
         \vertex[label=$X$](X) at (4,2) {};
          \vertex[label=$Y$](Y) at (6,0) {};
           \vertex[label=$Z$](Z) at (8,0) {};
        \Edge(f0)(f1)
        \Edge(f1)(f2)
        \Edge(f2)(X)
        \Edge(f2)(Y)
        \Edge(Y)(Z)
       
    \end{tikzpicture}\vspace{1cm}

$E_{8}$ :\begin{tikzpicture}[baseline={([yshift=-.5ex]current bounding box.center)}]
        \vertex[label=$f^{(0)}$](f0) at (0,0) {};
        \vertex[label=$f^{(1)}$](f1) at (2,0) {};
        \vertex[label=$f^{(2)}$](f2) at (4,0) {};
        \vertex[label=$f^{(3)}$](f3) at (6,0) {};
        \vertex[label=above right:{$f^{(4)}$}](f4) at (8,0) {};
         \vertex[label=$U$](U) at (8,2) {};
          \vertex[label=$V$](V) at (10,0) {};
           \vertex[label=$W$](W) at (12,0) {};
        \Edge(f0)(f1)
        \Edge(f1)(f2)
        \Edge(f2)(f3)
        \Edge(f3)(f4)
        \Edge(f4)(U)
        \Edge(f4)(V)
        \Edge(V)(W)
       
    \end{tikzpicture}

\vspace{1cm}
Where: 
\begin{itemize}\setlength\itemsep{1em}
 \item $P:= \frac{f^{(2N-2)} + S}{2}$,\quad 
 
 $Q:= \frac{f^{(2N-2)} - S}{2}$ in $D_{2N}$
\item  $X := \frac{1}{\sqrt{3}}f^{(3)} + \left(1 - \frac{2}{\sqrt{3}}\right)S$,\quad 

$Y := \left(1 -\frac{1}{\sqrt{3}}\right)f^{(3)} + \left(-1 + \frac{2}{\sqrt{3}}\right)S$,\quad 

$Z :=$ 
\xyJonesWenzlIdempotentSixPlusOne{Y} - $\frac{[3]_q}{[2]_q}$
\xyWenzlRecurrenceLastTermSixPlusOne{Y},
 in $E_6$
\item     $U :=  \frac{1}{4} \left(-\sqrt{5}+\sqrt{6 \left(1+\frac{1}{\sqrt{5}}\right)}+1\right)f^{(5)} + \frac{1}{2} \left(\sqrt{5}-\sqrt{6 \left(1+\frac{1}{\sqrt{5}}\right)}+1\right)S$,\quad 

$V := \frac{1}{4} \left(\sqrt{5}-\sqrt{6\left(1+\frac{1}{\sqrt{5}}\right)}+3\right)f^{(5)} +\frac{1}{2} \left(-\sqrt{5}+\sqrt{6 \left(1+\frac{1}{\sqrt{5}}\right)}-1\right)S$,\quad 

$W := $
\xyJonesWenzlIdempotentSixPlusOne{V} - $\frac{[5]_q}{\phi[2]_q}$
\xyWenzlRecurrenceLastTermSixPlusOne{V},
 in $E_8$.
\end{itemize}

Each of these categories is $\mathbb{Z} /2 \mathbb{Z}$-graded. The trivially graded piece of this grading is also a fusion category. These subcategories are called the adjoint subcategories.


The fusion graphs for the adjoint subcategories of the $ADE$ fusion categories are as follows:

$\ad(A_{2N})$ : 
    \begin{tikzpicture}[baseline={([yshift=-.5ex]current bounding box.center)}]
        \vertex[label=$f^{(0)}$](f0) at (0,0) {};
        \vertex[label=$f^{(2)}$](f2) at (2,0) {};
        \vertex[label=$f^{(4)}$](f4) at (4,0) {};
         \vertex[label=$f^{(2N-2)}$](fend) at (7,0) {};
        \Edge(f0)(f2)
        \Edge(f2)(f4)
        \Loop[dist = 1cm,dir = SO,style = {thick,-}](fend)(fend)
        \Loop[dist = 1cm,dir = SO,style = {thick,-}](f4)(f4)
        \Loop[dist = 1cm,dir = SO,style = {thick,-}](f2)(f2)

     \tikzset{EdgeStyle/.style={dashed}}
        \Edge (f4)(fend)
       
    \end{tikzpicture}

$\ad(A_{2N+1})$ : \begin{tikzpicture}[baseline={([yshift=-.5ex]current bounding box.center)}]
        \vertex[label=$f^{(0)}$](f0) at (0,0) {};
        \vertex[label=$f^{(2)}$](f2) at (2,0) {};
        \vertex[label=$f^{(4)}$](f4) at (4,0) {};
         \vertex[label=$f^{(2N-2)}$](f-2) at (7,0) {};
         \vertex[label=$f^{(2N)}$](fend) at (9,0) {};
        \Edge(f0)(f2)
        \Edge(f2)(f4)
        \Edge(f-2)(fend)
        \Loop[dist = 1cm,dir = SO,style = {thick,-}](f-2)(f-2)
         \Loop[dist = 1cm,dir = SO,style = {thick,-}](f2)(f2)
          \Loop[dist = 1cm,dir = SO,style = {thick,-}](f4)(f4)
     \tikzset{EdgeStyle/.style={dashed}}
        \Edge (f4)(f-2)
       
    \end{tikzpicture}

$\ad(D_{2N})$ :   \begin{tikzpicture}[baseline={([yshift=-.5ex]current bounding box.center)}]
        \vertex[label=$f^{(0)}$](f0) at (0,0) {};
        \vertex[label=$f^{(2)}$](f2) at (2,0) {};
        \vertex[label=$f^{(4)}$](f4) at (4,0) {};
         \vertex[label=$f^{(2N-4)}$](f-2) at (7,0) {};
         \vertex[label=$P$](P) at (9,1) {};
         \vertex[label=$Q$](Q) at (9,-1) {};
        \Edge(f0)(f2)
        \Edge(f2)(f4)
        \Edge(f-2)(P)
        \Edge(f-2)(Q)
        \Edge[style = bend right](Q)(P)
        \Loop[dist = 1cm,dir = SO,style = {thick,-}](f-2)(f-2)
         \Loop[dist = 1cm,dir = SO,style = {thick,-}](f2)(f2)
          \Loop[dist = 1cm,dir = SO,style = {thick,-}](f4)(f4)
     \tikzset{EdgeStyle/.style={dashed}}
        \Edge (f4)(f-2)
       
    \end{tikzpicture}

$\ad(E_{6})$ :\begin{tikzpicture}[baseline={([yshift=-.5ex]current bounding box.center)}]
        \vertex[label=$f^{(0)}$](f0) at (0,0) {};
        \vertex[label=above right:{$f^{(2)}$}](f2) at (2,0) {};
        \vertex[label=$Z$](Z) at (4,0) {};
        \Edge(f0)(f2)
        \Edge(f2)(Z)
        \Loop[dist = 1cm,dir = NO,style = {thick,-}](f2)(f2)
        \Loop[dist = 1cm,dir = SO,style = {thick,-}](f2)(f2)
       
    \end{tikzpicture}

$\ad(E_{8})$ : \begin{tikzpicture}[baseline={([yshift=-.5ex]current bounding box.center)}]
        \vertex[label=$f^{(0)}$](f0) at (0,0) {};
        \vertex[label=$f^{(2)}$](f2) at (2,0) {};
        \vertex[label=above right:{$f^{(4)}$}](f4) at (4,0) {};
         \vertex[label=$W$](W) at (6,0) {};
        \Edge(f0)(f2)
        \Edge(f2)(f4)
        \Edge(f4)(W)
        \Loop[dist = 1cm,dir = NO,style = {thick,-}](f4)(f4)
        \Loop[dist = 1cm,dir = SO,style = {thick,-}](f4)(f4)
        \Loop[dist = 1cm,dir = SO,style = {thick,-}](f2)(f2)
       
    \end{tikzpicture}

Planar algebra presentations for these adjoint subcategories can be acquired by taking the sub-planar algebra generated by the strand $f^{(2)}$ inside the full planar algebra. We explicitly describe the planar algebras for $\ad(A_{2N})$ and $\ad(D_{2N})$ as we will need them later on in this paper.

The $\ad(A_{N})$ planar algebra has a single generator $T \in P_{3}$ (which we draw as a trivalent vertex) with the following relations:
\begin{enumerate}
			\item $q$ a $2N+2$-th root of unity,
			\item \LOOP= $[3]_q$,
			\item $\rho(T) = T$,
			\item $\tau(T) = 0$,
			\item  \begin{tikzpicture}[scale = .5,baseline={([yshift=-.5ex]current bounding box.center)}]
				\draw [thick] (0,0) -- (0,1);
				\draw [thick] (0,1) -- (.707,1.707);
				\draw [thick] (0,1) -- (-.707,1.707);
				\draw [thick]  (.707,1.707) -- (0,2.414);
				\draw [thick]  (-.707,1.707) -- (0,2.414);	
				\draw [thick] (0,2.414)--(0,3.414);
				
 				\end{tikzpicture}  = $\left( \frac{[3]_q-1}{[2]_q} \right)$ \begin{tikzpicture}[scale = .5,baseline={([yshift=-.5ex]current bounding box.center)}]
				\draw [thick] (0,0) -- (0,3.414);
 				\end{tikzpicture},
			\item  \begin{tikzpicture}[scale = .5,baseline={([yshift=-.5ex]current bounding box.center)}]
				\draw [thick] (0,0) -- (.707,-.707);
				\draw [thick] (0,0) -- (-.707,-.707);
				\draw [thick] (0,0) -- (0,1);
				\draw [thick] (0,1) -- (.707,1.707);
				\draw [thick] (0,1) -- (-.707,1.707);
 				\end{tikzpicture} - \begin{tikzpicture}[scale = .5,baseline={([yshift=-.5ex]current bounding box.center)}]
				\draw [thick] (0,0) -- (-.707,.707);
				\draw [thick] (0,0) -- (-.707,-.707);
				\draw [thick] (0,0) -- (1,0);
				\draw [thick] (1,0) -- (1.707,.707);
				\draw [thick] (1,0) -- (1.707,-.707);
 				\end{tikzpicture}  = $\frac{1}{[2]_q} \Id - \frac{1}{[2]_q} \CupCap$.
\end{enumerate}


The $\ad(D_{2N})$ planar algebra has two generators $T\in P_3$ (again drawn as a trivalent vertex), and $S \in P_{2N-2}$ with the following relations:
  \begin{enumerate}
			\item $q$ a $8N-4$-th root of unity,
			\item \LOOP = $[3]_q$,
			\item $\rho(S) = -S$,
			\item $\rho(T) = T$,
			\item $\tau(S) = 0$,
			\item $\tau(T) = 0$,
			\item  \begin{tikzpicture}[scale = .5,baseline={([yshift=-.5ex]current bounding box.center)}]
				\draw [thick] (0,0) -- (0,1);
				\draw [thick] (0,1) -- (.707,1.707);
				\draw [thick] (0,1) -- (-.707,1.707);
				\draw [thick]  (.707,1.707) -- (0,2.414);
				\draw [thick]  (-.707,1.707) -- (0,2.414);	
				\draw [thick] (0,2.414)--(0,3.414);
				
 				\end{tikzpicture}  = $\left(\frac{[3]_q-1}{[2]_q}\right)$\begin{tikzpicture}[scale = .5,baseline={([yshift=-.5ex]current bounding box.center)}]
				\draw [thick] (0,0) -- (0,3.414);
 				\end{tikzpicture},
			\item  \begin{tikzpicture}[scale = .5,baseline={([yshift=-.5ex]current bounding box.center)}]
				\draw [thick] (0,0) -- (.707,-.707);
				\draw [thick] (0,0) -- (-.707,-.707);
				\draw [thick] (0,0) -- (0,1);
				\draw [thick] (0,1) -- (.707,1.707);
				\draw [thick] (0,1) -- (-.707,1.707);
 				\end{tikzpicture} - \begin{tikzpicture}[scale = .5,baseline={([yshift=-.5ex]current bounding box.center)}]
				\draw [thick] (0,0) -- (-.707,.707);
				\draw [thick] (0,0) -- (-.707,-.707);
				\draw [thick] (0,0) -- (1,0);
				\draw [thick] (1,0) -- (1.707,.707);
				\draw [thick] (1,0) -- (1.707,-.707);
 				\end{tikzpicture}  = $\frac{1}{[2]_q} \Id - \frac{1}{[2]_q} \CupCap$,
			\item $S \circ( T \otimes \operatorname{Id}_{N-3}) = 0$,
			\item $S\otimes S = f^{(N-1)}$.
			
\end{enumerate}

%
%

As the $\ad(A_N)$ planar algebras come from the $A_N$ planar algebras they inherit the braiding, so can also be thought of as braided planar algebras. We can now also equip the $\ad(D_{2N})$ planar algebras with a braiding. The two braidings on $\ad(D_{2N})$ that we care about are
\begin{center}
\Braid = $(q^2-1)$\Id + $q^{-2}$\CupCap - $(q^2 - q^{-2})$ \begin{tikzpicture}[scale = .6,baseline={([yshift=-.5ex]current bounding box.center)}]
				\draw [thick] (0,0) -- (-.707,.707);
				\draw [thick] (0,0) -- (-.707,-.707);
				\draw [thick] (0,0) -- (1,0);
				\draw [thick] (1,0) -- (1.707,.707);
				\draw [thick] (1,0) -- (1.707,-.707);
 				\end{tikzpicture},
\end{center}
and
\begin{center}
\Braid = $q^{-2}$\Id + $(q^{2}-1)$\CupCap - $(q^2 - q^{-2})$ \begin{tikzpicture}[scale = .6,baseline={([yshift=-.5ex]current bounding box.center)}]
				\draw [thick] (0,0) -- (.707,-.707);
				\draw [thick] (0,0) -- (-.707,-.707);
				\draw [thick] (0,0) -- (0,1);
				\draw [thick] (0,1) -- (.707,1.707);
				\draw [thick] (0,1) -- (-.707,1.707);
 				\end{tikzpicture},
\end{center}
We call these braided planar algebras $\ad(D_{2N})$ and $\ad(D_{2N})^\text{bop}$. Note that this is just the braiding of the $2$-coloured Jones polynomial specialised at $q$.

Note that there is some ambiguity in how we think of the objects of the adjoint subcategory. In the above fusion graphs the simple objects come directly from the corresponding full category, so the generating object for each of these is $f^{(2)}$. Yet if we construct the fusion categories from the above planar algebras the generating object will now be called $f^{(1)}$. To fix a convention for the rest of this paper we will regard objects of the adjoint subcategory as objects of the full category.

The goal of this paper is to compute the Brauer-Picard groups of all the fusion categories we have just described. To simplify our computations we notice that many of these fusion categories lie in the same Galois equivalence classes, which when paired with the following Lemma, reduces the number of fusion categories we need to compute the Brauer-Picard group for.

\begin{lemma}\label{lem:galoisPreservesBrPic}
Let $C$ and $D$ be Galois conjugate fusion categories, then $\BrPic(C) \cong \BrPic(D)$.
\end{lemma}
\begin{proof}
If $C$ and $D$ are Galois conjugate then so are $Z(C)$ and $Z(D)$. This Galois action carries each braided auto-equivalence of $Z(C)$ to a braided auto-equivalence of $Z(D)$ and vice versa.  
\end{proof}

To determine which categories are in the same Galois equivalence class we study the corresponding planar algebras. For a fixed $N$, all of the $A_{2N+1}$ planar algebras lie in a unique Galois orbit. The same holds for the $D_{2N}$, $E_6$, and $E_8$ planar algebras. The $A_{2N}$ planar algebras lie in two Galois orbits, depending on whether $\frac{2N+1}{i \pi}\operatorname{ln}(q)$ is even or odd. See \cite{AN-Survey} for more details on the $A_N$ cases. These Galois equivalences of planar algebras induce Galois equivalences of the corresponding fusion categories.

We now pick representatives from each Galois equivalence class. For each of the $A_{2N+1}$, $D_{2N}$, $E_6$, or $E_8$ families of fusion categories, our choice of representative is the category corresponding to the planar algebra with $q$ being the primitive root of unity closest to the positive real line, and positive rotational eigenvalue (when applicable). For each $A_{2N}$ family of fusion categories, our two representatives of the Galois orbits are the categories coming from the planar algebras with $q= e^\frac{i\pi}{2N+1}$ and $q= e^\frac{2i\pi}{2N+1}$. We call these fusion categories $A_{2N}^\text{deg}$ and $A_{2N}^\text{mod}$ respectively. These names are due to the fact that $A_{2N}^\text{mod}$ is a modular tensor category, while $A_{2N}^\text{deg}$ is not. The adjoint subcategories of $A_{2N+1}$, $D_{2N}$, $E_6$, and $E_8$, each inherit the unique Galois orbit of the full category. Furthermore the adjoint subcategories of $A_{2N}$ now also have a unique Galois orbit. Our choice of representative of the Galois orbit is the adjoint subcategory of $A_{2N}^\text{deg}$, which we will simply call $\ad(A_{2N})$.

\section{The Drinfeld Centres of the $ADE$ fusion categories}\label{sec:centres}

The main aim of this paper is to compute the Brauer-Picard groups of the $ADE$ fusion categories and adjoint subcategories. We do this via computing the braided auto-equivalence groups of the Drinfeld centre. This Section is devoted to computing these centres. With Proposition~\ref{prop:PA_to_FC} in mind is natural to want planar algebra presentations of the centres.

The centres of the full $ADE$ fusion categories are computed in \cite{MR1815993,AN-Survey}.
\begin{align*}
Z(A_{2N+1}) & \simeq A_{2N+1} \boxtimes A_{2N+1}^{\text{bop}}, \\
Z(A_{2N}^\text{mod}) & \simeq A_{2N}^\text{mod} \boxtimes {A_{2N}^\text{mod}}^{\text{bop}}, \\
Z(A_{2N}^\text{deg}) & \simeq \ad(A_{2N}^\text{deg}) \boxtimes \ad({A_{2N}^\text{deg}})^\text{bop} \boxtimes Z(\ad(A_3)), \\
Z(D_{2N}) &\simeq  A_{4N-3 } \boxtimes \ad(D_{2N})^{\text{bop}}, \\
Z(E_{6}) &\simeq A_{11} \boxtimes A_{3}^{\text{bop}}, \\
Z(E_{8}) &\simeq A_{29} \boxtimes \ad(A_{4})^{\text{bop}}.
\end{align*}
The "bops" on the right hand side all occur for the same reason. That is because of the general theorem from \cite{MR1815993}, stating that for a commutative algebra $A$ in a modular category $C$, the centre of $A-mod_C$ is $C\boxtimes (A-mod^0_C)^{\text{bop}}$ (here $A-mod^0_C$ is the category of dyslectic $A$ modules).

Note that we have given planar algebra presentations for the factors of these products in the preliminaries. 

The centres of the adjoint subcategories have yet to be described in the literature. When the adjoint subcategory admits a modular braiding the centres are trivial to compute.
\begin{lemma}
\begin{align*}
Z(\ad(A_{2N})) & \simeq \ad(A_{2N}) \boxtimes \ad(A_{2N})^{\text{bop}} \\
Z(\ad(D_{2N})) &\simeq  \ad(D_{2N}) \boxtimes \ad(D_{2N})^{\text{bop}}
\end{align*}
\end{lemma}
\begin{proof}
These fusion categories admit modular braidings, thus by a theorem of Muger \cite{MR1990929} the centres are as described.
\end{proof}

To compute the centres of the rest of the adjoint subcategories we appeal to a theorem of Gelaki, Naidu and Nikshych (stated below) which allows us to compute the centre of $\ad(C)$ given the centre of $C$. As we know the centres of all the $ADE$ fusion categories we are able to get explicit descriptions of the centres of the corresponding adjoint subcategories. We explicitly describe the details this theorem as we will be using it extensively for the rest of this Section.

\begin{cons}\label{rmk:centre_Construction}\cite[Corollary 3.7]{MR2587410}
Let $C$ be a $G$-graded fusion category, with trivially graded piece $C_0$. Consider all objects in the centre of $C$ that restrict to direct sums of the tensor identity in $C$. These objects form a $\operatorname{Rep}(G)$ subcategory. If we take the centraliser (see \cite{MR1990929} for details) of the $\Rep(G)$ subcategory of $Z(C)$, and de-equivariantize by the distinguished copy of $\operatorname{Rep}(G)$,  then the resulting category is braided equivalent to $Z(C_0)$.
\end{cons}

Straight away we can use this construction to get an explicit description of the centre of $\ad(E_8)$.
\begin{lemma}
$Z(\ad(E_8)) \simeq \ad(D_{16}) \boxtimes  \ad(A_4)^{\text{bop}}$
\end{lemma}

\begin{proof}
As there are only two simple objects of dimension $1$ in $A_{29} \boxtimes \ad(A_4)^{\text{bop}}$ it follows that the $\operatorname{Rep}(\mathbb{Z}/2\mathbb{Z})$ subcategory we are interested in must be the $\operatorname{Rep}(\mathbb{Z}/2\mathbb{Z})$ subcategory of $A_{29}$. We compute the centralizer of this subcategory to be $\ad(A_{29})\boxtimes \ad(A_4)^{\text{bop}}$. The result now follows as de-equivariantizing commutes with the taking of the adjoint subcategory, and $A_{29} // \operatorname{Rep}(\mathbb{Z}/2\mathbb{Z}) = D_{16}$.
\end{proof}

\subsection{Planar Algebras for the Drinfeld Centres of $\ad(A_{2N+1})$ and $\ad(E_6)$}

While we can apply the same techniques as in computing $Z(\ad(E_8))$ to computing $Z(\ad(A_{2N+1}))$ and $Z(\ad(E_6))$, the resulting categories don't naively have a planar algebra presentation. To achieve such a presentation we work through Construction~\ref{rmk:centre_Construction} on the level of the planar algebra.

We start by looking at the $\ad(A_{2N+1})$ family. The $\operatorname{Rep}(\mathbb{Z}/2\mathbb{Z})$ subcategory in $\color{blue} A_{2N+1}$ $\boxtimes$ $\color{red} A_{2N+1}^\text{bop}$ that we are interested in is generated by the simple object $\color{blue} f^{(2N)}$ $\boxtimes$ $\color{red} f^{(2N)}$. The centralizer of this subcategory is the subcategory of  $\color{blue} A_{2N+1}$ $\boxtimes$ $\color{red} A_{2N+1}^\text{bop}$ generated by the simple object $\color{blue}f^{(1)}$ $\boxtimes$ $\color{red}f^{(1)}$. Note that $\langle$$\color{blue}f^{(1)}$ $\boxtimes$ $\color{red}f^{(1)}$$\rangle$ is the tensor product planar algebra of $\color{blue} A_{2N+1}$ with $\color{red} A_{2N+1}^\text{bop}$ (see \cite[Section 2.3]{math.QA/9909027} for details on the tensor product of planar algebras).

The planar algebra for $\langle$$\color{blue}f^{(1)}$ $\boxtimes$ $\color{red}f^{(1)}$$\rangle$ can easily be described using basis vectors for the vector spaces. Let $\{t_i\}$ be the standard Temperley-Lieb basis of planar diagrams for $P_n$ of the planar algebra associated to $A_{2N+1}$. Then a basis for $P_n$ for the planar algebra of $\langle$$\color{blue}f^{(1)}$ $\boxtimes$ $\color{red}f^{(1)}$$\rangle$ are the vectors $\{$$\color{blue}t_i$$ \boxtimes$ $\color{red}t_j$$ \}$. Diagrammatically we can think of these basis vectors as a superposition of the two original basis vectors. Unfortunately this simple basis description of the planar algebra is not useful for computing automorphisms nor for taking de-equivariantizations. We want to find a generators and relations description (as in \cite{MR2577673,MR2559686}) as it is much better suited for these tasks.

The single strand in the $\langle$$\color{blue}f^{(1)}$ $\boxtimes$ $\color{red}f^{(1)}$$\rangle$ planar algebra is going to be the superposition of the red and blue strands, that is:
\begin{center}
\begin{tikzpicture}[baseline={([yshift=-.5ex]current bounding box.center)}]
\draw [thick] (0,0) -- (0,1);
\end{tikzpicture} := \begin{tikzpicture}[baseline={([yshift=-.5ex]current bounding box.center)}]
\draw [blue,thick] (0,0) -- (0,1);
\draw [red,thick] (0.05,0) -- (0.05,1);
\end{tikzpicture}.
\end{center}

Unfortunately just the strand does not generate the entire planar algebra, for example we can not construct the diagram 
\begin{center}\begin{tikzpicture}[baseline={([yshift=-.5ex]current bounding box.center)}]
\draw [red,thick] (0,0) arc [radius=.707,start angle =135, end angle =45];
\draw [red,thick] (0,1) arc [radius=.707,start angle = 225, end angle = 315];
\draw [blue,thick](0,0) arc [radius  = -.707,start angle = 135, end angle = 225];
\draw [blue,thick](1,0) arc [radius  = .707,start angle = 225, end angle = 135];
\end{tikzpicture} $\in P_4$\end{center}
using just the single strand. However it is proven in \cite[Theorem 3.1]{1308.5656} that the element
\begin{center} \begin{tikzpicture}[baseline={([yshift=-.5ex]current bounding box.center)}]
 \node [thick, rectangle, draw, minimum size = 29] at (0,0) (Z1) {$Z$};
\draw [thick] (-.4,-.5) arc [radius = .3, start angle = 0, end angle = -45];
\draw [thick] (.4,-.5) arc [radius = .3, start angle = 180, end angle = 225];
\draw [thick] (-.4,.5) arc [radius = -.3, start angle = 180, end angle = 225];
\draw [thick] (.4,.5) arc [radius = -.3, start angle = 0, end angle = -45]; 
 \end{tikzpicture} $:= [2]_q^{-1}$
\begin{tikzpicture}[baseline={([yshift=-.5ex]current bounding box.center)}]
\draw [red,thick] (0,0) arc [radius=.707,start angle =135, end angle =45];
\draw [red,thick] (0,1) arc [radius=.707,start angle = 225, end angle = 315];
\draw [blue,thick](0,0) arc [radius  = -.707,start angle = 135, end angle = 225];
\draw [blue,thick](1,0) arc [radius  = .707,start angle = 225, end angle = 135];
\end{tikzpicture} $\in P_4$ \end{center}
does generate everything.

 A generators and relations presentation of the braided planar algebra for $\langle \color{blue}f^{(1)}$ $\boxtimes$ $\color{red}f^{(1)}$$\rangle$ is given by

 \begin{align*}
& (1) \quad q = e^\frac{\pi i}{2N+2}, \quad  \quad \quad \hspace{.18cm} (2)\quad  \LOOP = [2]^2, \\[.5cm]&
 (3) \quad \begin{tikzpicture}[scale = 1,baseline={([yshift=-.5ex]current bounding box.center)}]
\draw [thick] (-.5,-.5) -- (-.5,.5);
\draw [thick] (-.5,.5) -- (.5,.5);
\draw [thick] (.5,.5) -- (.5,-.5);
\draw [thick] (.5,-.5) -- (-.5,-.5);
\node () at (0,0) {$Z$};
\draw [thick] (-.4,-.5) arc [radius = .3, start angle = 0, end angle = -45];
\draw [thick] (.4,-.5) arc [radius = .3, start angle = 180, end angle = 225];
\draw [thick] (-.4,.5) arc [radius = -.4, start angle = 0, end angle = -180];
 \end{tikzpicture} =\begin{tikzpicture}[baseline={([yshift=-.5ex]current bounding box.center)}]
 \draw [thick] (0,0) arc [radius = .5, start angle = 0, end angle = 180];
 \end{tikzpicture}, \quad  (4)\quad \begin{tikzpicture}[scale = 1,baseline={([yshift=-.5ex]current bounding box.center)}]
\draw [thick] (-.5,-.5) -- (-.5,.5);
\draw [thick] (-.5,.5) -- (.5,.5);
\draw [thick] (.5,.5) -- (.5,-.5);
\draw [thick] (.5,-.5) -- (-.5,-.5);
\node () at (0,0) {$\rotatebox{90}{Z}$};
\draw [thick] (-.4,-.5) arc [radius = .3, start angle = 0, end angle = -45];
\draw [thick] (.4,-.5) arc [radius = .3, start angle = 180, end angle = 225];
\draw [thick] (-.4,.5) arc [radius = -.4, start angle = 0, end angle = -180];
 \end{tikzpicture} =\begin{tikzpicture}[baseline={([yshift=-.5ex]current bounding box.center)}]
 \draw [thick] (0,0) arc [radius = .5, start angle = 0, end angle = 180];
 \end{tikzpicture}, \quad (5) \quad \rho^2 \left (\hspace{.1cm} \zDiagram[1]\hspace{.1cm}\right ) = \zDiagram[1], \\[.5cm]&
 (6) \quad \begin{tikzpicture}[scale = 1,baseline={([yshift=-.5ex]current bounding box.center)}]
\draw [thick] (-.5,-.5) -- (-.5,.5);
\draw [thick] (-.5,.5) -- (.5,.5);
\draw [thick] (.5,.5) -- (.5,-.5);
\draw [thick] (.5,-.5) -- (-.5,-.5);
\node () at (0,0) {$Z$};
\draw [thick] (-.4,.5) arc [radius = -.3, start angle = 180, end angle = 225];
\draw [thick] (.4,.5) arc [radius = -.3, start angle = 0, end angle = -45]; 
\draw [thick] (-.5,-.5-1.25) -- (-.5,.5-1.25);
\draw [thick] (-.5,.5-1.25) -- (.5,.5-1.25);
\draw [thick] (.5,.5-1.25) -- (.5,-.5-1.25);
\draw [thick] (.5,-.5-1.25) -- (-.5,-.5-1.25);
\node () at (0,-1.25) {$Z$};
\draw [thick] (-.4,-.5) -- (-.4,-.75);
\draw [thick] (.4,-.5) -- (.4,-.75);
\draw [thick] (-.4,-.5-1.25) arc [radius = .3, start angle = 0, end angle = -45];
\draw [thick] (.4,-.5-1.25) arc [radius = .3, start angle = 180, end angle = 225];
 \end{tikzpicture} = \zDiagram, \quad  
(7)\quad  \begin{tikzpicture}[scale = 1,baseline={([yshift=-.5ex]current bounding box.center)}]
\draw [thick] (-.5,-.5) -- (-.5,.5);
\draw [thick] (-.5,.5) -- (.5,.5);
\draw [thick] (.5,.5) -- (.5,-.5);
\draw [thick] (.5,-.5) -- (-.5,-.5);
\node () at (0,0) {$\rotatebox{90}{Z}$};
\draw [thick] (-.4,.5) arc [radius = -.3, start angle = 180, end angle = 225];
\draw [thick] (.4,.5) arc [radius = -.3, start angle = 0, end angle = -45]; 
\draw [thick] (-.5,-.5-1.25) -- (-.5,.5-1.25);
\draw [thick] (-.5,.5-1.25) -- (.5,.5-1.25);
\draw [thick] (.5,.5-1.25) -- (.5,-.5-1.25);
\draw [thick] (.5,-.5-1.25) -- (-.5,-.5-1.25);
\node () at (0,-1.25) {$\rotatebox{90}{Z}$};
\draw [thick] (-.4,-.5) -- (-.4,-.75);
\draw [thick] (.4,-.5) -- (.4,-.75);
\draw [thick] (-.4,-.5-1.25) arc [radius = .3, start angle = 0, end angle = -45];
\draw [thick] (.4,-.5-1.25) arc [radius = .3, start angle = 180, end angle = 225];
 \end{tikzpicture} = \rhoZDiagram, \quad
(8) \quad \begin{tikzpicture}[scale = 1,baseline={([yshift=-.5ex]current bounding box.center)}]
\draw [thick] (-.5,-.5) -- (-.5,.5);
\draw [thick] (-.5,.5) -- (.5,.5);
\draw [thick] (.5,.5) -- (.5,-.5);
\draw [thick] (.5,-.5) -- (-.5,-.5);
\node () at (0,0) {$Z$};
\draw [thick] (-.4,.5) arc [radius = -.3, start angle = 180, end angle = 225];
\draw [thick] (.4,.5) arc [radius = -.3, start angle = 0, end angle = -45]; 
\draw [thick] (-.5,-.5-1.25) -- (-.5,.5-1.25);
\draw [thick] (-.5,.5-1.25) -- (.5,.5-1.25);
\draw [thick] (.5,.5-1.25) -- (.5,-.5-1.25);
\draw [thick] (.5,-.5-1.25) -- (-.5,-.5-1.25);
\node () at (0,-1.25) {$\rotatebox{90}{Z}$};
\draw [thick] (-.4,-.5) -- (-.4,-.75);
\draw [thick] (.4,-.5) -- (.4,-.75);
\draw [thick] (-.4,-.5-1.25) arc [radius = .3, start angle = 0, end angle = -45];
\draw [thick] (.4,-.5-1.25) arc [radius = .3, start angle = 180, end angle = 225];
 \end{tikzpicture} = [2]_q^{-2} \CupCap, \\[.5cm]&
(9)\quad  \begin{tikzpicture}[baseline={([yshift=-.5ex]current bounding box.center)}]
 \node [thick, rectangle, draw, minimum size = 29] at (0,0) (Z1) {$Z$};
 \node [thick, rectangle, draw, minimum size = 29]  at (1.5,0) (Z2) {$Z$};
\draw [thick] (.16666,.5) arc [radius = .5833,start angle = 180, end angle = 0];
\draw [thick] (-.16666,.5) -- (.16666,1.5);
\draw [thick] (1.666666,.5) -- (1.34,1.5);
\draw [thick] (.1666,-.5) -- (.16666,-.8);
\draw [thick] (-.1666,-.5) -- (-.16666,-.8);
\draw [thick] (1.666,-.5) -- (1.6666,-.8);
\draw [thick] (1.34,-.5) -- (1.34,-.8);
 \end{tikzpicture} = \begin{tikzpicture}[baseline={([yshift=-.5ex]current bounding box.center)}]
 \node [thick, rectangle, draw, minimum size = 29] at (1,0) (Z1) {$Z$};
  \node [thick, rectangle, draw, minimum size = 29] at (0,-2) (Z1) {$Z$};
  \draw  [thick] (0.1666,-2.5) -- (0.1666,-2.8);
    \draw  [thick] (-0.1666,-2.5) -- (-0.16663,-2.8);
      \draw  [thick] (1-1.1666,-1.5) -- (1-.1666,-.5);
      \draw  [thick] (1+.1666,-.5) -- (1+.1666,-2.8);
   \draw [thick] (.83333,-1.5) arc [radius = .333,start angle = 0, end angle = 180];
   \draw [thick] (1-.16666,-1.5) -- (1-.16666,-2.8);
  \draw [thick] (1.1666,.5) -- (1.166666,.8);
   \draw [thick] (1-.1666,.5) -- (1-.16666,.8); 
 \end{tikzpicture}  = \begin{tikzpicture}[baseline={([yshift=-.5ex]current bounding box.center)}]
 \node [thick, rectangle, draw, minimum size = 29] at (0,0) (Z1) {$Z$};
  \node [thick, rectangle, draw, minimum size = 29] at (1,-2) (Z1) {$Z$};
  \draw  [thick] (1.1666,-2.5) -- (1.1666,-2.8);
    \draw  [thick] (.83333,-2.5) -- (.83333,-2.8);
      \draw  [thick] (1.1666,-1.5) -- (.1666,-.5);
      \draw  [thick] (-.1666,-.5) -- (-.1666,-2.8);
   \draw [thick] (.83333,-1.5) arc [radius = .333,start angle = 0, end angle = 180];
   \draw [thick] (.16666,-1.5) -- (.16666,-2.8);
   \draw [thick] (.1666,.5) -- (.166666,.8);
   \draw [thick] (-.1666,.5) -- (-.16666,.8);
 \end{tikzpicture}, \quad  
\\[.5cm]&
 (11) \quad \Braid = \Id + \CupCap -q \zDiagram[.8] - q^{-1}\rhoZDiagram[.8]. \\[.5cm]&
 (12) \quad \begin{tikzpicture}[baseline={([yshift=-.5ex]current bounding box.center)}]
 \node [thick, rectangle, draw, minimum size = 29] at (0,0) (S1) {$\color{blue}f^{(2N+1)}$ $\boxtimes$ $\color{red}f^{(2N+1)}$};
\draw [thick] (-.9,.5) -- (-.9,.8);
\draw [thick] (.9,.5) -- (.9,.8);
\draw [thick] (-.9,-.5) -- (-.9,-.8);
\draw [thick] (.9,-.5) -- (.9,-.8);
\draw [dotted,thick] (.8,.65) -- (-.8,.65);
\draw [dotted,thick] (.8,-.65) -- (-.8,-.65);
\end{tikzpicture} = 0.
 \end{align*}

All of these relations can easily be checked to hold in $\langle\color{blue}f^{(1)}$ $\boxtimes$ $\color{red}f^{(1)}$$\rangle$ by using the definition of $Z$. We have certainly overdone the number the relations required for this planar algebra. For example relations $(3)$ + $(5)$ + $(9)$ imply relation $(6)$. We have included the additional relations as it makes evaluation easier and there is little overhead in showing that the extra relations hold for $Z$.

To prove that we have given sufficient relations we need to show that we can evaluate any closed diagram. We can think of a closed diagram in the $\langle \color{blue}f^{(1)}$ $\boxtimes$ $\color{red}f^{(1)}$$\rangle$ planar algebra as a $4$-valent graph, with the vertices being $Z$'s or $\rho(Z)$'s. Any closed $4$-valent graph must contain either a triangle, bigon or a loop. We can remove a loop with relations $3$ and $4$. We can pop bigons with relations $6$, $7$ and $8$. Thus to complete our evaluation argument all we need to do is show that triangles can be popped as well.
 
 \begin{lemma}
 Suppose $D$ is a diagram containing three $Z$'s as a triangle, then using the above relations $D$ can be reduced to a diagram with at most two $Z$'s.
 \end{lemma}
 \begin{proof}
In $D$ two of the $Z$'s must form one of the diagrams appearing in relation $9$. Locally apply the relation to $D$ to obtain a diagram with two $Z$'s forming a bigon. Pop this bigon using one of relations $6$, $7$, or $8$ to obtain a diagram with one or two $Z$'s.
 \end{proof}
 
 By repeatedly popping triangles, bigons, and closed loops eventually one ends up at a scalar multiple of the empty diagram.
 
To de-equivariantize the $\langle \color{blue}f^{(1)}$ $\boxtimes$ $\color{red}f^{(1)}$$\rangle$ planar algebra by the subcategory $\langle\color{blue}f^{(2N)}$ $\boxtimes$ $\color{red}f^{(2N)}$$\rangle$ we need to add an isomorphism from the tensor identity to $\color{blue}f^{(2N)}$ $\boxtimes$ $\color{red}f^{(2N)}$. In planar algebra language this corresponds to adding a generator $S\in P_{2N}$  such that $SS^{-1}$ =$\color{blue}f^{(2N)}$ $\boxtimes$ $\color{red}f^{(2N)}$$  \in P_{4N}$ and $S^{-1}S = 1\in P_0$. Note that by taking the trace of the first condition we arrive at the second.

 See \cite{MR2559686} for an example of this in which they de-equivariantize the $A_{4N-3}$ planar algebra to obtain a planar algebra for $D_{2N}$. By doing the same procedure we arrive at the following presentation of the planar algebra for the centre of $\ad(A_{2N+1})$.
 
 \begin{prop}\label{prop:PA_Centre_A_Odd}
 The braided planar algebra for the Drinfeld centre of $\ad(A_{2N+1})$ has $2$ generators $Z \in P_4$ and $S \in P_{2N}$ satisfying the following relations:
 \begin{align*}
& (1) \quad q = e^\frac{\pi i}{2N+2}, \quad  \quad \quad \hspace{.18cm}  (2)\quad  \LOOP = [2]_q^2 ,\quad \\[.5cm]&
 (3)\quad \begin{tikzpicture}[scale = 1,baseline={([yshift=-.5ex]current bounding box.center)}]
\draw [thick] (-.5,-.5) -- (-.5,.5);
\draw [thick] (-.5,.5) -- (.5,.5);
\draw [thick] (.5,.5) -- (.5,-.5);
\draw [thick] (.5,-.5) -- (-.5,-.5);
\node () at (0,0) {$Z$};
\draw [thick] (-.4,-.5) arc [radius = .3, start angle = 0, end angle = -45];
\draw [thick] (.4,-.5) arc [radius = .3, start angle = 180, end angle = 225];
\draw [thick] (-.4,.5) arc [radius = -.4, start angle = 0, end angle = -180];
 \end{tikzpicture} =\begin{tikzpicture}[baseline={([yshift=-.5ex]current bounding box.center)}]
 \draw [thick] (0,0) arc [radius = .5, start angle = 0, end angle = 180];
 \end{tikzpicture}, \quad  (4)\quad \begin{tikzpicture}[scale = 1,baseline={([yshift=-.5ex]current bounding box.center)}]
\draw [thick] (-.5,-.5) -- (-.5,.5);
\draw [thick] (-.5,.5) -- (.5,.5);
\draw [thick] (.5,.5) -- (.5,-.5);
\draw [thick] (.5,-.5) -- (-.5,-.5);
\node () at (0,0) {$\rotatebox{90}{Z}$};
\draw [thick] (-.4,-.5) arc [radius = .3, start angle = 0, end angle = -45];
\draw [thick] (.4,-.5) arc [radius = .3, start angle = 180, end angle = 225];
\draw [thick] (-.4,.5) arc [radius = -.4, start angle = 0, end angle = -180];
 \end{tikzpicture} =\begin{tikzpicture}[baseline={([yshift=-.5ex]current bounding box.center)}]
 \draw [thick] (0,0) arc [radius = .5, start angle = 0, end angle = 180];
 \end{tikzpicture}, \quad (5) \quad \rho^2\left(\hspace{.1cm}\zDiagram[1]\hspace{.1cm}\right) = \zDiagram[1] \\[.5cm]&
 (6) \quad \begin{tikzpicture}[scale = 1,baseline={([yshift=-.5ex]current bounding box.center)}]
\draw [thick] (-.5,-.5) -- (-.5,.5);
\draw [thick] (-.5,.5) -- (.5,.5);
\draw [thick] (.5,.5) -- (.5,-.5);
\draw [thick] (.5,-.5) -- (-.5,-.5);
\node () at (0,0) {$Z$};
\draw [thick] (-.4,.5) arc [radius = -.3, start angle = 180, end angle = 225];
\draw [thick] (.4,.5) arc [radius = -.3, start angle = 0, end angle = -45]; 
\draw [thick] (-.5,-.5-1.25) -- (-.5,.5-1.25);
\draw [thick] (-.5,.5-1.25) -- (.5,.5-1.25);
\draw [thick] (.5,.5-1.25) -- (.5,-.5-1.25);
\draw [thick] (.5,-.5-1.25) -- (-.5,-.5-1.25);
\node () at (0,-1.25) {$Z$};
\draw [thick] (-.4,-.5) -- (-.4,-.75);
\draw [thick] (.4,-.5) -- (.4,-.75);
\draw [thick] (-.4,-.5-1.25) arc [radius = .3, start angle = 0, end angle = -45];
\draw [thick] (.4,-.5-1.25) arc [radius = .3, start angle = 180, end angle = 225];
 \end{tikzpicture} = \zDiagram, \quad  
(7)\quad  \begin{tikzpicture}[scale = 1,baseline={([yshift=-.5ex]current bounding box.center)}]
\draw [thick] (-.5,-.5) -- (-.5,.5);
\draw [thick] (-.5,.5) -- (.5,.5);
\draw [thick] (.5,.5) -- (.5,-.5);
\draw [thick] (.5,-.5) -- (-.5,-.5);
\node () at (0,0) {$\rotatebox{90}{Z}$};
\draw [thick] (-.4,.5) arc [radius = -.3, start angle = 180, end angle = 225];
\draw [thick] (.4,.5) arc [radius = -.3, start angle = 0, end angle = -45]; 
\draw [thick] (-.5,-.5-1.25) -- (-.5,.5-1.25);
\draw [thick] (-.5,.5-1.25) -- (.5,.5-1.25);
\draw [thick] (.5,.5-1.25) -- (.5,-.5-1.25);
\draw [thick] (.5,-.5-1.25) -- (-.5,-.5-1.25);
\node () at (0,-1.25) {$\rotatebox{90}{Z}$};
\draw [thick] (-.4,-.5) -- (-.4,-.75);
\draw [thick] (.4,-.5) -- (.4,-.75);
\draw [thick] (-.4,-.5-1.25) arc [radius = .3, start angle = 0, end angle = -45];
\draw [thick] (.4,-.5-1.25) arc [radius = .3, start angle = 180, end angle = 225];
 \end{tikzpicture} = \rhoZDiagram, \quad
(8) \quad \begin{tikzpicture}[scale = 1,baseline={([yshift=-.5ex]current bounding box.center)}]
\draw [thick] (-.5,-.5) -- (-.5,.5);
\draw [thick] (-.5,.5) -- (.5,.5);
\draw [thick] (.5,.5) -- (.5,-.5);
\draw [thick] (.5,-.5) -- (-.5,-.5);
\node () at (0,0) {$Z$};
\draw [thick] (-.4,.5) arc [radius = -.3, start angle = 180, end angle = 225];
\draw [thick] (.4,.5) arc [radius = -.3, start angle = 0, end angle = -45]; 
\draw [thick] (-.5,-.5-1.25) -- (-.5,.5-1.25);
\draw [thick] (-.5,.5-1.25) -- (.5,.5-1.25);
\draw [thick] (.5,.5-1.25) -- (.5,-.5-1.25);
\draw [thick] (.5,-.5-1.25) -- (-.5,-.5-1.25);
\node () at (0,-1.25) {$\rotatebox{90}{Z}$};
\draw [thick] (-.4,-.5) -- (-.4,-.75);
\draw [thick] (.4,-.5) -- (.4,-.75);
\draw [thick] (-.4,-.5-1.25) arc [radius = .3, start angle = 0, end angle = -45];
\draw [thick] (.4,-.5-1.25) arc [radius = .3, start angle = 180, end angle = 225];
 \end{tikzpicture} = [2]_q^{-2} \CupCap, \\[.5cm]&
(9)\quad  \begin{tikzpicture}[baseline={([yshift=-.5ex]current bounding box.center)}]
 \node [thick, rectangle, draw, minimum size = 29] at (0,0) (Z1) {$Z$};
 \node [thick, rectangle, draw, minimum size = 29]  at (1.5,0) (Z2) {$Z$};
\draw [thick] (.16666,.5) arc [radius = .5833,start angle = 180, end angle = 0];
\draw [thick] (-.16666,.5) -- (.16666,1.5);
\draw [thick] (1.666666,.5) -- (1.34,1.5);
\draw [thick] (.1666,-.5) -- (.16666,-.8);
\draw [thick] (-.1666,-.5) -- (-.16666,-.8);
\draw [thick] (1.666,-.5) -- (1.6666,-.8);
\draw [thick] (1.34,-.5) -- (1.34,-.8);
 \end{tikzpicture} = \begin{tikzpicture}[baseline={([yshift=-.5ex]current bounding box.center)}]
 \node [thick, rectangle, draw, minimum size = 29] at (1,0) (Z1) {$Z$};
  \node [thick, rectangle, draw, minimum size = 29] at (0,-2) (Z1) {$Z$};
  \draw  [thick] (0.1666,-2.5) -- (0.1666,-2.8);
    \draw  [thick] (-0.1666,-2.5) -- (-0.16663,-2.8);
      \draw  [thick] (1-1.1666,-1.5) -- (1-.1666,-.5);
      \draw  [thick] (1+.1666,-.5) -- (1+.1666,-2.8);
   \draw [thick] (.83333,-1.5) arc [radius = .333,start angle = 0, end angle = 180];
   \draw [thick] (1-.16666,-1.5) -- (1-.16666,-2.8);
  \draw [thick] (1.1666,.5) -- (1.166666,.8);
   \draw [thick] (1-.1666,.5) -- (1-.16666,.8); 
 \end{tikzpicture}  = \begin{tikzpicture}[baseline={([yshift=-.5ex]current bounding box.center)}]
 \node [thick, rectangle, draw, minimum size = 29] at (0,0) (Z1) {$Z$};
  \node [thick, rectangle, draw, minimum size = 29] at (1,-2) (Z1) {$Z$};
  \draw  [thick] (1.1666,-2.5) -- (1.1666,-2.8);
    \draw  [thick] (.83333,-2.5) -- (.83333,-2.8);
      \draw  [thick] (1.1666,-1.5) -- (.1666,-.5);
      \draw  [thick] (-.1666,-.5) -- (-.1666,-2.8);
   \draw [thick] (.83333,-1.5) arc [radius = .333,start angle = 0, end angle = 180];
   \draw [thick] (.16666,-1.5) -- (.16666,-2.8);
   \draw [thick] (.1666,.5) -- (.166666,.8);
   \draw [thick] (-.1666,.5) -- (-.16666,.8);
 \end{tikzpicture}, \quad  
\\[.5cm]& 
 (10) \quad \begin{tikzpicture}[baseline={([yshift=-.5ex]current bounding box.center)}]
 \node [thick, rectangle, draw, minimum size = 29] at (0,0) (S1) {$S$};
 \node [thick, rectangle, draw, minimum size = 29] at (0,-1.5) (S2) {$S$};
\draw [thick] (-.3,.5) -- (-.3,.8);
\draw [thick] (.3,.5) -- (.3,.8);
\draw [dotted,thick] (-.2,.65) -- (.2,.65);
\draw [dotted,thick] (-.2,.65) -- (.2,.65);
\draw [thick] (-.3,-2) -- (-.3,-2.3);
\draw [thick] (.3,-2) -- (.3,-2.3);
\draw [dotted,thick] (-.2,-2.15) -- (.2,-2.15);
\draw [dotted,thick] (-.2,-2.15) -- (.2,-2.15);
\end{tikzpicture} = \begin{tikzpicture}[baseline={([yshift=-.5ex]current bounding box.center)}]
 \node [thick, rectangle, draw, minimum size = 29] at (0,0) (S1) {$\color{blue}f^{(2N)}$ $\boxtimes$ $\color{red}f^{(2N)}$};
\draw [thick] (-.9,.5) -- (-.9,.8);
\draw [thick] (.9,.5) -- (.9,.8);
\draw [thick] (-.9,-.5) -- (-.9,-.8);
\draw [thick] (.9,-.5) -- (.9,-.8);
\draw [dotted,thick] (.8,.65) -- (-.8,.65);
\draw [dotted,thick] (.8,-.65) -- (-.8,-.65);
\end{tikzpicture},
\quad (11) \quad \begin{tikzpicture}[baseline={([yshift=-.5ex]current bounding box.center)}]
 \node [thick, rectangle, draw, minimum size = 29] at (0,0) (S1) {$S$};
\draw [thick] (-.3,.5) -- (-.3,.8);
\draw [thick] (.3,.5) -- (.3,.8);
\draw [dotted,thick] (-.2,.65) -- (.2,.65);
\draw [dotted,thick] (-.2,.65) -- (.2,.65);
\draw [thick] circle [radius = .86];
\draw [thick] circle [radius = 1.86];
\path[decorate,decoration={text along path,
text={Any diagram with no $S$'s},text align=center}]
(-1.36,0) arc [start angle=-180,end angle=0,radius=1.4];
\end{tikzpicture} =0, \\[2ex]
 &(12) \quad \Braid = \Id + \CupCap -q \zDiagram[.8] - q^{-1}\rhoZDiagram[.8],  \\[.5cm]
 & (13) \quad \begin{tikzpicture}[baseline={([yshift=-.5ex]current bounding box.center)}]
 \node [thick, rectangle, draw, minimum size = 29] at (0,0) (S1) {$\color{blue}f^{(2N+1)}$ $\boxtimes$ $\color{red}f^{(2N+1)}$};
\draw [thick] (-.9,.5) -- (-.9,.8);
\draw [thick] (.9,.5) -- (.9,.8);
\draw [thick] (-.9,-.5) -- (-.9,-.8);
\draw [thick] (.9,-.5) -- (.9,-.8);
\draw [dotted,thick] (.8,.65) -- (-.8,.65);
\draw [dotted,thick] (.8,-.65) -- (-.8,-.65);
\end{tikzpicture} = 0.
 \end{align*}
 \end{prop}
Again we have almost certainly overdone the relations to obtain a nice evaluation algorithm. However now it is not so obvious why relation (11) should exist in this planar algebra.

Consider a diagram $D$ that only contains a single $S$. By attaching an $S$ to the bottom of relation (10) we get the relation:

\begin{center}\begin{tikzpicture}[baseline={([yshift=-.5ex]current bounding box.center)}]
 \node [thick, rectangle, draw, minimum size = 29] at (0,0) (S1) {$S$};
\draw [thick] (-.3,.5) -- (-.3,.8);
\draw [thick] (.3,.5) -- (.3,.8);
\draw [dotted,thick] (-.2,.65) -- (.2,.65);
\draw [dotted,thick] (-.2,.65) -- (.2,.65);
\end{tikzpicture}\quad=\quad\begin{tikzpicture}[baseline={([yshift=-.5ex]current bounding box.center)}]
 \node [thick, rectangle, draw, minimum size = 29] at (0,0) (S1) {$S$};
\draw [thick] (-.3,.5) -- (-.3,.8);
\draw [thick] (.3,.5) -- (.3,.8);
\draw [dotted,thick] (-.5,1.1) -- (.5,1.1);
 \node [thick, rectangle, draw, minimum size = 29] at (0,2) (S1) {$\color{blue}f^{(2N)}$ $\boxtimes$ $\color{red}f^{(2N)}$};
\draw [thick] (-.9,.5+2) -- (-.9,.8+2);
\draw [thick] (.9,.5+2) -- (.9,.8+2);
\draw [thick] (-.9,-.5+2) -- (-.9,-.8+2);
\draw [thick] (.9,-.5+2) -- (.9,-.8+2);
\draw [dotted,thick] (.8,.65+2) -- (-.8,.65+2);
\draw [thick] (-.9,1.2) -- (-.3,.8);
\draw [thick] (.9,1.2) -- (.3,.8);
\end{tikzpicture} 
\end{center}

Therefore in $D$ we can fit a $\color{blue}f^{(2N)}$ $\boxtimes$ $\color{red}f^{(2N)}$ between the $S$ and the rest of the diagram. Consider the blue component of the diagram. There must be a blue cap on $\color{blue}f^{(2N)}$ which kills the diagram.

The reader might argue that relation (11) isn't really a satisfying relation, as it really involves an infinite number of diagrams. The relation is also non-local as it is defined on entire diagrams. It appears that in the $Z(\ad(A_{2N+1}))$ planar algebra one can deduce relation (11) from relation (10). However this argument is fairly lengthy. As the above presentation is sufficient for our purposes in this paper we omit this argument and work with the less satisfying relations. As we will see soon, the planar algebra for $Z(\ad(E_6))$ has a similar unsatisfying relation. For this case we are unsure if the relation can be deduced from the others, though we suspect that it can.

 \begin{proof}[Proof of Proposition \ref{prop:PA_Centre_A_Odd}]
 To show that we have given enough relations we have to show that any closed diagram can be evaluated to a  scalar. If we have a diagram with multiple $S$'s, then using the fact that the planar algebra is braided along with relation (10) we can reduce to a diagram with $0$ or $1$ $S$'s. See \cite{MR2559686} for an in depth description on how this algorithm works. We have already shown that any diagram with $0$ $S$'s can be evaluated. Relation (11) exactly says that any diagram with $1$ $S$ is evaluated to $0$ as in such a diagram there must be a cup between two adjacent strands of $S$. 

We now claim that this planar algebra is exactly the planar algebra for the de-equivariantization of $\langle \color{blue}f^{(1)}$ $\boxtimes$ $\color{red}f^{(1)}$$\rangle$ by the copy of $\operatorname{Rep}(\mathbb{Z}/2\mathbb{Z}) \simeq \langle \color{blue}f^{(2N)}$ $\boxtimes$ $\color{red}f^{(2N)}$$\rangle$. To see this notice that the above planar algebra has an order two automorphism $S \mapsto -S$. Taking the fixed point planar algebra under this action of $\mathbb{Z}/2\mathbb{Z}$ gives the sub-planar algebra consisting of elements with an even number of $S$'s. However the braiding along with relation (10) implies that such a diagram is equal to a diagram with no $S$'s. Hence the fixed point planar algebra recovers the planar algebra for $\langle \color{blue}f^{(1)}$ $\boxtimes$ $\color{red}f^{(1)}$$\rangle$. 

Theorem 4.44 of \cite{MR2609644} states that braided fusion categories containing $\operatorname{Rep}(G)$, and $G$-crossed braided fusion categories, are in bijection via equivariantization, and de-equivariantization. As the braided tensor category associated to the above planar algebra equivariantizes to give $\langle \color{blue}f^{(1)}$ $\boxtimes$ $\color{red}f^{(1)}$$\rangle$, it follows that the braided tensor category associated to the above planar algebra is a de-equivariantization of $\langle \color{blue}f^{(1)}$ $\boxtimes$ $\color{red}f^{(1)}$$\rangle$ by some copy of $\operatorname{Rep}(\mathbb{Z}/2\mathbb{Z})$. We have to show that the copy of $\operatorname{Rep}(\mathbb{Z}/2\mathbb{Z})$ is precisely $\langle \color{blue}f^{(2N)}$ $\boxtimes$ $\color{red}f^{(2N)}$$\rangle$.

When $N$ is odd there is a unique copy of $\operatorname{Rep}(\mathbb{Z}/2\mathbb{Z})$ in $\langle \color{blue}f^{(1)}$ $\boxtimes$ $\color{red}f^{(1)}$$\rangle$ and so the claim is trivial. When $N$ is even there are 3 distinct copies, generated by the objects $ \color{blue}f^{(0)}$ $\boxtimes$ $\color{red}f^{(N)}$, $ \color{blue}f^{(N)}$ $\boxtimes$ $\color{red}f^{(0)}$, and $\color{blue}f^{(N)}$ $\boxtimes$ $\color{red}f^{(N)}$. The de-equivariantizations corresponding to these subcategories are $\langle \color{blue}f^{(1)}$ $\boxtimes$ $\color{red}f^{(1)}$$\rangle \subset$ $\color{blue}A_{2N+1}$ $\boxtimes$ $\color{red}D_{N+2}^\text{op}$, $\langle \color{blue}f^{(1)}$ $\boxtimes$ $\color{red}f^{(1)}$$\rangle \subset$ $\color{blue} D_{N+2}$ $\boxtimes$ $\color{red}A_{2N+1}^\text{op}$, and $Z(\ad(A_{2N+1}))$.

The former two categories don't admit braidings, and hence can't come from the above planar algebra. Thus the corresponding fusion category for the above planar algebra is $Z(\ad(A_{2N+1}))$.

\end{proof}

Note that when $N=1$ we have a generators and relations presentation of the planar algebra associated to $Z( \Rep( \mathbb{Z} / 2\mathbb{Z}))$, the toric code.
 
The computation for the centre of $\ad(E_6)$ is almost identical, except we now start with the $\color{blue} A_{11}$ $\boxtimes$ $\color{red} A_{3}^\text{bop}$ planar algebra. The only difference in the computation is that the object $\color{blue}f^{(10)}$ $\boxtimes$ $\color{red}f^{(2)}$ doesn't naively make sense in planar algebra language. To deal with this one has to choose an isomorphic copy of $\color{red}f^{(2)}$ living in $\color{red}P_{20}$ boxspace of the $\color{red}A_3^\text{bop}$ planar algebra. For example you can choose the projection:

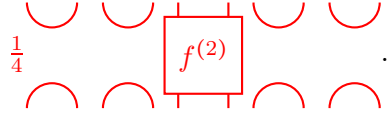
\begin{center}$\color{red}\frac{1}{4}$\begin{tikzpicture}[baseline={([yshift=-.5ex]current bounding box.center)}]
\node [thick,rectangle,red,minimum size = 29,draw] at (0,0) (f2) {$f^{(2)}$};
\draw [thick,red] (0.33,.5) -- (0.33,0.7);
\draw [thick,red] (-0.33,.5) -- (-0.33,0.7);
\draw [thick,red] (0.33,-.5) -- (0.33,-0.7);
\draw [thick,red] (-0.33,-.5) -- (-0.33,-0.7);
\draw [thick,red] (.66,.7) arc [radius = -.33, start angle = 0, end angle = 180];
\draw [thick,red] (.66+1,.7) arc [radius = -.33, start angle = 0, end angle = 180];
\draw [thick,red] (.66,-.7) arc [radius = -.33, start angle = 0, end angle = -180];
\draw [thick,red] (.66+1,-.7) arc [radius = -.33, start angle = 0, end angle = -180];
\draw [thick,red] (-.66,.7) arc [radius = -.33, start angle = 180, end angle = 0];
\draw [thick,red] (-.66-1,.7) arc [radius = -.33, start angle = 180, end angle = 0];
\draw [thick,red] (-.66,-.7) arc [radius = .33, start angle = 0, end angle = 180];
\draw [thick,red] (-.66-1,-.7) arc [radius = .33, start angle = 0, end angle = 180];
 \end{tikzpicture}.\end{center}

The resulting planar algebra for the centre is as follows: 

  \begin{prop}
 The braided planar algebra for the Drinfeld centre of $\ad(E_6)$ has $2$ generators $Z \in P_4$ and $S \in P_{10}$ satisfying the following relations:
\begin{align*}
& (1) \quad q = e^\frac{\pi i}{12} , \quad  \quad  \quad \quad \hspace{.16cm}   (2)\quad  \LOOP = \sqrt{2}[2]_q ,\quad \\[.5cm]&
 (3)\quad \begin{tikzpicture}[scale = 1,baseline={([yshift=-.5ex]current bounding box.center)}]
\draw [thick] (-.5,-.5) -- (-.5,.5);
\draw [thick] (-.5,.5) -- (.5,.5);
\draw [thick] (.5,.5) -- (.5,-.5);
\draw [thick] (.5,-.5) -- (-.5,-.5);
\node () at (0,0) {$Z$};
\draw [thick] (-.4,-.5) arc [radius = .3, start angle = 0, end angle = -45];
\draw [thick] (.4,-.5) arc [radius = .3, start angle = 180, end angle = 225];
\draw [thick] (-.4,.5) arc [radius = -.4, start angle = 0, end angle = -180];
 \end{tikzpicture} =\begin{tikzpicture}[baseline={([yshift=-.5ex]current bounding box.center)}]
 \draw [thick] (0,0) arc [radius = .5, start angle = 0, end angle = 180];
 \end{tikzpicture}, \quad  (4)\quad \begin{tikzpicture}[scale = 1,baseline={([yshift=-.5ex]current bounding box.center)}]
\draw [thick] (-.5,-.5) -- (-.5,.5);
\draw [thick] (-.5,.5) -- (.5,.5);
\draw [thick] (.5,.5) -- (.5,-.5);
\draw [thick] (.5,-.5) -- (-.5,-.5);
\node () at (0,0) {$\rotatebox{90}{Z}$};
\draw [thick] (-.4,-.5) arc [radius = .3, start angle = 0, end angle = -45];
\draw [thick] (.4,-.5) arc [radius = .3, start angle = 180, end angle = 225];
\draw [thick] (-.4,.5) arc [radius = -.4, start angle = 0, end angle = -180];
 \end{tikzpicture} =\frac{[2]_q}{\sqrt{2}}\begin{tikzpicture}[baseline={([yshift=-.5ex]current bounding box.center)}]
 \draw [thick] (0,0) arc [radius = .5, start angle = 0, end angle = 180];
 \end{tikzpicture}, \quad (5) \quad \rho^2(\zDiagram[1]) = \zDiagram[1], \\[.5cm]&
 (6) \quad \begin{tikzpicture}[scale = 1,baseline={([yshift=-.5ex]current bounding box.center)}]
\draw [thick] (-.5,-.5) -- (-.5,.5);
\draw [thick] (-.5,.5) -- (.5,.5);
\draw [thick] (.5,.5) -- (.5,-.5);
\draw [thick] (.5,-.5) -- (-.5,-.5);
\node () at (0,0) {$Z$};
\draw [thick] (-.4,.5) arc [radius = -.3, start angle = 180, end angle = 225];
\draw [thick] (.4,.5) arc [radius = -.3, start angle = 0, end angle = -45]; 
\draw [thick] (-.5,-.5-1.25) -- (-.5,.5-1.25);
\draw [thick] (-.5,.5-1.25) -- (.5,.5-1.25);
\draw [thick] (.5,.5-1.25) -- (.5,-.5-1.25);
\draw [thick] (.5,-.5-1.25) -- (-.5,-.5-1.25);
\node () at (0,-1.25) {$Z$};
\draw [thick] (-.4,-.5) -- (-.4,-.75);
\draw [thick] (.4,-.5) -- (.4,-.75);
\draw [thick] (-.4,-.5-1.25) arc [radius = .3, start angle = 0, end angle = -45];
\draw [thick] (.4,-.5-1.25) arc [radius = .3, start angle = 180, end angle = 225];
 \end{tikzpicture} = \zDiagram, \quad  
(7)\quad  \begin{tikzpicture}[scale = 1,baseline={([yshift=-.5ex]current bounding box.center)}]
\draw [thick] (-.5,-.5) -- (-.5,.5);
\draw [thick] (-.5,.5) -- (.5,.5);
\draw [thick] (.5,.5) -- (.5,-.5);
\draw [thick] (.5,-.5) -- (-.5,-.5);
\node () at (0,0) {$\rotatebox{90}{Z}$};
\draw [thick] (-.4,.5) arc [radius = -.3, start angle = 180, end angle = 225];
\draw [thick] (.4,.5) arc [radius = -.3, start angle = 0, end angle = -45]; 
\draw [thick] (-.5,-.5-1.25) -- (-.5,.5-1.25);
\draw [thick] (-.5,.5-1.25) -- (.5,.5-1.25);
\draw [thick] (.5,.5-1.25) -- (.5,-.5-1.25);
\draw [thick] (.5,-.5-1.25) -- (-.5,-.5-1.25);
\node () at (0,-1.25) {$\rotatebox{90}{Z}$};
\draw [thick] (-.4,-.5) -- (-.4,-.75);
\draw [thick] (.4,-.5) -- (.4,-.75);
\draw [thick] (-.4,-.5-1.25) arc [radius = .3, start angle = 0, end angle = -45];
\draw [thick] (.4,-.5-1.25) arc [radius = .3, start angle = 180, end angle = 225];
 \end{tikzpicture} = \frac{[2]_q}{\sqrt{2}}\rhoZDiagram, \quad
(8) \quad \begin{tikzpicture}[scale = 1,baseline={([yshift=-.5ex]current bounding box.center)}]
\draw [thick] (-.5,-.5) -- (-.5,.5);
\draw [thick] (-.5,.5) -- (.5,.5);
\draw [thick] (.5,.5) -- (.5,-.5);
\draw [thick] (.5,-.5) -- (-.5,-.5);
\node () at (0,0) {$Z$};
\draw [thick] (-.4,.5) arc [radius = -.3, start angle = 180, end angle = 225];
\draw [thick] (.4,.5) arc [radius = -.3, start angle = 0, end angle = -45]; 
\draw [thick] (-.5,-.5-1.25) -- (-.5,.5-1.25);
\draw [thick] (-.5,.5-1.25) -- (.5,.5-1.25);
\draw [thick] (.5,.5-1.25) -- (.5,-.5-1.25);
\draw [thick] (.5,-.5-1.25) -- (-.5,-.5-1.25);
\node () at (0,-1.25) {$\rotatebox{90}{Z}$};
\draw [thick] (-.4,-.5) -- (-.4,-.75);
\draw [thick] (.4,-.5) -- (.4,-.75);
\draw [thick] (-.4,-.5-1.25) arc [radius = .3, start angle = 0, end angle = -45];
\draw [thick] (.4,-.5-1.25) arc [radius = .3, start angle = 180, end angle = 225];
 \end{tikzpicture} = \frac{1}{2} \CupCap, \\[.5cm]&
(9)\quad  \begin{tikzpicture}[baseline={([yshift=-.5ex]current bounding box.center)}]
 \node [thick, rectangle, draw, minimum size = 29] at (0,0) (Z1) {$Z$};
 \node [thick, rectangle, draw, minimum size = 29]  at (1.5,0) (Z2) {$Z$};
\draw [thick] (.16666,.5) arc [radius = .5833,start angle = 180, end angle = 0];
\draw [thick] (-.16666,.5) -- (.16666,1.5);
\draw [thick] (1.666666,.5) -- (1.34,1.5);
\draw [thick] (.1666,-.5) -- (.16666,-.8);
\draw [thick] (-.1666,-.5) -- (-.16666,-.8);
\draw [thick] (1.666,-.5) -- (1.6666,-.8);
\draw [thick] (1.34,-.5) -- (1.34,-.8);
 \end{tikzpicture} = \begin{tikzpicture}[baseline={([yshift=-.5ex]current bounding box.center)}]
 \node [thick, rectangle, draw, minimum size = 29] at (1,0) (Z1) {$Z$};
  \node [thick, rectangle, draw, minimum size = 29] at (0,-2) (Z1) {$Z$};
  \draw  [thick] (0.1666,-2.5) -- (0.1666,-2.8);
    \draw  [thick] (-0.1666,-2.5) -- (-0.16663,-2.8);
      \draw  [thick] (1-1.1666,-1.5) -- (1-.1666,-.5);
      \draw  [thick] (1+.1666,-.5) -- (1+.1666,-2.8);
   \draw [thick] (.83333,-1.5) arc [radius = .333,start angle = 0, end angle = 180];
   \draw [thick] (1-.16666,-1.5) -- (1-.16666,-2.8);
  \draw [thick] (1.1666,.5) -- (1.166666,.8);
   \draw [thick] (1-.1666,.5) -- (1-.16666,.8); 
 \end{tikzpicture}  = \begin{tikzpicture}[baseline={([yshift=-.5ex]current bounding box.center)}]
 \node [thick, rectangle, draw, minimum size = 29] at (0,0) (Z1) {$Z$};
  \node [thick, rectangle, draw, minimum size = 29] at (1,-2) (Z1) {$Z$};
  \draw  [thick] (1.1666,-2.5) -- (1.1666,-2.8);
    \draw  [thick] (.83333,-2.5) -- (.83333,-2.8);
      \draw  [thick] (1.1666,-1.5) -- (.1666,-.5);
      \draw  [thick] (-.1666,-.5) -- (-.1666,-2.8);
   \draw [thick] (.83333,-1.5) arc [radius = .333,start angle = 0, end angle = 180];
   \draw [thick] (.16666,-1.5) -- (.16666,-2.8);
   \draw [thick] (.1666,.5) -- (.166666,.8);
   \draw [thick] (-.1666,.5) -- (-.16666,.8);
 \end{tikzpicture}, \quad  
\\[.5cm]& 
 (10) \quad \begin{tikzpicture}[baseline={([yshift=-.5ex]current bounding box.center)}]
 \node [thick, rectangle, draw, minimum size = 29] at (0,0) (S1) {$S$};
 \node [thick, rectangle, draw, minimum size = 29] at (0,-1.5) (S2) {$S$};
\draw [thick] (-.3,.5) -- (-.3,.8);
\draw [thick] (.3,.5) -- (.3,.8);
\draw [dotted,thick] (-.2,.65) -- (.2,.65);
\draw [dotted,thick] (-.2,.65) -- (.2,.65);
\draw [thick] (-.3,-2) -- (-.3,-2.3);
\draw [thick] (.3,-2) -- (.3,-2.3);
\draw [dotted,thick] (-.2,-2.15) -- (.2,-2.15);
\draw [dotted,thick] (-.2,-2.15) -- (.2,-2.15);
\end{tikzpicture} = \begin{tikzpicture}[baseline={([yshift=-.5ex]current bounding box.center)}]
 \node [thick, rectangle, draw, minimum size = 29] at (0,0) (S1) {$\color{blue}f^{(10)}$ $\boxtimes$ $\color{red}f^{(2)}$};
\draw [thick] (-.9,.5) -- (-.9,.8);
\draw [thick] (.9,.5) -- (.9,.8);
\draw [thick] (-.9,-.5) -- (-.9,-.8);
\draw [thick] (.9,-.5) -- (.9,-.8);
\draw [dotted,thick] (.8,.65) -- (-.8,.65);
\draw [dotted,thick] (.8,-.65) -- (-.8,-.65);
\end{tikzpicture},
\quad (11) \quad \begin{tikzpicture}[baseline={([yshift=-.5ex]current bounding box.center)}]
 \node [thick, rectangle, draw, minimum size = 29] at (0,0) (S1) {$S$};
\draw [thick] (-.3,.5) -- (-.3,.8);
\draw [thick] (.3,.5) -- (.3,.8);
\draw [dotted,thick] (-.2,.65) -- (.2,.65);
\draw [dotted,thick] (-.2,.65) -- (.2,.65);
\draw [thick] circle [radius = .86];
\draw [thick] circle [radius = 1.86];
\path[decorate,decoration={text along path,
text={Any diagram with no $S$'s},text align=center}]
(-1.36,0) arc [start angle=-180,end angle=0,radius=1.4];
\end{tikzpicture} =0, \\[2ex]
 &(12) \quad \Braid = q^{-1}\Id + q\CupCap -q^2 \zDiagram[.8] - q^{-2}\rhoZDiagram[.8],  \\[.5cm]
 & (13) \quad \begin{tikzpicture}[baseline={([yshift=-.5ex]current bounding box.center)}]
 \node [thick, rectangle, draw, minimum size = 29] at (0,0) (S1) {$\color{blue}f^{(11)}$ $\boxtimes$ $\color{red}f^{(3)}$};
\draw [thick] (-.9,.5) -- (-.9,.8);
\draw [thick] (.9,.5) -- (.9,.8);
\draw [thick] (-.9,-.5) -- (-.9,-.8);
\draw [thick] (.9,-.5) -- (.9,-.8);
\draw [dotted,thick] (.8,.65) -- (-.8,.65);
\draw [dotted,thick] (.8,-.65) -- (-.8,-.65);
\end{tikzpicture} = 0.
 \end{align*}
 \end{prop}
\begin{proof}
As mentioned this proof is almost identical to the $Z(\ad(A_{2N+1}))$ case. In fact there is a unique copy of $\operatorname{Rep}(\mathbb{Z}/2\mathbb{Z})$ in $\color{blue} A_{11}$ $\boxtimes$ $\color{red} A_{3}^\text{bop}$ which causes this case to be simpler.
\end{proof}
\begin{rmk}
Note that we haven't explicitly given descriptions of $\color{blue}f^{(2N)}$ $\boxtimes$ $\color{red}f^{(2N)}$ and $\color{blue}f^{(10)}$ $\boxtimes$ $\color{red}f^{(2)}$ in terms of the generator $Z$. This is because these descriptions are particularly nasty to write down. We just appeal to the existence of such a description from \cite[Theorem 4]{MR1815260}. If the reader requires a better description one can be obtained by applying the Jones-Wenzl recursion formula to both idempotents and rewriting the resulting recursive formula in terms of $Z$'s and $\rho(Z)$'s.
\end{rmk}

\section{Planar algebra auto-morphisms and auto-equivalences of tensor categories}\label{sec:autos}
\subsection{Planar Algebra Automorphisms}
As mentioned earlier in this paper we wish to compute the auto-equivalence group of a fusion category by studying the associated planar algebra. In \cite{1607.06041} the authors show an equivalence between the categories of (braided) planar algebras and the category of (braided) pivotal categories with specified generating object. As described in that paper the latter category is really a 2-category, with 1-morphisms being pivotal functors that fix the specified generating object up to isomorphism. The 2-morphisms are natural transformations such that the component on the generating object is the identity of the generating object. It is shown that there is at most one 2-morphism between any pair of 1-morphisms, and further that it must be an isomorphism. Thus there is no loss of generality in truncating to a 1-category. See \cite[Definition 3.4]{1607.06041} for details. We specialise their result as Proposition~\ref{prop:PA_to_FC}.

The (braided) fusion categories that we wish to compute the auto-equivalence groups for satisfy the necessary conditions to apply the above equivalence of categories (admit pivotal structures and generated by a specified object). However the restriction that the component of a natural transformation is the identity turns out to be too strong. Thus planar algebra homomorphisms which are non-equal can be mapped to isomorphic functors of the associated pivotal tensor categories. To fix this we have to work out explicitly which planar algebra automorphisms are mapped to isomorphic auto-equivalences. In an appendix to this paper we define planar algebra natural transformations in such a way so the group of planar algebra automorphisms is isomorphic to the group of auto-equivalences of the associated based pivotal category.

The aim of this paper is to compute the Brauer-Picard groups of the $ADE$ fusion categories via braided auto-equivalences of their centres. The above theorem suggests that we should begin by studying the braided automorphisms of the planar algebras associated to the centres. To remind the reader these are $A_N$, $\ad(A_{2N})$, $\ad(D_{2N})$, $Z(\ad(A_{2N+1}))$, and $Z(\ad(E_6))$. Note that many of the centres appear as Deligne products of these categories e.g. $Z(E_6) \simeq A_{11}\boxtimes A_3^\text{op}$. We deal with such products in Section~\ref{sec:BrPic}.


Roughly speaking a planar algebra automorphism is a collection of vector space automorphisms, one for each box space, that commute with the action of planar tangles. More details on planar algebra automorphisms can be found in \cite{math.QA/9909027}. As we have described the planar algebras we are interested in terms of generators and relations it is enough to determine how the automorphisms act on the generators, as any element of the planar algebra will be a sum of diagrams of generators connected by planar tangle. This makes the $A_N$ case particularly easy as this planar algebra has no generators.

\begin{lemma}\label{lem:pa_Autos_AN}
The planar algebras for $A_N$ have no non-trivial automorphisms.
\end{lemma}
\begin{proof}
Any diagram in the $A_N$ planar algebra is entirely planar tangle. As automorphisms of planar algebras have to commute with planar tangles, only the identity automorphism can exist.
\end{proof}

Unfortunately the other cases are not as easy.
\begin{lemma}\label{lem:pa_Autos_Ad_A2N}
 There are two braided planar algebra automorphisms of $\ad(A_{N})$. When lifted to the braided fusion category $\ad(A_{N})$ the corresponding auto-equivalences are naturally isomorphic.
 \end{lemma}
 
 \begin{proof}
 Let $\phi$ be a planar algebra automorphism of $\ad(A_{N})$. As the $\ad(A_{N})$ planar algebra has the single generator $T \in P_{3}$ it is enough to consider how $\phi$ behaves on $T$. As $P_{3}$ is one dimensional it follows that $\phi(T) = \alpha T$ for some $\alpha \in \mathbb{C}$. As $\phi$ must fix the single strand we can apply $\phi$ to relation (5) to show that $\alpha \in \{1,-1\}$. It can be verified that $\phi(T) = -T$ is consistent with the other relations and hence determines a valid planar algebra automorphism.
 
 We now claim that the auto-equivalence of the $\ad(A_{N})$ fusion category generated by this planar algebra automorphism is naturally isomorphic to the identity. Let $(n_1,p_1)$ and $(n_2,p_2)$ be objects of the $\ad(A_{N})$ fusion category, where $n_i \in \mathbb{N}$ and $p_i$ are projections in $P_{n_i}$. We define the components of our natural isomorphism to be $\tau_{(n,p)} := (-1)^n$. We need to verify that the following diagram commutes for any morphism $f: (n_1,p_1) \to (n_2,p_2) $ :
 
$$\begin{CD}
(n_1,p_1)@>\phi(f)>> (n_2,p_2) \\
@VV{(-1)^{n_1}}V @VV(-1)^{n_2}V\\
(n_1,p_1)@>f>> (n_2,p_2)
\end{CD}$$

Recall that $f \in P_{n_1+n_2}$. First suppose that $n_1+n_2$ is even, then the number of $T$ generators in $f$ must be even, thus $\phi(f) = f$. However in this case we also have that $(-1)^{n_1} = (-1)^{n_2}$ and so the above diagram commutes. If $n_1+n_2$ is odd, then the number of $T$ generators in $f$ must be odd, thus $\phi(f) = -f$. However in this case we also have that $(-1)^{n_1} = -(-1)^{n_2}$ and so the above diagram commutes.
 \end{proof} 
 
  \begin{lemma}\label{lem:PA_Autos_of_1/2D_{2n}}
 There are four braided planar algebra automorphisms of $\ad(D_{2N})$. When lifted to the braided fusion category $\ad(D_{2N})$ two of the auto-equivalences are naturally isomorphic to the other two. Furthermore the non-trivial automorphism lifts to a non-gauge braided auto-equivalence of the category $\ad(D_{2N})$.
 \end{lemma}
 \begin{proof}
 Let $\phi$ be a planar algebra automorphism of $\ad(D_{2N})$. Recall the planar algebra $\ad(D_{2N})$ has two generators $T \in P_3$ and $S \in P_{2N-2}$. As $P_3$ is one dimensional it follows that $\phi(T) = \alpha T$ for some $\alpha \in \mathbb{C}$. Applying relation (7) shows us that $\alpha \in \{1,-1\}$.
 
 The vector space $P_{2N-2}$ can be written $TL_{2N-2} \oplus \mathbb{C}S$, thus $\phi(S) = f + \beta S$ where $f \in TL_{2N-2}$ and $\beta \in \mathbb{C}$. We can use relation (10) to show that $\phi(S)$ must be uncappable, and hence $f$ must be 0 as there are no non-trivial uncappable Temperley-Lieb elements. Further relation (10) now shows that $\beta \in \{1,-1\}$. It can be verified that any combination of choices of $\alpha$ and $\beta$ gives a valid planar algebra automorphism. This gives us four automorphisms.
 
 However as monoidal auto-equivalences, the automorphisms which send $T \mapsto -T$ are naturally isomorphic to the corresponding automorphisms which leaves $T$ fixed. The proof of this is identical to the argument used in the previous proof. Thus we get that as a fusion category $\ad(D_{2N})$ has a single non-trivial auto-equivalence. We can see that this non-trivial auto-equivalence is not a gauge auto-equivalence as it exchanges the simple objects $P$ and $Q$.
 \end{proof}

   \begin{lemma}\label{lem:pa_Autos_Centre_Ad_A2N+1}
 There are two braided automorphisms of the $Z(\ad(A_{2N+1}))$ planar algebra. The non-trivial automorphism lifts to a non-gauge braided auto-equivalence of $Z(\ad(A_{2N+1}))$.
  \end{lemma}
 \begin{proof}
 Let $\phi$ be a braided planar algebra automorphism of $Z(\ad(A_{2N+1}))$. As a planar algebra $Z(\ad(A_{2N+1}))$ has two generators $Z \in P_4$ and $S \in P_{2N}$. Except in the special case when $N=2$ the box space $P_4$ is $4$-dimensional, spanned by $\left\{ \ \Id[.707] \ , \ \CupCap[.707] \ , \ \zDiagram[.707] \ , \ \rhoZDiagram[.707] \ \right\}$. Thus we can write $\phi\left(\zDiagram[.707]\right)$ as a linear combination of these basis elements. As $\phi$ is a braided automorphism by definition it preserves the braiding, that is 
 \begin{equation*}
\Id[.707]+ \CupCap[.707] + q \ \zDiagram[.707] + q^{-1} \ \rhoZDiagram[.707] = \Id[.707]+ \CupCap[.707] +  q \ \phi\left(\zDiagram[.707]\right) + q^{-1} \ \phi\left(\rhoZDiagram[.707]\right).
 \end{equation*}
 Solving this equation gives the unique solution $\phi\left(\zDiagram[.707]\right) = \zDiagram[.707]$.
 
 For the $N=2$ case we repeat the same proof but with $P_4$ also having $S$ as a basis vector.
 
 Now we have to determine where $\phi$ sends the generator $S\in P_{2N}$. Recall that relation (11) shows that any diagram with a single $S$ is zero. We claim that $S$ is the only vector (up to scalar) in $P_{2N}$ with this property. Let $v\in P_{N}$ be such a vector, as $P_{2N} \cong TL_{2N}\boxtimes TL_{2N}\oplus \mathbb{C}S$ we can write $v = f_1 \boxtimes f_2 + \alpha S$ where $f_1$ and $f_2$ are elements of $TL_{2N}$. As $\alpha S$ and $v$ are uncappable, it follows that $f_1 \boxtimes f_2$ must also be. In particular this implies that both $f_1$ and $f_2$ are uncappable Temperley-Lieb diagrams, and so must be $0$. Thus $v = \alpha S$. 

Relation (10) now implies that $\phi(S) = \pm S$ as we know $\phi$ fixes $Z$ and hence the right hand side of the relation. It can be verified that $S \mapsto -S$ is consistent with the other relations. The automorphism $\phi(S) = -S$ lifts to a non-gauge braided auto-equivalence of the associated category as it exchanges the simple objects $\frac{f^{(N)}+S}{2}$ and $\frac{f^{(N)}-S}{2}$.
%
%
 \end{proof}

     \begin{lemma}\label{lem:PA_Autos_of_ZE6}
 There are two braided automorphisms of the $Z(\ad(E_6))$ planar algebra. The non-trivial automorphism lifts to a non-gauge braided auto-equivalence of $Z(\ad({E_6}))$.
 \end{lemma}
 \begin{proof}
 Almost identical to the proof of Lemma~\ref{lem:pa_Autos_Centre_Ad_A2N+1}.
 \end{proof}

 \subsection{Braided auto-equivalences}
The aim of this subsection is to leverage our knowledge of planar algebra automorphisms to compute the braided auto-equivalence group of the associated braided category. As mentioned in the introduction, for our examples, planar algebra automorphisms contain the gauge auto-equivalences of the associated braided category. The results of the previous section show that there are in fact no non-trivial gauge auto-equivalences for any of the categories we are interested in! This means that the group of braided auto-equivalences is a subgroup of the automorphism group of the fusion ring.

Our proofs to compute braided auto-equivalence group of $C$ is as follows. First we compute the fusion ring automorphisms of $C$. Then we analyse the $t$-values of the simple objects of $C$ to rule out fusion ring automorphisms that can't lift to braided auto-equivalences of $C$. Finally we construct braided auto-equivalences of $C$ to realise the remaining fusion ring automorphisms.
 
 \begin{lemma}\label{lem:brAutadAN}
 We have $\BrAut(\ad(A_{2N})) = \{e\}$.
 \end{lemma}
 \begin{proof}
 The fusion ring of $\ad(A_{2N})$ has no non-trivial automorphisms. Thus in light of Theorem~\ref{prop:PA_to_FC} and Lemma~\ref{lem:pa_Autos_Ad_A2N} there are no monoidal auto-equivalences of the category $\ad(A_{2N})$, and in particular no braided auto-equivalences.
 \end{proof}

 \begin{lemma}\label{lem:aut_D10}
We have \[\BrAut(\ad(D_{2N})) = \begin{cases} 
      \mathbb{Z} / 2\mathbb{Z}& \text{when } N   \neq 5\\
        S_3 & \text{when } N = 5.
   \end{cases}
\]
 \end{lemma}
 
 \begin{proof}
 We break this proof up into three parts: $N=2$, $N=5$ and all other $N$.

\medskip

\textbf{ Case $N=2$:}

The category $\ad(D_4)$ has fusion ring isomorphic to that of $\operatorname{Vec}(\mathbb{Z}/3\mathbb{Z})$. It is straightforward to see that the only possible fusion ring automorphism exchanges the two non-trivial objects. Lemma~\ref{lem:PA_Autos_of_1/2D_{2n}} tells us two pieces of information. First is that there are no gauge auto-equivalences of $\ad(D_4)$, and thus there are at most two braided auto-equivalences of $\ad(D_4)$. Second is that there exists a non-trivial braided auto-equivalence of $\ad(D_4)$, which realises the upper bound on the braided auto-equivalence group.
\medskip

\textbf{Case $N=5$:}
Studying the $\ad(D_{10})$ fusion ring we see that there are 6 possible automorphisms, corresponding to any permutation of the objects $f^{(2)}$, $P$, and $Q$. From Lemma~\ref{lem:PA_Autos_of_1/2D_{2n}} we see that there are no non-trivial gauge auto-equivalences, thus $\BrAut(\ad(D_{2N}))$ is a subgroup of $S_3$.

%
%
%
%
To show that $\BrAut(\ad(D_{10})) = S_3$ we notice from Lemma~\ref{lem:PA_Autos_of_1/2D_{2n}} that $\ad(D_{10})$ has an order 2 braided auto-equivalence that exchanges the objects $P$ and $Q$, and from \cite[Theorem 4.3]{MR2783128} $\ad(D_{10})$ has an order 3 braided auto-equivalence. Therefore Lagrange's theorem implies that the order of $\BrAut(\ad(D_{10}))$ is at least $6$ and the result follows.

\medskip

 \textbf{Case $N>2, N\neq 5$:}
 For these cases, the only simple objects with the same dimension are $P$ and $Q$. Thus there can be at most two fusion ring automorphisms of the $\ad(D_{2N})$ fusion ring. The result now follows by the same argument as in the $N=2$ case.
\end{proof}

\begin{lemma}\label{lem:brAutAN}
We have \[\BrAut(A_N) =\begin{cases} 
      \{e\} & N   \equiv \{1,2,4\} \pmod 4\\
        \mathbb{Z} / 2\mathbb{Z} & N  \equiv 3\pmod 4 .
   \end{cases}
\]

\end{lemma}
\begin{proof}
When $N$ is odd, the $A_N$ fusion ring has the non-trivial automorphism:

\[f^{(n)} \mapsto 
\begin{cases} 
      f^{(n)}  & n \text{ is even }\\
        f^{(N-n-1)}  & n \text{ is odd. }
        \end{cases} \]

\textbf{Case $N   \equiv \{0,2\} \pmod 4$}: 
When $N$ is even there are no non-trivial automorphisms of the $A_N$ fusion ring. Thus the result follows as there are no non-trivial gauge auto-equivalences of the category $A_{N}$.
\medskip

\textbf{Case $N  \equiv 1 \pmod 4 $:}

In this case there is a fusion ring automorphism exchanging the generating object $f^{(1)}$ and $f^{(N-2)}$, however these two objects have different $t$-values so there is no braided auto-equivalence realising the fusion ring automorphism. Thus the result follows as there are no non-trivial gauge auto-equivalences of the category $A_{N}$.
\medskip

\textbf{Case $N  \equiv 3 \pmod 4 $:}

We first give a bound on the size of the auto-equivalence group of $A_N$. On the level of fusion rings there are two automorphisms, with the non-trivial one exchanging $f^{(1)}$ and $f^{(N-2)}$. As there are no non-trivial gauge auto-equivalences of the category $A_{N}$, we have that $\BrAut(A_N) \subseteq  \mathbb{Z} / 2\mathbb{Z}$.

To show existence of the non-trivial braided auto-equivalence it suffices to show that the category generated by $f^{(N-2)}$ is equivalent to $A_N$ as a braided tensor category. It is known that the $T$-matrix is a complete invariant of braided fusion categories with $A_N$ fusion rules \cite{MR1239440}. The braided fusion category generated by $f^{(N-2)}$ has the same $T$-matrix as $A_N$, hence they are braided equivalent.
\end{proof}

We can slightly modify the above argument to also compute the tensor auto-equivalences of $A_N$.
\begin{lemma}\label{lem:tenAutAN}
We have \[ \Aut_\otimes(A_N) = \begin{cases} 
      \{e\} & N   \equiv 0 \pmod 2\\
        \mathbb{Z} / 2\mathbb{Z} & N   \equiv 1 \pmod 2.
   \end{cases}
\]
\end{lemma}
\begin{proof}
When $N$ is even there are no non-trivial automorphisms of the $A_N$ fusion ring. Thus as there are no automorphisms of the $A_N$ planar algebra, we have $\Aut_\otimes(A_N) = \{e\}$.

When $N$ is odd the same argument as we used in Lemma~\ref{lem:brAutAN} shows that $\Aut(A_N) \subseteq \mathbb{Z} / 2 \mathbb{Z}$. To realise the non-trivial tensor auto-equivalence we need to show that the tensor category generated by $f^{(N-2)}$ is equivalent to $A_N$. Recall that the categorical dimension of the generating object of $A_N$ is a complete invariant of these tensor categories \cite{MR1239440}. The result then follows as $f^{(1)}$ and $f^{(N-2)}$ have the same categorical dimension.
\end{proof}
While the $\ad(A_{2N+1})$ categories do not appear as factors of any of the centres we are studying in this paper, we can show the existence of an exceptional monoidal auto-equivalence of $\ad(A_7)$. This auto-equivalence will be useful when trying later when trying to construct invertible bimodules over $\ad(A_7)$.
\begin{lemma}\label{lem:adA7auto}
We have $\Aut_\otimes(\ad(A_7)) = \mathbb{Z} /2 \mathbb{Z}$.
\end{lemma}
\begin{proof}
From the previous subsection we know that there are no gauge auto-equivalences of $\ad(A_7)$. A quick analysis of the fusion ring of $\ad(A_7)$ reveals a single non-trivial automorphism, exchanging $f^{(2)}$ and $f^{(4)}$.

Consider the planar algebras 
\[   \text{PA}( \ad(A_7); f^{(2)}) \qquad \text{and} \qquad \text{PA}( \ad(A_7); f^{(4)}) .\]
Both of these planar algebras contain sub-planar algebras generated by the trivalent vertex. By considering the fusion rules for $\ad(A_7)$ we can see that both these sub-planar algebras have box space dimensions $(1,0,1,1,3,....)$, thus by the main theorem of \cite{1501.06869} must be $\text{SO}(3)_q$ for $q$ either $e^{\frac{i\pi}{4}}$ or $e^{\frac{3i\pi}{4}}$. As the categorical dimension of both $f^{(2)}$ and $f^{(4)}$ in $\ad(A_7)$ is $1 + \sqrt{2}$, we must have that both sub-planar algebras are $\text{SO}(3)_{e^{\frac{i\pi}{4}}}$.

By again considering fusion rules for $\ad(A_7)$ we can see that $\text{PA}( \ad(A_7); f^{(4)}) = \text{SO}(3)_{e^{\frac{i\pi}{4}}} = \text{PA}( \ad(A_7); f^{(2)}) $. Thus there is a planar algebra isomorphism
\[ \text{PA}( \ad(A_7); f^{(2)}) \to \text{PA}( \ad(A_7); f^{(4)}),\]
which by Proposition~\ref{prop:PA_to_FC} gives us a monoidal equivalence between based categories $(\ad(A_7), f^{(2)})$ and $(\ad(A_7), f^{(4)})$. Forgetting the basing realises the non-trivial monoidal auto-equivalence of $\ad(A_7)$.
\end{proof}

\begin{lemma}
 We have $\BrAut(Z(\ad(E_6))) = \mathbb{Z} / 2\mathbb{Z}$.
\end{lemma}
\begin{proof}
Recall that $Z(\ad(E_6))$ is a de-equivariantization of a sub-category of  $A_{11} \boxtimes A_3^\text{bop}$. Thus we can pick representatives of the simple objects of $Z(\ad(E_6))$ as in Table~\ref{tab:1/2E6dims}. We include the Frobenius-Perron dimensions and twists of these simple objects (from formulas in \cite{AN-Survey,MR2640343}).

\begin{table}[!h]

    \centering
    
    \begin{tabular}{c|cc}
    	\toprule
			$X$   &   FPdim($X$)  & $ t_X  $     \\ 
	\midrule
			      $f^{(0)} \boxtimes  f^{(0)}$  &  $1$  & $1$    \\ 
            	              $f^{(2)} \boxtimes  f^{(0)}$  &  $1 + \sqrt{3}$ & $e^\frac{2\pi i}{6}$    \\ 
	              $f^{(4)} \boxtimes  f^{(0)}$  &  $2 + \sqrt{3}$   &  $-1$  \\ 
	              $f^{(6)} \boxtimes  f^{(0)}$  &  $2 + \sqrt{3}$  & $1$    \\ 
	              $f^{(8)} \boxtimes  f^{(0)}$  &  $1 + \sqrt{3}$ & $e^\frac{8\pi i}{6}$      \\ 
	              $f^{(10)} \boxtimes  f^{(0)}$  &  $1$ & $-1$      \\ 
	              $f^{(1)} \boxtimes  f^{(1)}$  &  $1 + \sqrt{3}$  & $-i$   \\ 
	              $f^{(3)} \boxtimes  f^{(1)}$  &  $3+\sqrt{3}$  &$1$     \\ 
	              $\frac{1}{2}(f^{(5)} \boxtimes  f^{(1)}+S)$  &  $1 + \sqrt{3}$  &$e^\frac{5\pi i}{6}$    \\ 
	             $\frac{1}{2}(f^{(5)} \boxtimes  f^{(1)}-S)$  &  $1 + \sqrt{3}$  &$e^\frac{5\pi i}{6}$    \\ 
    	\bottomrule
    \end{tabular}
\caption{\label{tab:1/2E6dims} Dimensions and $t$-values for the simple objects of $Z(\ad(E_6))$}
    \end{table}
By considering dimensions and twists, we can see that there is only one possible non-trivial fusion ring automorphism of $Z(\ad(E_6))$, that exchanges the objects $\frac{1}{2}(f^{(5)} \boxtimes  f^{(1)}+S)$ and $\frac{1}{2}(f^{(5)} \boxtimes  f^{(1)}-S)$. Thus, as there are no non-trivial gauge auto-equivalences of the category $Z(\ad(E_6))$, we must have that $\BrAut(Z(\ad(E_6))) \subseteq \mathbb{Z} / 2\mathbb{Z}$. A non-trivial braided auto-equivalence of $Z(\ad(E_6))$ is constructed in Lemma \ref{lem:PA_Autos_of_ZE6}.
\end {proof}
\begin{rmk}
The dimensions and $T$ matrix for the centre for $Z(\ad(E_6))$ have also been computed in \cite{MR2468378}.
\end{rmk}

We end this section with the most difficult case, the centre of $\ad(A_{2N+1})$. Unfortunately we are unable to explicitly construct the braided auto-equivalences for many cases. Instead we are forced to find invertible bimodules over $\ad(A_{2N+1})$ and appeal to the isomorphism from the invertible bimodules to the braided auto-equivalences of the centre.
\begin{lemma}\label{lem:centre_AN_autos}
 We have \[ \BrAut(Z(\ad(A_{2N+1}))) =  \begin{cases} 
      \mathbb{Z} / 2\mathbb{Z}& \text{when } N   \equiv 0 \pmod 2 \text{ or } N = 1\\
        (\mathbb{Z} / 2\mathbb{Z})^2 & \text{when } N  \equiv 1 \pmod 2 \text{ and } N \neq \{1,3\} \\
        D_{2\cdot 4} & \text{when } N = 3.
   \end{cases}
\]
\end{lemma}
\begin{proof}
As in the $Z(\ad(E_6))$ case we choose representatives for the simple objects of $Z(\ad(A_{2N+1}))$. These representatives are:
\begin{equation*}
\left\{ f^{(n)} \boxtimes f^{(m)} \mid n < 2N -m \text{ and } n-m  \equiv 0\pmod 2\right\} \bigcup \left\{ f^{(2N-n)} \boxtimes f^{(n)}\mid 0 \leq n < N\right\}  \bigcup \left\{\frac{ f^{(N)}\boxtimes f^{(N)} \pm S}{2} \right\}.
\end{equation*}
Thus the dimensions of the simple objects are among the set 
\begin{equation*}
\left\{[n+1][m+1] \mid n < 2N-m \text{ and } n-m  \equiv 0 \pmod 2\right\} \bigcup \left\{ [2N-n+1][n+1] \mid 0 \leq n < N \right\}  \bigcup \left\{ \frac{[N+1]^2}{2}\right\}. 
\end{equation*}
When $N\neq \{1,3\}$ the automorphism group of the fusion ring is $(\mathbb{Z}/2\mathbb{Z})^3$, generated by the three automorphisms
\begin{itemize}
\item $f^{(n)}\boxtimes f^{(m)} \leftrightarrow  f^{(m)}\boxtimes f^{(n)}$ for all $n,m$ even,
\item $f^{(n)}\boxtimes f^{(m)} \leftrightarrow  f^{(2N-n)}\boxtimes f^{(m)}$ for all $n,m$ odd,
\item $\frac{ f^{(N)}\boxtimes f^{(N)} + S}{2} \leftrightarrow  \frac{ f^{(N)}\boxtimes f^{(N)} - S}{2}$.
\end{itemize}

When $N = 1$ the automorphism group of the fusion ring is $S_3$, generated by the two automorphisms
\begin{itemize}
\item $\frac{ f^{(N)}\boxtimes f^{(N)} + S}{2} \leftrightarrow  \frac{ f^{(N)}\boxtimes f^{(N)} - S}{2}$,
\item $f^{(2)} \boxtimes f^{(0)} \rightarrow \frac{ f^{(N)}\boxtimes f^{(N)} + S}{2} \rightarrow \frac{ f^{(N)}\boxtimes f^{(N)} - S}{2} \rightarrow f^{(2)} \boxtimes f^{(0)}$.
\end{itemize}

When $N = 3$ the automorphism group of the fusion ring has order 16, and is generated by the three automorphisms
\begin{itemize}
\item $\frac{ f^{(3)}\boxtimes f^{(3)} + S}{2}   \leftrightarrow \frac{ f^{(3)}\boxtimes f^{(3)} - S}{2}$,

\item $f^{(1)} \boxtimes f^{(1)} \rightarrow \frac{ f^{(3)}\boxtimes f^{(3)} + S}{2} \rightarrow f^{(5)} \boxtimes f^{(1)}  \rightarrow \frac{ f^{(3)}\boxtimes f^{(3)} - S}{2} \rightarrow f^{(1)} \boxtimes f^{(1)}$, $f^{(2)} \boxtimes f^{(0)} \leftrightarrow f^{(0)} \boxtimes f^{(4)}$, and $f^{(4)} \boxtimes f^{(0)} \leftrightarrow f^{(0)} \boxtimes f^{(2)}$,

\item $f^{(2)}\boxtimes f^{(0)} \leftrightarrow f^{(0)}\boxtimes f^{(2)}$, $f^{(4)}\boxtimes f^{(0)} \leftrightarrow f^{(0)}\boxtimes f^{(4)}$.

\end{itemize}
It can be easily verified that the first two of these automorphisms satisfy the relations of the usual $s$ and $r$ generators of $D_{2\cdot 4}$.

From Lemma~\ref{lem:pa_Autos_Centre_Ad_A2N+1} we know there are no non-trivial braided Gauge auto-equivalences of $Z(\ad(A_{2N+1}))$. Thus these groups are upper bounds for $\BrAut(Z(\ad(A_{2N+1})))$.

\medskip
\textbf{Case $N   \equiv 0 \pmod 2$:}
When $N  \equiv 0 \pmod 2 $ the $t$ values for $f^{(1)}\boxtimes f^{(1)}$ and  $f^{(2N-1)}\boxtimes f^{(1)}$ are different, and the $t$ values for $f^{(2)}\boxtimes f^{(0)}$ and $f^{(0)}\boxtimes f^{(2)}$ are different. Thus the only possible braided auto-equivalence is $\frac{ f^{(N)}\boxtimes f^{(N)} + S}{2} \leftrightarrow  \frac{ f^{(N)}\boxtimes f^{(N)} - S}{2}$, which is constructed in Lemma~\ref{lem:pa_Autos_Centre_Ad_A2N+1}. Thus $\BrAut(Z(\ad(A_{2N+1})))= \mathbb{Z} / 2\mathbb{Z}$.

\medskip

\textbf{Case $N  \equiv 1 \pmod 2 $ and $N \neq \{1,3\}$:}
When $N  \equiv 1 \pmod 2 $ and $N \neq \{1\}$ the $t$-values for $f^{(2)}\boxtimes f^{(0)}$ and $f^{(0)}\boxtimes f^{(2)}$ are different. Thus we have an upper bound of $(\mathbb{Z}/2\mathbb{Z})^2$ on $\BrAut(Z(\ad(A_{2N+1})))$.

To complete the proof we need to construct four braided auto-equivalences. Instead we show the existence of four invertible bimodules over $\ad(A_{2N+1})$. There is the trivial bimodule, of rank $N+1$, which gives us one. The odd graded piece of the $\mathbb{Z}/2\mathbb{Z}$-graded fusion category $A_{2N+1}$ is an invertible bimodule. This bimodule has rank $N$, and thus is not equivalent to the trivial bimodule. This gives us two invertible bimodules over $\ad(A_{2N+1})$, if we can show the existence of a third invertible bimodule then the group must be $(\mathbb{Z} / 2\mathbb{Z})^2$ via an application of Lagrange's theorem.

Consider the object $A = \mathbf{1} \oplus f^{(2N)}$ in $\ad(A_{2N+1})$. This object has a unique algebra structure in $A_{2N+1}$ \cite{MR1976459,AN-Survey}, and furthermore the category of $A$ bimodules in $A_{2N+1}$ is equivalent to $A_{2N+1}$ \cite{MR1839381}. Therefore we can apply Lemma~\ref{lem:algRes} to see that $A-\operatorname{bimod}_{\ad(A_{2N+1})} \simeq \ad(A_{2N+1})$. Thus $A-$mod is an invertible bimodule over $\ad(A_{2N+1})$. The rank of $A-$mod is $\frac{N+1}{2}$, and so is non-equivalent to either of the two previous invertible bimodules.
\medskip

\textbf{Case $N=1$: }

The representatives of the four simple objects of $Z(\ad(A_{3}))$ are $f^{(0)}\boxtimes f^{(0)}$, $\frac{f^{(1)}\boxtimes f^{(1)} - S}{2}$, $\frac{f^{(1)}\boxtimes f^{(1)} + S}{2}$, and $f^{(2)}\boxtimes f^{(0)}$. The $t$-values of these objects are $1$, $1$, $1$, and $-1$ respectively. We can see that the only possibility for a non-trivial braided auto-equivalence is $\frac{f^{(1)}\boxtimes f^{(1)} +S}{2} \leftrightarrow \frac{f^{(1)}\boxtimes f^{(1)} - S}{2}$. This braided auto-equivalence is constructed in Lemma~\ref{lem:pa_Autos_Centre_Ad_A2N+1}.

\medskip

\textbf{Case $N=3$:}
The $t$-values for $f^{(2)}\boxtimes f^{(0)} \leftrightarrow f^{(0)}\boxtimes f^{(2)}$ are different in $Z(\ad(A_7))$, thus we have see that $\BrAut(Z(\ad(A_7))) \subseteq  D_{2\cdot 4}$.

As in the $N  \equiv 1 \pmod 2 $ case we unfortunately have to explicitly construct $8$ invertible bimodules over $\ad(A_7)$. Using the same arguments as before we have the algebra objects $\mathbf{1}$, $\mathbf{1} \oplus f^{(2)}$, and  $\mathbf{1} \oplus f^{(6)}$ which give rise to invertible bimodules. Each of these bimodules is distinct as the ranks are 4, 3, and 2 respectively. We can twist each of these bimodules by the non-trivial outer auto-equivalence of $\ad(A_7)$ from Lemma~\ref{lem:adA7auto} to find $6$ distinct invertible bimodules over $\ad(A_7)$. As $6$ doesn't divide the order of $D_{2\cdot 4}$ it follows from Lagrange's theorem that $\BrAut(Z(\ad(A_7))) = D_{2\cdot 4}$.
\end {proof}

We also need to calculate the braided auto-equivalences of the opposite braided category for many of the examples above.  However there is a canonical isomorphism $\BrAut(C) \cong \BrAut(C^\text{bop})$ so we really don't need to worry.

\section{The Brauer-Picard Groups of the $ADE$ fusion categories}\label{sec:BrPic}
We spend this Section tying up the loose ends to complete our computations of the Brauer-Picard groups of the $ADE$ fusion categories. Our only remaining problem is to compute the braided auto-equivalence group of the centres that are products of two modular categories. I.e The centre of $E_6$ which is $A_{11}\boxtimes A_3^\text{bop}$.

Consider a product of braided tensor categories $C \boxtimes D$. Given a braided auto-equivalence of $C$ and a braided auto-equivalence of $D$ one gets a braided auto-equivalence of $C \boxtimes D$ by acting independently on each factor. This determines an injection $\BrAut(C) \times \BrAut(D) \to \BrAut(C\boxtimes D)$. This injection is an isomorphism if and only if every braided auto-equivalence of $C \boxtimes D$ restricts to braided auto-equivalences of the factors.

For most of the centres we are interested in it turns out that every braided auto-equivalence of the product restricts to auto-equivalences of the factors. Thus the results of Section~\ref{sec:autos} are sufficient to compute the Brauer-Picard group. We consider the Frobenius-Perron dimensions and $t$-values of the generating objects of the factors (in our examples the factors are always singly generated). These values are invariant under action by a braided auto-equivalence. Thus if the only other objects in the product with the same Frobenius-Perron dimension and $t$-value as the generating object also lie in the same factor then the generating object of that factor must be mapped within the factor by any braided auto-equivalence. As the generating object of the factor is mapped within the same factor it follows that the rest of the factor must be mapped within the same factor and thus the braided auto-equivalence restricts to that factor. If there exists an object outside the original factor with the same dimension and $t$-value as the generating object of that factor then we consider fusion rules to show that such a braided auto-equivalence can't exist.

The only centres where there exist braided auto-equivalences that don't restrict to the factors are some of the $Z(A_N)$'s. We use an ad-hoc method to compute the group of braided auto-equivalences for this special case.

We explicitly spell out the details of how we show that every braided auto-equivalence of $Z(E_6)$ and $Z(E_8)$ restricts to the factors.
\begin{lemma}\label{lem:ZE6_autos}
We have $\BrAut(Z(E_6))$ = $\BrAut(A_{11}) \times \BrAut(A_3^\text{bop})$.
\end{lemma}
\begin{proof}
Recall from Section~\ref{sec:centres} that $Z(E_6)=A_{11} \boxtimes A_3^\text{bop}$. Thus we have to show that every braided auto-equivalence of $A_{11} \boxtimes A_3^\text{bop}$ restricts to the factors. As described above we start by considering Frobenius-Perron dimensions and $t$-values. We compute tables for these values for all the simple objects of $A_{11} \boxtimes A_3^\text{bop}$  using the formulas in \cite{MR2640343,AN-Survey}.

\begin{table}[h!]

    \centering
    \resizebox{\linewidth}{!}{%
    \begin{tabular}{c|ccccccccccc}
    	\toprule
			FPdim     &                    &   &&&&$A_{11} $&&&&\\ 
	\midrule
			     & $1$                &  $1.932$    &   $2.732$       &   $3.346$   &    $3.732$   &  $3.864$       &  $3.732$  & $3.346$    & $2.732$   &   $1.932$   & $1$   \\ 
            	
    	$A_3^\text{bop}$ & $1.414$ & $2.732$   &  $  3.864$   &  $4.732$     &    $5.278$    &  $5.464$  &  $5.278$   &   $4.732$    & $  3.864$  &$2.732$  &$1.414$\\ 
	
	                   & $1$                &  $1.932$    &   $2.732$       &   $3.346$   &    $3.732$  &  $3.864$       &  $3.732$     & $3.346$    & $2.732$   &   $1.932$   & $1$   \\ 
    	                 
    	\bottomrule
    \end{tabular}}
    \caption{\label{tab:ZE6_Dimensions}Approximate Frobenius-Perron dimensions of the simple objects of $Z(E_6)$}
\end{table}

\begin{table}[h!]

    \centering
    \resizebox{\linewidth}{!}{%
    \begin{tabular}{c|ccccccccccc}
    	\toprule
		T	     &                    &   &&&&$A_{11} $&&&&\\ 
	\midrule
			     & $1$                &  $e^{\frac{15\pi i}{24}}$    &   $e^{\frac{8\pi i }{24}}$       &   $e^{\frac{27\pi i}{24}}$   &    $-1$                     &  $e^{\frac{47\pi i}{24}}$       &  $1$                     &  $e^{\frac{27\pi i}{24}}$    & $e^{\frac{32\pi i} {24}}$  &  $e^{\frac{15\pi i}{24}}$   & $-1$   \\ 
            	
    	$A_3^\text{bop}$ & $e^{\frac{21\pi i}{24}}$ & $ i$   &  $   e^{\frac{29\pi i}{24}}$   &  $1$     &    $e^{\frac{45\pi i}{24}}$    &  $e^{\frac{20\pi i}{24}}$  & $e^{\frac{21\pi i}{24}}$   &   $1$     & $e^{\frac{5\pi i}{24}}$  & $ i  $  & $e^{\frac{45\pi i}{24}}$\\ 
	
	                    & $-1$                &  $e^{\frac{39\pi i}{24}}$    &   $e^{\frac{32\pi i }{24}}$       &   $e^{\frac{3\pi i}{24}}$   &    $1$                     &  $e^{\frac{23\pi i}{24}}$       &  $-1$                     &  $e^{\frac{3\pi i}{24}}$    & $e^{\frac{8\pi i} {24}}$  &  $e^{\frac{39\pi i}{24}}$   & $1$   \\ 
    	                 
    	\bottomrule
    \end{tabular}}
    \caption{\label{tab:ZE6_T}$t$ values of the simple objects of $Z(E_6)$}
\end{table}

Consulting these tables we can see that the only simple object with the same pair of Frobenius-Perron dimension and $t$-value as the generating object of the $A_{11}$ subcategory is also in the $A_{11}$ subcategory. The pair for the generating object of the $A_3^\text{bop}$ subcategory is unique among all other simples. Thus the group of braided auto-equivalences of $A_{11} \boxtimes A_3^\text{bop}$ is the product of the braided auto-equivalence groups of the factors.
 \end{proof}

 \begin{lemma}\label{lem:ZE8_autos}
We have $\BrAut(Z(E_8))$ = $\BrAut(A_{29})\times \BrAut(\ad(A_4)^\text{bop})$.
 \end{lemma}
 \begin{proof}
 Recall from Section~\ref{sec:centres} that the Drinfeld centre of $E_8$ is $A_{29}\boxtimes \ad(A_4)^\text{bop}$. This proof will be similar to the $E_6$ case in that we show the generating objects of the $A_{29}$ and $\ad(A_4)^\text{bop}$ factors are mapped within the original factor, and hence the auto-equivalence group of the centre is the product of the auto-equivalence group of the two factors.

Let's first look at $f^{(1)}\boxtimes f^{(0)}$, the generating object of the $A_{29}$ factor. The only other object with the same dimension is the object $f^{(27)}\boxtimes f^{(0)}$. This other object also lives in the $A_{29}$ factor and thus any auto-equivalence of $A_{29}\boxtimes \ad(A_4)^\text{bop}$ maps the generating object of the $A_{29}$ factor within the $A_{29}$ factor, and hence restricts to an auto-equivalence of the $A_{29}$ factor.

Now we look at the generating object of the $\ad(A_4)^\text{bop}$ factor, that is $f^{(0)}\boxtimes f^{(2)}$. Considering dimensions there are two possible simple objects that $f^{(0)}\boxtimes f^{(2)}$ may be sent to. These are $f^{(0)}\boxtimes f^{(2)}$, and $f^{(28)}\boxtimes f^{(2)}$. Aiming for a contradiction, suppose there was an auto-equivalence of $A_{29} \boxtimes  \ad(A_4)^\text{bop}$ sending $f^{(0)}\boxtimes f^{(2)}$ to $f^{(28)}\boxtimes f^{(2)}$. In the category $ \ad(A_4)^\text{bop}$ there exists a non-zero morphism $f^{(2)}\otimes f^{(2)} \to f^{(2)}$. Embedding this morphism into $A_{29} \boxtimes  \ad(A_4)^\text{bop}$ and applying our auto-equivalence gives a non-zero morphism
\begin{equation*}
f^{(0)}\boxtimes  (f^{(2)}\otimes f^{(2)})  \to f^{(28)}  \boxtimes f^{(2)}.
\end{equation*}

However this implies that there exists a non-zero morphism from $f^{(0)}\to f^{(28)}$ in $A_{29}$, which is a contradiction as $f^{(0)}$ and $f^{(28)}$ are non-isomorphic simple objects. Thus by the same argument as earlier any auto-equivalence restricts to an auto-equivalence of the $\ad(A_4)^\text{bop}$ factor.

\end{proof}

The rest of the Drinfeld centres (with the exception of the centres of certain $A_N$'s) described in Section~\ref{sec:centres} which are products have braided auto-equivalence groups which decompose in a similar fashion. 
\begin{lemma}\label{lem:tensor_decomp}
We have the following:
\begin{align*}
\BrAut( Z(\ad(A_{2N}))) &=   \BrAut( \ad(A_{2N})) \times \BrAut(\ad(A_{2N})^\text{bop})  \\
\BrAut( Z(A_{2N}^\text{deg} ) ) &= \BrAut( \ad(A_{2N}^\text{deg})) \times \BrAut(\ad(A_{2N}^\text{deg})^\text{bop}) \times \BrAut( Z(\ad(A_3))) \\
\BrAut( Z(D_{2N}) &=   \BrAut( A_{4N-3}) \times \BrAut(\ad(D_{2N})^\text{bop}) \\
\BrAut(Z(\ad(D_{2N}))) &=    \BrAut( \ad(D_{2N})) \times \BrAut(\ad(D_{2N})^\text{bop})  \\
\BrAut( Z( \ad(E_8))) &=   \BrAut( \ad(D_{16})) \times \BrAut(\ad(A_4)^\text{bop}).
\end{align*}

\end{lemma}
\begin{proof}
The proof of these statements use the same techniques as in Lemmas~\ref{lem:ZE6_autos} and \ref{lem:ZE8_autos}. 
\end{proof}

Finally we need to deal with the $A_N$'s that admit modular braidings. Recall that when $A_N$ is modular the centre is $A_N \boxtimes A_N^\text{bop}$.
\begin{lemma}
We have \[\BrAut(A_N \boxtimes A_N^\text{bop})  \begin{cases} 
     \{e\} & N  \equiv \{0,2\} \pmod 4 \\
        \mathbb{Z} / 2\mathbb{Z} & N  \equiv 1  \pmod 4 \\
        (\mathbb{Z} / 2\mathbb{Z})^2 & N   \equiv 3 \pmod 4. \\
   \end{cases}
\]
\end{lemma}  
\begin{proof}
\textbf{Case $N  \equiv \{0,2\} \pmod 4 $:}

When $N$ is even the generating objects of the $A_N$ and $A_N^\text{bop}$ factors of $A_N \boxtimes A_N^\text{bop}$ are both fixed by any braided auto-equivalence. This can be seen by a fusion rule / $t$-value argument. Thus $\BrAut(A_N \boxtimes A_N^\text{bop}) = \BrAut(A_N) \times \BrAut(A_N^\text{bop})$ which we know is trivial from Lemma~\ref{lem:brAutAN}.

\textbf{Case $N   \equiv 1 \pmod 4$:}

Again we look at the generating objects of the $A_N$ and $A_N^\text{bop}$ factors. It suffices to consider only auto-equivalences that move these objects as if they were fixed then we could apply Lemma~\ref{lem:brAutAN} to see that the auto-equivalence is trivial. By considering fusion rules and $t$-values we see that there are only two possible auto-equivalences of $A_N \boxtimes A_N^\text{bop}$. The non-trivial auto-equivalence is determined by:
\begin{equation*}
f^{(1)}\boxtimes f^{(0)} \mapsto f^{(1)}\boxtimes f^{(N-1)} \text{ and } f^{(0)}\boxtimes f^{(1)} \mapsto f^{(N-1)}\boxtimes f^{(1)}.
\end{equation*}

There is an injective map $\operatorname{Out}_\otimes(C) \hookrightarrow \BrAut(Z(C))$ that sends the outer tensor auto-equivalence $F$ to the invertible bimodule $_CC_{F(C)}$. Lemma~\ref{lem:tenAutAN} shows that there is a non-trivial outer tensor auto-equivalence of $A_N$. Thus $ \mathbb{Z} /2\mathbb{Z} \subseteq \BrAut(A_N \boxtimes A_N^\text{bop})$ and we are done.

\textbf{Case $N \equiv 3 \pmod 4 $:}

By considering fusion rules and $t$-values it can be shown that every braided auto-equivalence of $A_N \boxtimes A_N^\text{bop}$ restricts to braided auto-equivalences of the factors. Therefore $\BrAut(A_N \boxtimes A_N^\text{bop}) =  \BrAut(A_N) \times \BrAut(A_N^\text{bop}) = (\mathbb{Z} / 2\mathbb{Z})^2$.
\end{proof}

Putting everything together we can now complete the proof of our main Theorem.

\begin{proof}[Proof of Theorem \ref{thm:BrPic}]
The results of Section~\ref{sec:autos} and \ref{sec:BrPic} compute the braided auto-equivalence group of the centre of a representative for each Galois orbit of the $ADE$ fusion categories and adjoint subcategories. The isomorphism $\BrAut(Z(C)) \cong \BrPic(C)$ gives us the Brauer-Picard groups. Finally applying Lemma~\ref{lem:galoisPreservesBrPic} extends our computation to every fusion category realising the $ADE$ fusion rules and the associated adjoint subcategories.

\end{proof}

 \section{Explicit constructions of invertible bimodules from braided autoequivalences}\label{sec:D10}

To compute the extension theory of a fusion category one ideally wants to have explicit descriptions of the invertible bimodules. So far our analysis has just revealed the group structure of the Brauer-Picard groups. In most cases this is sufficient to completely understand the invertible bimodules over the category e.g. The Brauer-Picard group of $\ad(E_6)$ is $\mathbb{Z} /2 \mathbb{Z}$ which corresponds to the two graded pieces of $E_6$. Exceptions to this are the $\ad(D_{10})$ categories which have Brauer-Picard group $S_3 \times S_3$, and the $\ad(A_7)$ categories which have Brauer-Picard group $D_{2\cdot 4}$. The aim of this Section is to give explicit descriptions of all the invertible bimodules over these categories. We achieve this by chasing through the isomorphism $\BrPic(C) \cong \BrAut(Z(C))$.

An invertible bimodule over $C$ corresponds to a left $C$-module $M$ along with an equivalence $C \xrightarrow{\sim} \operatorname{Fun}(M,M)_C$ \cite{MR2677836}. This description gives rise to a natural action of $\Aut(C)$ on $\BrPic(C)$ by precomposing a tensor auto-equivalence of $C$ to get another equivalence $C\xrightarrow{\sim} C \xrightarrow{\sim} \operatorname{Fun}(M,M)_C$. If we restrict this action to outer auto-equivalences of $C$ then this action is free and faithful. Therefore up to the action of $\Out_\otimes(C)$, invertible bimodules over $C$ correspond to left $C$-module categories such that $C \cong \operatorname{Fun}(M,M)_C$. In the language of algebra objects this corresponds to finding all simple algebra objects $A \in C$ such that $A$-bimod $\cong C$. The isomorphism $\BrPic(C) \cong \BrAut(Z(C))$ allows us to compute the underlying algebra object $A$ corresponding to an invertible bimodule we get from a braided auto-equivalence of $Z(C)$.
 \begin{cons}\cite{MR2677836}
 
 Let $F \in \BrAut{Z(C)}$. We will construct an algebra $A \in C$ such that $A-\operatorname{mod}_C$ is equivalent to the image of $F$ under the isomorphism~\ref{eq:iso} (as left $C$-module categores). The object $\mathbf{1}$ is trivially an algebra object in $C$, therefore inducing $\mathbf{1}$ up to the centre of $C$ gives an algebra object of the centre, as the induction functor is lax monoidal. As $F$ is a tensor auto-equivalence $F^{-1}(I(\mathbf{1}))$ also has the structure of an algebra. Finally restricting back down to $C$ gives us an algebra object back in $C$. However $R(F^{-1}(I(\mathbf{1})))$ may not be indecomposable as an algebra object. Let $A$ be any simple algebra object in the decomposition of $R(F^{-1}(I(\mathbf{1})))$, then independent of choice of $A$ the corresponding module category is always the same. Thus we can choose any such $A$.
 \end{cons}
 
 \subsection{Invertible bimodules over $\ad(D_{10})$}

The issue with directly applying the above construction to get such an $A$ is that it can be difficult to determine how $R(F^{-1}(I(\mathbf{1})))$ decomposes into simple algebra objects. To deal with this problem we use the algorithm described in \cite[Chapter 3]{MR2909758} to obtain a finite list of possible simple algebra objects in $\ad(D_{10})$. Note that the algorithm does not guarantee all of the objects returned can be realised as algebra objects. We include the rank of the corresponding module category as this information will be a useful later.
 
 \begin{prop}\label{lem:AlgD10}
 Let $A$ be an algebra object of $\ad(D_{10})$. Then $A$ is one of: 
 
  \begin{table}[h!]

    \centering
    \resizebox{\linewidth}{!}{%
    \begin{tabular}{c|c|c}
    \toprule

			     Rank $3$ & Rank $4$ & Rank $6$  \\ 
	\midrule
		       $\mathbf{1} \oplus  f^{(6)}$ &  $\mathbf{1} \oplus f^{(2)}$ &  $\mathbf{1}$ \\
    	                  $\mathbf{1} \oplus f^{(2)} \oplus  f^{(6)} \oplus P\oplus Q$  &  $\mathbf{1} \oplus P$ & $\mathbf{1} \oplus  f^{(4)} \oplus P$ \\
	                   $\mathbf{1} \oplus 2f^{(2)}\oplus 3 f^{(4)}\oplus 4 f^{(6)}\oplus 2P\oplus 2Q$& $\mathbf{1} \oplus Q$ & $\mathbf{1} \oplus  f^{(4)} \oplus Q$ \\
	                   & $\mathbf{1} \oplus f^{(2)} \oplus  f^{(4)} \oplus  f^{(6)}$ & $\mathbf{1} \oplus f^{(2)} \oplus  f^{(4)}$ \\
	                   &  $\mathbf{1} \oplus  f^{(4)} \oplus  f^{(6)} \oplus P$ &$\mathbf{1} \oplus f^{(2)} \oplus  f^{(4)}\oplus  f^{(6)}\oplus P\oplus Q$ \\
	                   & $\mathbf{1} \oplus  f^{(4)} \oplus  f^{(6)} \oplus Q$ & $\mathbf{1} \oplus f^{(2)} \oplus 2 f^{(4)}\oplus 2 f^{(6)}\oplus P\oplus Q$ \\
	                   &  $\mathbf{1} \oplus f^{(2)} \oplus  f^{(4)}\oplus 2 f^{(6)}\oplus P\oplus Q$& \\
	                   & $\mathbf{1} \oplus 2f^{(2)} \oplus 2 f^{(4)}\oplus 2 f^{(6)}\oplus P\oplus Q$ & \\
	                   & $\mathbf{1} \oplus f^{(2)} \oplus 2 f^{(4)}\oplus 2 f^{(6)}\oplus 2P\oplus Q$ & \\
			&  $\mathbf{1} \oplus f^{(2)} \oplus 2 f^{(4)}\oplus 2 f^{(6)}\oplus P\oplus 2Q$& \\
    	\bottomrule
    \end{tabular}}
\end{table}

 \end{prop}
\begin{proof}
This list was computed using the algorithm from \cite[Chapter 3]{MR2909758}. Our implementation was tested against the results of \cite{MR2909758}.
\end{proof}

With this list of possible simple algebra objects we wish to determine how $R(F^{-1}(I(\mathbf{1})))$ decomposes into simple algebra objects.

 Recall from Lemmas~\ref{lem:aut_D10} and \ref{lem:tensor_decomp} that the group of braided auto-equivalences of $Z(\ad(D_{10})) = \ad(D_{10}) \boxtimes \ad(D_{10})^\text{bop}$ is $S_3 \times S_3$ where each $S_3$ factor independently permutes the objects $f^{(2)},P$ and $Q$ in $\ad(D_{10})$ and $\ad(D_{10})^\text{bop}$ respectively.
  
In general describing the induction and restriction functors between the categories $C$ and $Z(C)$ is a difficult problem. However as $\ad(D_{10})$ is modular these functors behave quite nicely. The induction of the tensor identity $\mathbf{1}$ up to the centre $\ad(D_{10}) \boxtimes \ad(D_{10})^\text{bop}$ gives the object $\oplus_{X \in \operatorname{Irr}(\ad(D_{10}))} X\boxtimes X^\text{bop}$ where we write $X^\text{bop}$ to specify the object $X$ in the opposite category. The restriction of an object $X \boxtimes Y^\text{bop}$ in $\ad(D_{10}) \boxtimes \ad(D_{10})^\text{bop}$ back down to $\ad(D_{10})$ is given by $X\otimes Y^*$. We compute the following table of $R(F^{-1}(I(\mathbf{1})))$ for each $F\in \BrAut(Z(D_{10}))$.
 
 \begin{table}[h!]\label{tab:big_alg_objects}

    \centering
    \resizebox{\linewidth}{!}{%
    \begin{tabular}{cc|cccccc}
    \toprule
                              &&&&&$\BrAut(\ad(D_{10}))$&&\\
    	
			     &$\times$&          $id$         &   $P\leftrightarrow Q$& $f^{(2)}\leftrightarrow P$ &    $f^{(2)}\leftrightarrow Q$&  $f^{(2)}\rightarrow P \rightarrow Q$   & $f^{(2)}\rightarrow Q \rightarrow P$  \\ 
	\midrule
			   &$id$    & $I$       &  $B_{f^{(2)}}$  &   $B_Q $ & $B_P$ &    $A$    &  $A$      \\ 
            	
    	                     &$P\leftrightarrow Q$ & $B_{f^{(2)}}$ &     $ I$    &  $  A$      &  $A$    & $B_P$  &  $B_Q$     \\ 
	
   	                 &$f^{(2)}\leftrightarrow P$ & $B_Q$  & $A $      &  $I$         &  $A$     &  $B_{f^{(2)}}$ &   $B_P$   \\  
	                     
	                   \raisebox{\dimexpr 2ex}[0pt][0pt]{\rotatebox[origin=c]{90}{ $\BrAut(\ad(D_{10}^\text{bop}))$}}    &$f^{(2)}\leftrightarrow Q$& $B_P$  &  $A$      &   $A$       &   $I$     & $B_Q$  &  $B_{f^{(2)}}$      \\ 
            	
    	                     & $f^{(2)}\rightarrow P \rightarrow Q$& $A$      & $ B_P$  &  $  B_{f^{(2)}}$ & $B_Q$ &  $I$       &  $A$     \\ 
	
   	                   & $f^{(2)}\rightarrow Q \rightarrow P$ & $A$      & $B_Q $  &  $B_P$    & $B_{f^{(2)}}$ &    $A$   &   $I$  \\
    	                 
    	\bottomrule
    \end{tabular}}
\end{table}
 \newpage
Where 
\begin{align*}
I & := 6\mathbf{1} \oplus 3f^{(2)} \oplus 6 f^{(4)} \oplus 3 f^{(6)} \oplus 3P \oplus 3Q \\
A &:= 3\mathbf{1} \oplus 3f^{(2)} \oplus 3 f^{(4)} \oplus 6 f^{(6)} \oplus 3P \oplus 3Q \\
B_{f^{(2)}} &:= 4\mathbf{1} \oplus 5f^{(2)} \oplus 4 f^{(4)} \oplus 5 f^{(6)} \oplus 2P \oplus 2Q \\
B_P &:= 4\mathbf{1} \oplus 2f^{(2)}\oplus 4 f^{(4)} \oplus 5 f^{(6)} \oplus 5P \oplus 2Q \\
B_Q &:= 4\mathbf{1} \oplus 2f^{(2)} \oplus 4 f^{(4)} \oplus 5 f^{(6)} \oplus 2P \oplus 5Q.
\end{align*}

\begin{lemma}\label{lem:decomp1}
There exist unique decompositions of $I$, $B_{f^{(2)}}$,$B_P$, and $B_Q$ into simple algebra objects. These decompositions are
\begin{align*}
I =& (\mathbf{1}) \oplus (\mathbf{1} \oplus  f^{(4)} \oplus P) \oplus (\mathbf{1} \oplus  f^{(4)} \oplus Q) \oplus (\mathbf{1} \oplus f^{(2)} \oplus  f^{(4)}) \oplus (\mathbf{1} \oplus f^{(2)} \oplus  f^{(4)}\oplus  f^{(6)}\oplus P\oplus Q) \\ 
     & \oplus (\mathbf{1} \oplus f^{(2)} \oplus 2 f^{(4)}\oplus 2 f^{(6)}\oplus P\oplus Q), \\
B_{f^{(2)}} =& (\mathbf{1} \oplus f^{(2)}) \oplus (\mathbf{1} \oplus f^{(2)} \oplus  f^{(4)} \oplus  f^{(6)}) \oplus (\mathbf{1} \oplus f^{(2)} \oplus  f^{(4)}\oplus 2 f^{(6)}\oplus P\oplus Q) \\
   & \oplus  (\mathbf{1} \oplus 2 f^{(2)} \oplus 2 f^{(4)}\oplus 2 f^{(6)}\oplus P\oplus Q), \\
B_P = & (\mathbf{1} \oplus P) \oplus (\mathbf{1} \oplus  f^{(4)} \oplus  f^{(6)} \oplus P) \oplus (\mathbf{1} \oplus f^{(2)} \oplus  f^{(4)}\oplus 2 f^{(6)}\oplus P\oplus Q)  \\
     &\oplus  (\mathbf{1} \oplus f^{(2)} \oplus 2 f^{(4)}\oplus 2 f^{(6)}\oplus 2P\oplus Q), \\
B_Q =& (\mathbf{1} \oplus Q) \oplus (\mathbf{1} \oplus  f^{(4)} \oplus  f^{(6)}\oplus Q) \oplus (\mathbf{1} \oplus f^{(2)} \oplus  f^{(4)}\oplus 2 f^{(6)}\oplus P\oplus Q)  \\
   &\oplus  (\mathbf{1} \oplus f^{(2)} \oplus 2 f^{(4)}\oplus 2 f^{(6)}\oplus P\oplus 2Q).
\end{align*}
\end{lemma}
\begin{proof}
We brute force check all possible combinations of simple algebra objects in Lemma~\ref{lem:AlgD10} and see that only the above decompositions are possible. To simplify our computations recall that all simple algebra objects in the decomposition of $R(F^{-1}(I(\mathbf{1})))$ will have equivalent module categories. As module categories with different ranks are clearly non-equivalent we can restrict our attention to combinations of algebra objects whose corresponding module categories have the same rank.
\end{proof}
\begin{lemma}\label{lem:uniqueDecomp}
The objects $\mathbf{1} \oplus f^{(2)}$, $\mathbf{1} \oplus P$, and $\mathbf{1} \oplus Q$ have unique algebra object structures. 
\end{lemma}
\begin{proof}
Existence of an algebra object structure follows from Lemma~\ref{lem:decomp1}. As $\Hom(f^{(2)}\otimes f^{(2)},f^{(2)})$, $\Hom(P\otimes P,P)$, and $\Hom(Q\otimes Q,Q)$ are all $1$-dimensional we can apply \cite[Lemma 8]{MR1976459} to get uniqueness.
\end{proof}

\begin{lemma}\label{lem:decomp2}
The algebra object $A$ decomposes in to simple algebra objects as $(\mathbf{1} \oplus  f^{(6)}) \oplus (\mathbf{1} \oplus f^{(2)} \oplus  f^{(6)}\oplus P\oplus Q) \oplus (\mathbf{1} \oplus 2f^{(2)} \oplus 3 f^{(4)}\oplus 4 f^{(6)}\oplus 2P\oplus 2Q)$.
\end{lemma}
\begin{proof}
Using the same technique as in Lemma~\ref{lem:decomp1} we can show that $A$ either decomposes as 

\[(\mathbf{1} \oplus  f^{(6)}) \oplus (\mathbf{1} \oplus f^{(2)} \oplus  f^{(6)}\oplus P\oplus Q) \oplus (\mathbf{1} \oplus 2f^{(2)} \oplus 3 f^{(4)}\oplus 4 f^{(6)}\oplus 2P\oplus 2Q),\]
 or 
\[3\times (\mathbf{1}\oplus f^{(2)}\oplus  f^{(4)}\oplus 2 f^{(6)} \oplus P \oplus Q).\] We will rule out the possibility of the latter decomposition.

Suppose such a decomposition did exist, then this would imply that the corresponding bimodule category would have underlying algebra object $\mathbf{1}\oplus f^{(2)}\oplus  f^{(4)}\oplus 2 f^{(6)} \oplus P \oplus Q$, and hence is equivalent to $\mathbf{1}\oplus f^{(2)}\oplus  f^{(4)}\oplus 2 f^{(6)} \oplus P \oplus Q$-mod as a left module category. The category $\mathbf{1}\oplus f^{(2)}\oplus  f^{(4)}\oplus 2 f^{(6)} \oplus P \oplus Q$-mod has module fusion graph:

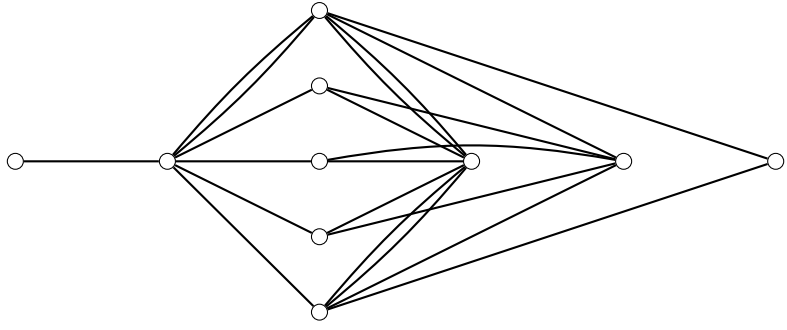
\begin{center}\begin{tikzpicture}[baseline={([yshift=-.5ex]current bounding box.center)}]
        \vertex[label=$$](1) at (0,0) {};
        \vertex[label=$$](2) at (2,0) {};
        \vertex[label=$$](3) at (4,0) {};
\vertex[label=$$](4) at (4,1) {};
\vertex[label=$$](5) at (4,2) {};
\vertex[label=$$](6) at (4,-1) {};
\vertex[label=$$](7) at (4,-2) {};
         \vertex[label=$$](8) at (6,0) {};
         \vertex[label=$$](9) at (8,0) {}; 
         \vertex[label=$$](10) at (10,0) {};
        \Edge(1)(2)
 \Edge[style = {bend left = 5}](2)(5)
 \Edge[style = {bend right = 5}](2)(5)
 \Edge(2)(4)
\Edge(2)(3)
\Edge(2)(6)
\Edge(2)(7)
 \Edge[style = {bend right = 5}](8)(5)
 \Edge[style = {bend left = 5}](8)(5)
 \Edge(8)(4)
\Edge(8)(3)
\Edge(8)(6)
 \Edge[style = {bend right = 5}](8)(7)
 \Edge[style = {bend left = 5}](8)(7)
\Edge(4)(9)
\Edge[style = {bend left=10}](3)(9)
\Edge(6)(9)
\Edge(5)(10)
\Edge(7)(10)
\Edge(5)(9)
\Edge(7)(9)
       
    \end{tikzpicture}\end{center}

The dimension of object at the far right we compute to be $\sqrt{[3]+1}$. Thus the internal hom of this object is an algebra object in $\ad(D_{10})$ with dimension $[3]+1$, and so must be one of $\mathbf{1} \oplus f^{(2)}$, $\mathbf{1} \oplus P$, or $\mathbf{1} \oplus Q$. This implies $\mathbf{1}\oplus f^{(2)}\oplus  f^{(4)}\oplus 2 f^{(6)} \oplus P \oplus Q$-mod is equivalent to one of $\mathbf{1} \oplus f^{(2)}$-mod, $\mathbf{1} \oplus P$-mod, or $\mathbf{1} \oplus Q$-mod. 

Assume without loss of generality that $\mathbf{1}\oplus f^{(2)}\oplus  f^{(4)}\oplus 2 f^{(6)} \oplus P \oplus Q $-mod is equivalent to $\mathbf{1} \oplus f^{(2)}$-mod, as the following argument works with $f^{(2)}$ replaced with $P$ or $Q$. Lemma~\ref{lem:uniqueDecomp} shows us that there exists a unique algebra object structure on $\mathbf{1} \oplus f^{(2)}$, and thus up to action of tensor auto-equivalences there is a unique invertible bimodule with underlying algebra object $\mathbf{1} \oplus f^{(2)}$. However Lemma~\ref{lem:decomp1} tells us that the underlying algebra object of any of the invertible bimodules corresponding to $B_{f^{(2)}}$ must be $\mathbf{1} \oplus f^{(2)}$. As $B_{f^{(2)}}$ appears $6$ times in Table~\ref{tab:big_alg_objects} it follows that the algebra object structure on $\mathbf{1} \oplus f^{(2)}$ coming from $A$ is different than the one coming from $B_{f^{(2)}}$. Thus there are two algebra object structures on $\mathbf{1} \oplus f^{(2)}$, but this is a contradiction to Lemma~\ref{lem:uniqueDecomp}. Therefore $A$ cannot decompose as $3\times (\mathbf{1}\oplus f^{(2)}\oplus  f^{(4)}\oplus 2 f^{(6)} \oplus P \oplus Q)$.
 \end{proof}
 
 \begin{cor}\label{lem:Zalg}
 There exists two algebra object structures on $\mathbf{1} \oplus  f^{(6)}$ such that $(\mathbf{1} \oplus  f^{(6)})-\operatorname{bimod}  \simeq \ad(D_{10})$.
 \end{cor}
 \begin{proof}
The above table and Lemma~\ref{lem:decomp2} shows that there are exactly $12$ invertible bimodule categories over $\ad(D_{10})$ whose underlying algebra object is $\mathbf{1} \oplus f^{(6)}$. As $\Out_\otimes(\ad(D_{10})) = S_3$ has order 6 (Lemma~\ref{lem:aut_D10}), up to the action of outer tensor auto-equivalences there are two different left $\ad(D_{10})$ modules with dual equivalent to $\ad(D_{10})$, and hence two algebra object structures on $\mathbf{1} \oplus  f^{(6)}$ such that $(\mathbf{1} \oplus  f^{(6)})$-bimod $ \simeq \ad(D_{10})$.
 \end{proof}

The algebra objects $\mathbf{1} \oplus  f^{(6)}$ give rise to subfactors when $\ad(D_{10})$ is unitary. These are the GHJ subfactors corresponding to the odd part of $E_7$ as a module over $D_{10}$.
\begin{cor}
There exist two inequivalent subfactors of index $1+ \cos(\frac{\pi}{9})\csc(\frac{\pi}{18}) \approxeq 6.411$ whose even and dual even parts are $\ad(D_{10})$, and with principle graph
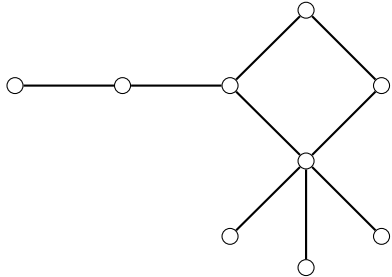
\begin{center}
\begin{tikzpicture}[scale = .707,baseline={([yshift=-.5ex]current bounding box.center)}]
        \vertex (1) at (0,0) {};
        \vertex (2) at (2,0) {};
\vertex (3) at (4,0) {};
\vertex (4) at (5.414,1.414) {};
\vertex (5) at (5.414,-1.414) {};
\vertex (6) at (6.818,0) {};
\vertex (7) at (5.414,-3.414) {};
\vertex (8) at (4,-2.828) {};
\vertex (9) at (6.818,-2.828) {};

       \Edge(1)(2)
\Edge(2)(3)
\Edge(3)(4)
\Edge(3)(5)
\Edge(4)(6)
\Edge(6)(5)
\Edge(5)(7)
\Edge(5)(8)
\Edge(5)(9)

    \end{tikzpicture}\end{center}
\end{cor}
Summarising the above discussion, we have:
 \begin{theorem}
 Up to action of outer tensor auto-equivalences there are six invertible bimodules over $\ad(D_{10})$. These invertible bimodules come from the algebra objects $\mathbf{1}$, $\mathbf{1} \oplus f^{(2)}$, $\mathbf{1} \oplus P$, and $\mathbf{1} \oplus Q$ each of which have a unique algebra object structure, and $\mathbf{1} \oplus f^{(6)}$ which has two algebra object structures.
 \end{theorem}

 \subsection{Bimodules over $\ad(A_7)$}
 \setcounter{MaxMatrixCols}{20}

 Recall that the group of braided auto-equivalences of $Z(\ad(A_7))$ is $D_{2\cdot 4}$ with generators $r$ and $s$ as described in Lemma~\ref{lem:centre_AN_autos}. The same calculations as for $D_{10}$ reveal the group structure on the invertible bimodules of $\ad(A_7)$. The induction matrix \[
\begin{bmatrix}
    1&0&0&0&1&0&0&0&1&0&0&1&0&0\\
    0&1&0&0&1&1&0&1&1&1&1&0&1&0\\
    0&0&1&0&0&1&1&0&1&1&1&1&0&1\\
    0&0&0&1&0&0&1&0&0&1&0&0&1&0
\end{bmatrix}
\]
for this computation was computed from the induction matrix for the modular category $A_7$. The ordering of the rows is the standard ordering of the simples objects of $\ad(A_7)$. The ordering of the columns is as follows:

\begin{align*} \mathbf{1}\boxtimes \mathbf{1}, f^{(2)}\boxtimes \mathbf{1},f^{(4)}\boxtimes \mathbf{1},f^{(6)}\boxtimes \mathbf{1},&f^{(1)}\boxtimes f^{(1)},f^{(3)}\boxtimes f^{(1)},f^{(5)}\boxtimes f^{(1)}, \mathbf{1}\boxtimes f^{(2)},f^{(2)}\boxtimes f^{(2)},f^{(4)}\boxtimes f^{(2)},f^{(1)}\boxtimes f^{(3)},\\ &\frac{f^{(3)}\boxtimes f^{(3)}+S}{2},\frac{f^{(3)}\boxtimes f^{(3)}-S}{2},\mathbf{1} \boxtimes f^{(4)}. 
\end{align*}
With this data we compute the group structure as follows:
 
  \begin{table}[h!]

    \centering

    \begin{tabular}{c|cccccccc}
    \toprule
                   $\BrAut(Z(\ad(A_7)))$           & $e$ & $r$ & $r^2$ & $r^3$ & $s$ & $rs$ & $r^2s$ & $r^3s$ \\
	\midrule
		$A$	   &$\mathbf{1}$    & $\mathbf{1} \oplus f^{(2)}$      &  $\mathbf{1} \oplus f^{(6)}$   &   $\mathbf{1} \oplus f^{(4)}$   & $\mathbf{1} \oplus f^{(2)}$   &    $\mathbf{1}$    & $\mathbf{1} \oplus f^{(4)}$ &   $\mathbf{1} \oplus f^{(6)}$    \\                	                 
    	\bottomrule
    \end{tabular}
\end{table}

\appendix

\section{Planar algebra natural isomorphisms}
In \cite{1607.06041} the authors show an equivalence between the category of planar algebras and the category of based pivotal categories (pivotal categories with a chosen self-dual generating object). However their definition of the latter category assumes that the natural transformations between pivotal functors must have trivial component on the chosen generating object. This assumption turns out to be too restrictive as there exist planar algebras whose automorphism group is non-trivial, yet whose associated based pivotal category has trivial based auto-equivalence group, e.g. $\ad(A_{2N})$. In this appendix we define natural isomorphisms between planar algebra homomorphisms. We then show that the group of planar algebra automorphisms, and the group of based auto-equivalences of the associated pivotal category are isomorphic (when both groups are considered up to natural isomorphism).

Let's begin by looking at natural transformations in the category of based pivotal categories. Let $F,G: (C,X) \to (D,Y)$ be two based pivotal functors, and $\eta: F \to G$ a natural transformation between them. In \cite{1607.06041} it is shown that $\eta$ is entirely determined by its component on $Y$, that is a morphism in $\Hom_D(Y,Y)$. Consider the duality map $\epsilon_Y: Y \otimes Y \to \mathbf{1}$. As $F$ and $G$ are pivotal functors they fix this duality map, and therefore the following diagram commutes:
$$\begin{CD}
Y\otimes Y @>\epsilon_Y >> \mathbf{1} \\
@VV{\eta_Y \otimes \eta_Y}V @VV{\text{id}_1}V\\
Y\otimes Y @>\epsilon_Y>>  \mathbf{1}.
\end{CD}$$
Therefore $\eta_Y\circ (\eta_Y)^* = \text{id}_Y$. In particular this implies that every natural transformation between based pivotal functors is in fact an isomorphism.

Recall that one gets a planar algebra from a pointed pivotal category $(D,Y)$ by taking $P_{N} := \Hom_D(\mathbf{1} \to Y^{\otimes N})$. We wish to carry the notion of natural isomorphism across this construction. As a natural isomorphism between based pivotal functors is determined by its component on $Y$, i.e an element of $\Hom_D( Y,Y)$, it follows that a natural isomorphism of planar algebra homomorphisms should be determined by an element of $P_2$. By pushing the naturality condition of a natural transformation of functors through the planar algebra construction we arrive at the following definition.

\begin{dfn}
Let $P$ and $Q$ be planar algebras, and $\phi,\psi : Q \to P$ planar algebra homomorphisms. A natural isomorphism $\phi \to \psi$ is an element   \begin{tikzpicture}[baseline={([yshift=-.5ex]current bounding box.center)}]
\node [thick, rectangle, draw] at (0,0) (Z1) {$\eta$};
\draw [thick] (0,-0.25) --  (0,-.5);
\draw [thick] (0,0.25) -- (0,.5);
\end{tikzpicture} $ \in Q_{2}$ satisfying for any element   \begin{tikzpicture}[baseline={([yshift=-.5ex]current bounding box.center)}]
\node [thick, rectangle, draw,minimum size = 29] at (0,0) (Z1) {$f$};
\draw [thick] (0,0.5) -- (0,.75);
\draw [thick] (0.3,0.5) -- (0.3,.75);
\draw [thick] (-0.3,0.5) -- (-.3,.75);
\node at (0,1) {$n$};
\end{tikzpicture} $ \in P_{n}$ the naturality condition
\[
\begin{tikzpicture}[baseline={([yshift=-.5ex]current bounding box.center)}]
\node [thick, rectangle, draw,minimum size = 45] at (0,0) (Z1) {$\phi(f)$};
\draw [thick] (0,0.78) -- (0,1);
\draw [thick] (.6,0.78) -- (0.6,1);
\draw [thick] (-.6,0.78) -- (-0.6,1);

\node [thick, rectangle, draw] at (0,1.25) (Z1) {$\eta$};
\node [thick, rectangle, draw] at (0.6,1.25) (Z1) {$\eta$};
\node [thick, rectangle, draw] at (-0.6,1.25) (Z1) {$\eta$};

\draw [thick] (0,1.75) -- (0,1.5);
\draw [thick] (.6,1.75) -- (0.6,1.5);
\draw [thick] (-.6,1.75) -- (-0.6,1.5);
\end{tikzpicture} \quad = \quad \begin{tikzpicture}[baseline={([yshift=-.5ex]current bounding box.center)}]
\node [thick, rectangle, draw,minimum size = 45] at (0,0) (Z1) {$\psi(f)$};
\draw [thick] (0,0.78) -- (0,1.25);
\draw [thick] (.6,0.78) -- (0.6,1.25);
\draw [thick] (-.6,0.78) -- (-0.6,1.25);



\end{tikzpicture}\]
\end{dfn}

If the planar algebra is irreducible, i.e $P_2 \cong \mathbb{C}$, this implies that the only possible natural isomorphisms are $\pm \text{id}_1$. For the general case the natural isomorphisms live in $\{ u \in P_2 : uu^* = \text{id}_1 \}$, the orthogonal group of $P_2$.

If we consider the group of planar algebra isomorphisms up to natural isomorphism we get the following lemma.

\begin{lemma}
Up to natural isomorphism, the groups $\Aut(P)$ and $\Aut^\text{piv}(C_P;X)$ are isomorphic.
\end{lemma}
\begin{proof}
To complete this claim we need to show that everything in the kernel of the map in Proposition~\ref{prop:PA_to_FC} is naturally isomorphic to the identity planar algebra map. Let $F$ be a planar algebra automorphism of $P$ that maps to the identity functor of $C_P$. That is there exists a natural isomorphism $\mu : F \simeq \text{Id}_{C_P}$. Let $\eta$ be the component of $\mu$ on the distinguished generating object of $C_P$ i.e the single strand. It follows from the earlier discussion that $\eta$ is a planar algebra natural isomorphism from $F$ to the identity morphism on $P$.
\end{proof}

Planar algebra natural isomorphisms should also make sense for oriented planar algebras. We leave it up to a motivated reader to work out the details.

\bibliography{bibliography} 

\newcommand{\noopsort}[1]{}\def\cprime{$'$} \def\cprime{$'$} \def\cprime{$'$}
\begin{thebibliography}{10}

\bibitem{MR2577673}
Stephen Bigelow.
\newblock Skein theory for the {$ADE$} planar algebras.
\newblock {\em J. Pure Appl. Algebra}, 214(5):658--666, 2010.
\newblock \arxiv{math.QA/0903.0144} \mathscinet{MR2577673}
  \doi{10.1016/j.jpaa.2009.07.010}.

\bibitem{MR1193933}
Jocelyne Bion-Nadal.
\newblock An example of a subfactor of the hyperfinite {${\rm II}\sb 1$} factor
  whose principal graph invariant is the {C}oxeter graph {$E\sb 6$}.
\newblock In {\em Current topics in operator algebras ({N}ara, 1990)}, pages
  104--113. World Sci. Publ., River Edge, NJ, 1991.
\newblock \mathscinet{MR1193933}.

\bibitem{MR1815993}
Jens B{\"o}ckenhauer, David~E. Evans, and Yasuyuki Kawahigashi.
\newblock Longo-{R}ehren subfactors arising from {$\alpha$}-induction.
\newblock {\em Publ. Res. Inst. Math. Sci.}, 37(1):1--35, 2001.
\newblock \mathscinet{MR1815993} \arxiv{math/0002154v1}.

\bibitem{MR3373393}
Costel-Gabriel Bontea and Dmitri Nikshych.
\newblock On the {B}rauer-{P}icard group of a finite symmetric tensor category.
\newblock {\em J. Algebra}, 440:187--218, 2015.
\newblock \arxiv{1408.6445} \mathscinet{MR3373393}
  \doi{10.1016/j.jalgebra.2015.06.006}.

\bibitem{1703.06543}
Alexei Davydov, Pavel Etingof, and Dmitri Nikshych.
\newblock Autoequivalences of tensor categories attached to quantum groups at
  roots of 1.
\newblock 2017.
\newblock \arxiv{1703.06543}.

\bibitem{MR2609644}
Vladimir Drinfeld, Shlomo Gelaki, Dmitri Nikshych, and Victor Ostrik.
\newblock On braided fusion categories. {I}.
\newblock {\em Selecta Math. (N.S.)}, 16(1):1--119, 2010.
\newblock \arxiv{0906.0620} \mathscinet{MR2609644}
  \doi{10.1007/s00029-010-0017-z}.

\bibitem{AN-Survey}
Cain Edie-Michell and Scott Morrison.
\newblock A survey of categories with $a_n$ fusion rules, 2017.
\newblock in preparation.

\bibitem{MR2677836}
Pavel Etingof, Dmitri Nikshych, and Victor Ostrik.
\newblock Fusion categories and homotopy theory.
\newblock {\em Quantum Topol.}, 1(3):209--273, 2010.
\newblock (with an appendix by Ehud Meir), \arxiv{0909.3140}
  \mathscinet{MR2677836} \doi{10.4171/QT/6}.

\bibitem{MR2183279}
Pavel Etingof, Dmitri Nikshych, and Viktor Ostrik.
\newblock On fusion categories.
\newblock {\em Ann. of Math. (2)}, 162(2):581--642, 2005.
\newblock \arxiv{math.QA/0203060} \mathscinet{MR2183279}
  \doi{10.4007/annals.2005.162.581}.

\bibitem{MR1239440}
J{\"u}rg Fr{\"o}hlich and Thomas Kerler.
\newblock {\em Quantum groups, quantum categories and quantum field theory},
  volume 1542 of {\em Lecture Notes in Mathematics}.
\newblock Springer-Verlag, Berlin, 1993.

\bibitem{MR2587410}
Shlomo Gelaki, Deepak Naidu, and Dmitri Nikshych.
\newblock Centers of graded fusion categories.
\newblock {\em Algebra Number Theory}, 3(8):959--990, 2009.
\newblock \arxiv{0905.3117} \mathscinet{MR2587410}
  \doi{10.2140/ant.2009.3.959}.

\bibitem{MR2909758}
Pinhas Grossman and Noah Snyder.
\newblock {Quantum subgroups of the Haagerup fusion categories}.
\newblock {\em Comm. Math. Phys.}, 311(3):617--643, 2012.
\newblock \arxiv{1102.2631} \mathscinet{MR2909758}
  \doi{10.1007/s00220-012-1427-x}.

\bibitem{MR3449240}
Pinhas Grossman and Noah Snyder.
\newblock The {B}rauer-{P}icard group of the {A}saeda-{H}aagerup fusion
  categories.
\newblock {\em Trans. Amer. Math. Soc.}, 368(4):2289--2331, 2016.
\newblock \arxiv{1202.4396} \doi{10.1090/tran/6364} \mathscinet{MR3449240}.

\bibitem{1607.06041}
{Andr\'e} Henriques, David Penneys, and James Tener.
\newblock Planar algebras in braided tensor categories, 2016.
\newblock \arxiv{1607.06041}.

\bibitem{MR2468378}
Seung-Moon Hong, Eric Rowell, and Zhenghan Wang.
\newblock On exotic modular tensor categories.
\newblock {\em Commun. Contemp. Math.}, 10(suppl. 1):1049--1074, 2008.
\newblock \arxiv{0710.5761} \mathscinet{MR2468378}
  \doi{10.1142/S0219199708003162}.

\bibitem{1605.08398}
Ammar Husain.
\newblock ${G}$-extensions of quantum group categories and functorial {SPT},
  2016.
\newblock \arxiv{1605.08398}.

\bibitem{MR1145672}
Masaki Izumi.
\newblock Application of fusion rules to classification of subfactors.
\newblock {\em Publ. Res. Inst. Math. Sci.}, 27(6):953--994, 1991.
\newblock \mathscinet{MR1145672} \doi{10.2977/prims/1195169007}.

\bibitem{MR1313457}
Masaki Izumi.
\newblock On flatness of the {C}oxeter graph {$E\sb 8$}.
\newblock {\em Pacific J. Math.}, 166(2):305--327, 1994.
\newblock \mathscinet{MR1313457} \euclid{euclid.pjm/1102621140}.

\bibitem{math.QA/9909027}
Vaughan F.~R. Jones.
\newblock {Planar algebras, I}, 1999.
\newblock \arxiv{math.QA/9909027}.

\bibitem{MR1929335}
Vaughan F.~R. Jones.
\newblock The annular structure of subfactors.
\newblock In {\em Essays on geometry and related topics, {V}ol. 1, 2},
  volume~38 of {\em Monogr. Enseign. Math.}, pages 401--463. Enseignement
  Math., Geneva, 2001.
\newblock \mathscinet{MR1929335}.

\bibitem{MR1250465}
Andr{\'e} Joyal and Ross Street.
\newblock Braided tensor categories.
\newblock {\em Adv. Math.}, 102(1):20--78, 1993.
\newblock \doi{10.1006/aima.1993.1055} \mathscinet{MR1250465}.

\bibitem{MR1308617}
Yasuyuki Kawahigashi.
\newblock On flatness of {O}cneanu's connections on the {D}ynkin diagrams and
  classification of subfactors.
\newblock {\em J. Funct. Anal.}, 127(1):63--107, 1995.
\newblock \mathscinet{MR1308617} \doi{10.1006/jfan.1995.1003}.

\bibitem{MR1815260}
Zeph~A. Landau.
\newblock Fuss-{C}atalan algebras and chains of intermediate subfactors.
\newblock {\em Pacific J. Math.}, 197(2):325--367, 2001.
\newblock \mathscinet{MR1815260} \doi{10.2140/pjm.2001.197.325}.

\bibitem{1308.5656}
Zhengwei Liu.
\newblock Exchange relation planar algebras of small rank, 2013.
\newblock \arxiv{1308.5656}.

\bibitem{1603.04318}
Ian Marshall and Dmitri Nikshych.
\newblock On the {B}rauer-picard groups of fusion categories, 2016.
\newblock \arxiv{1603.04318}.

\bibitem{MR2559686}
Scott Morrison, Emily Peters, and Noah Snyder.
\newblock Skein theory for the {$\mathcal{D}_{2n}$} planar algebras.
\newblock {\em J. Pure Appl. Algebra}, 214(2):117--139, 2010.
\newblock \arxiv{math/0808.0764} \mathscinet{MR2559686}
  \doi{10.1016/j.jpaa.2009.04.010}.

\bibitem{MR2783128}
Scott Morrison, Emily Peters, and Noah Snyder.
\newblock Knot polynomial identities and quantum group coincidences.
\newblock {\em Quantum Topol.}, 2(2):101--156, 2011.
\newblock \arxiv{1003.0022} \mathscinet{MR2783128} \doi{10.4171/QT/16}.

\bibitem{1501.06869}
Scott Morrison, Emily Peters, and Noah Snyder.
\newblock Categories generated by a trivalent vertex, 2015.
\newblock \arxiv{1501.06869}.

\bibitem{MR1966525}
Michael M{\"u}ger.
\newblock From subfactors to categories and topology. {II}. {T}he quantum
  double of tensor categories and subfactors.
\newblock {\em J. Pure Appl. Algebra}, 180(1-2):159--219, 2003.
\newblock \mathscinet{MR1966525} \doi{10.1016/S0022-4049(02)00248-7}
  \arxiv{math.CT/0111205}.

\bibitem{MR1990929}
Michael M{\"u}ger.
\newblock On the structure of modular categories.
\newblock {\em Proc. London Math. Soc. (3)}, 87(2):291--308, 2003.
\newblock \mathscinet{MR1990929} \doi{10.1112/S0024611503014187}.

\bibitem{MR3210925}
Dmitri Nikshych and Brianna Riepel.
\newblock Categorical {L}agrangian {G}rassmannians and {B}rauer-{P}icard groups
  of pointed fusion categories.
\newblock {\em J. Algebra}, 411:191--214, 2014.
\newblock \arxiv{1309.5026} \mathscinet{MR3210925}
  \doi{10.1016/j.jalgebra.2014.04.013}.

\bibitem{MR1976459}
Victor Ostrik.
\newblock Module categories, weak {H}opf algebras and modular invariants.
\newblock {\em Transform. Groups}, 8(2):177--206, 2003.
\newblock \mathscinet{MR1976459} \arxiv{math/0111139}.

\bibitem{MR1839381}
V.~B. Petkova and J.-B. Zuber.
\newblock The many faces of {O}cneanu cells.
\newblock {\em Nuclear Phys. B}, 603(3):449--496, 2001.
\newblock \arxiv{hep-th/0101151} \mathscinet{MR1839381}
  \doi{10.1016/S0550-3213(01)00096-7}.

\bibitem{MR2640343}
Zhenghan Wang.
\newblock {\em Topological quantum computation}, volume 112 of {\em CBMS
  Regional Conference Series in Mathematics}.
\newblock Published for the Conference Board of the Mathematical Sciences,
  Washington, DC; by the American Mathematical Society, Providence, RI, 2010.
\newblock \doi{10.1090/cbms/112} \mathscinet{MR2640343}.

\bibitem{MR873400}
Hans Wenzl.
\newblock On sequences of projections.
\newblock {\em C. R. Math. Rep. Acad. Sci. Canada}, 9(1):5--9, 1987.
\newblock \mathscinet{MR873400}.

\bibitem{MR1617550}
Feng Xu.
\newblock New braided endomorphisms from conformal inclusions.
\newblock {\em Comm. Math. Phys.}, 192(2):349--403, 1998.
\newblock \mathscinet{MR1617550} \doi{10.1007/s002200050302}.

\bibitem{MR2670925}
Feng Xu.
\newblock On intermediate subfactors of {G}oodman-de la {H}arpe-{J}ones
  subfactors.
\newblock {\em Comm. Math. Phys.}, 298(3):707--739, 2010.
\newblock \arxiv{1002.2744} \mathscinet{MR2670925}
  \doi{10.1007/s00220-010-1001-3}.

\end{thebibliography}
\bibliographystyle{plain}
\end{document}